\documentclass[10pt]{article}

\usepackage{xcolor} 
\newcommand{\Cr}[1]{\textbf{\textcolor{red}{{#1}}}}

\usepackage{enumerate}

\usepackage{hyperref}

\usepackage{amsmath, amsrefs}
\usepackage{amsthm}
\usepackage{amssymb}
\usepackage{amsfonts}
\usepackage{bbm}

\usepackage[title, toc]{appendix}

\usepackage{tikz}
\usetikzlibrary{arrows}
\usetikzlibrary{decorations.markings}
\usepackage{color}
\usepackage{graphicx}
\usepackage{fullpage}
\usepackage{float}
\usepackage{wrapfig}
\usepackage{caption}
\usepackage{subcaption}
\usepackage{listings}
\usepackage{verbatim}
\numberwithin{equation}{section}

\DeclareMathOperator{\Var}{Var}

\DeclareMathOperator*{\im}{Im}
\DeclareMathOperator*{\re}{Re}

\newcommand{\beq}{ \begin{equation} }
\newcommand{\eeq}{ \end{equation} }
\newcommand{\beqq}{ \begin{equation*} }
\newcommand{\eeqq}{ \end{equation*} }

\def \dd {\mathrm{d}}

\def \a {\alpha}
\def \b {\beta}
\def \g {\gamma}
\def \e {\varepsilon}
\def \d {\delta}
\def \la {\lambda}

\def \w {\omega}

\newcommand{\bR}{\mathbb{R}}

\newcommand{\bC}{\mathbb{C}}

\newcommand{\bE}{\mathbb{E}}
\newcommand{\bP}{\mathbb{P}}

\newcommand{\cF}{\mathcal{F}}

\def \cF {\mathcal{F}}
\def \EE {\mathbb{E}}
\def \PP {\mathbb{P}}
\def \RR {\mathbb{R}}
\def \CC {\mathbb{C}}
\def \cN {\mathcal{N}}
\def \cG {\mathcal{G}}

\def \gt {\tilde{\gamma}}
\def \tg {\tilde{\gamma}}
\def \hg {\hat{\gamma}}

\newcommand{\ii}{\mathrm{i}}

\def \hG {\widehat{G}}
\newcommand{\Gmu}{\widehat{G}}
\def \muone {\mu_1^{(1)}}
\def \mutwo {\mu_1^{(2)}}
\def \muf {\mu_1^{(1)}}
\def \mus {\mu_1^{(2)}}
\def \event {\mathcal{E}_\e}
\def \Fone {\mathcal{F}^{(1)}_\d}
\def \Ftwo {\mathcal{F}^{(2)}_K}
\def \Fthree {\mathcal{F}^{(3)}_{s,t}}
\def \Ffour {\mathcal{F}^{(4)}_{r,R}}
\def \Gs {\Gamma_s}
\def \bv {\mathbf{v}}

\newtheorem{theorem}{Theorem}[section]

\newtheorem{corollary}[theorem]{Corollary}
\newtheorem{lemma}[theorem]{Lemma}

\newtheorem{definition}[theorem]{Definition}

\theoremstyle{remark}
\newtheorem{remark}[theorem]{Remark}

\DeclareMathOperator{\TW}{TW}
\DeclareMathOperator{\MP}{MP}

\begin{document}
\title{Free energy of the bipartite spherical SK model at critical temperature}
\author{Elizabeth W. Collins-Woodfin\footnote{Department of Mathematics \& Statistics, McGill University,
Montreal, QC, H3A 0G4, Canada \newline email: \texttt{elizabeth.collins-woodfin@mail.mcgill.ca}} \and Han Gia Le\footnote{Department of Mathematics, University of Michigan,
Ann Arbor, MI, 48109, USA \newline email: \texttt{hanle@umich.edu}}}
	\date{\today}

	\maketitle

\begin{abstract}
The spherical Sherrington--Kirkpatrick (SSK) model and its bipartite analog both exhibit the phenomenon that their free energy fluctuations are asymptotically Gaussian at high temperature but asymptotically Tracy--Widom at low temperature.  This was proved in two papers by Baik and Lee, for all non-critical temperatures.  The case of critical temperature was recently computed for the SSK model in two separate papers, one by Landon and the other by Johnstone, Klochkov, Onatski, Pavlyshyn.  In the current paper, we derive the critical temperature result for the bipartite SSK model.  In particular, we find that the free energy fluctuations exhibit a transition when the temperature is in a window of size $n^{-1/3}\sqrt{\log n}$ around the critical temperature, the same window for the SSK model.  Within this transitional window, the asymptotic fluctuations of the free energy are the sum of independent Gaussian and Tracy--Widom random variables.
\end{abstract}

\tikzset
{decoration=
	{markings,
		=at position 0.5 with {\arrow{stealth}}
	},
	plain/.style={line width=0.8pt},
	arrow/.style={plain,postaction=decorate}
}

\section{Introduction}
The Sherrington--Kirkpatrick (SK) and spherical Sherrington--Kirkpatrick (SSK) models devised in the 1970s are two classical examples of mean-field spin models in which the magnetic behavior of $N$ particles, encoded in a spin vector $\boldsymbol{\sigma}$, is governed by their identically distributed random pairwise interactions.  The SK model has Ising spins $\boldsymbol{\sigma}\in\{-1,1\}^N$, and SSK is the continuous analog with $\boldsymbol{\sigma}\in\{\RR^N:\|\boldsymbol{\sigma}\|^2=N\}$.  For a detailed exposition on these models, we refer readers to the book by Panchenko \cite{PanchenkoSKBook}. One limitation of these models is that their mean-field structure, meaning that all pairs of particles interact according to the same rule.  With the aim of 
 reflecting inhomogeneities and community structures (e.g., in theoretical biology, social and neural networks),  scholars have developed various extensions beyond mean-field models.  

One extension is the multi-species model, in which the set of $N$ spins is partitioned into a fixed number of disjoint subsets or ``species" \cite{BCMT15}. The random interactions between spins are not identically distributed as in the SK and SSK models, but rather have variance depending on the species structure. For a $k$-species model,  the covariance structure can be encoded in a $k\times k$ matrix $\Delta^2$, where $\Delta^2_{s,t}$ denotes the variance of the random interaction between a spin in species $s$ and a spin in species $t$.  In bipartite models, $k=2$ and $\Delta^2=\left( \begin{smallmatrix} 0 & 1 \\ 1 & 0 \end{smallmatrix} \right)$, meaning that interactions are only between spins in different species.  Bipartite models have important applications in biology and neural networks \cite{bipartitebiology},
\cite{bipartiteneuralnet1}, \cite{bipartiteneuralnet2}.
Another multi-species model (with applications in artificial intelligence) is the deep Boltzmann machine, where the species or ``layers" are ordered, and interactions are only between spins in adjacent layers \cite{SH09, ACM20, ACM21a, ACM21b, Genovese23}.

Another direction of generalizing the SK and SSK models is to allow interactions, not only between pairs, but among groups of spins.  A $p$-spin model has interactions among groups of $p$ spins.  Likewise, a $(p,q)$-spin bipartite model, has interactions between a group of $p$ spins from one species and a group of $q$ spins from the other species. The case of spherical spins for this model was studied by Auffinger and Chen \cite{bipartiteAuffingerChen}, where they obtained a minimization formula for the limiting free energy at sufficiently high temperature. 

The current paper focuses on the bipartite $(1,1)$-spin SSK model. 
The set-up for this model is as follows.  Given two positive integers $n,m$, we define spin variables
\[
\boldsymbol{\sigma}=(\sigma_1,\sigma_2,...,\sigma_n)\in S_{n-1},\quad\boldsymbol{\tau}=(\tau_1,\tau_2,...,\tau_m)\in S_{m-1},
\]
where
\[
S_{n-1}=\{\mathbf{u}\in\RR^n: \|\mathbf{u}\|^2=n\}.
\]
The Hamiltonian for the model is given by
\[
H(\boldsymbol\sigma,\boldsymbol\tau)=\frac{1}{\sqrt{n+m}}\sum_{i=1}^n\sum_{j=1}^m J_{ij}\sigma_i\tau_j
\]
where $J_{ij}$ are independent, standard Gaussian random variables.  The Gibbs measure and the free energy for this model at inverse temperature $\b>0$ are 
\beq\label{def:Fn}
p(\boldsymbol\sigma,\boldsymbol\tau)=\frac{1}{Z_{n,m}}e^{\beta H(\boldsymbol\sigma,\boldsymbol\tau)},\quad
F_{n,m}(\b)=\frac{1}{n+m}\log Z_{n,m},
\eeq
respectively, where $Z_{m,n}$ is a normalization factor (i.e. partition function), 
\beq\label{eq:Znm_apriori}
Z_{n,m}=\int_{S_{m-1}}\int_{S_{n-1}}e^{\beta H(\boldsymbol\sigma,\boldsymbol\tau)}\dd\omega_{n}\dd\omega_{m},
\eeq
and $d\omega_{n}$ is the uniform probability measure on $S_{n-1}$.  

\subsection{Background and related literature}

The free energy of SK and SSK has been well-studied, although more is known in the spherical setting.  The limiting free energy was first conjectured by Parisi for SK \cite{Parisi80} and Crisanti--Sommers for SSK \cite{crisanti1992sphericalp} and both conjectures were rigorously proved by Talagrand \cite{TalagrandSK,TalagrandSSK}.  The fluctuations of the SK model are only known at high temperature \cite{ALR87,Banerjee_2019,CN95,FZ87}, but more is known for the spherical model, where additional analytic techniques are available.  In 2016, Baik and Lee analyzed the fluctuations of the SSK free energy at non-critical temperature and found that the fluctuations at high temperature are asymptotically Gaussian while those at low temperature are asymptotically Tracy--Widom \cite{BaikLeeSSK}. The fluctuations at the critical temperature was left open. 

The fluctuations at critical temperature of SSK free energy were studied by Landon \cite{Landon_crit} and by Johnstone, Klochkov, Onatski and Pavlyshyn \cite{JKOP2}, independently. Both papers showed that the critical scaling for the inverse temperature is $\beta=\b_c+bn^{-1/3}\sqrt{\log n}$.  Landon proved that, for fixed $b\leq0$ and for $b\to 0$, the fluctuations are Gaussian while, for $b\to+\infty$ at any rate, the fluctuations are Tracy--Widom. For fixed $b>0$, Landon showed tightness but did not obtain the limiting distribution. On the other hand, Johnstone et al. were able to compute fluctuations for all fixed $b$.  Their result for $b\leq0$ agrees with that of Landon and, for $b>0$, they showed that the fluctuations are a sum of independent Gaussian and Tracy--Widom random variables.
 
Departure from mean-field structure generally leads to more challenging analysis. While the problem of limiting free energy is solved for general one-species mixed $p$-spin SK and SSK models \cite{Parisi80,TalagrandSK,TalagrandSSK,Panchenko14,Chen13}, limiting results remain incomplete for the multi-species and $(p,q)$-spin bipartite models. For the multi-species SK model, limiting free energy is only verified under the assumption of positive-definite  $\Delta^2$ (Barra et al. \cite{BCMT15} proposed a Parisi-type formula and proved an upper bound, and Panchenko \cite{Panchenko15} proved a matching lower bound). For general $\Delta^2$, we only have a lower bound \cite{Panchenko15}. The bipartite model, one of the most natural multi-species examples, belongs to the indefinite $\Delta^2$ case, and is still open in the case of Ising spins (conjecture on the limiting free energy was made \cite{BGGPT14, BGG11}).  When it comes to fluctuations, a central limit theorem (CLT) for the free energy of the two-species SK model for general $\Delta^2$ was obtained in high temperature by \cite{Liu21}. 

For the bipartite SSK model, more is known. Baik and Lee \cite{BaikLeeBipartite} obtained both the limit and the asymptotic fluctuations of the free energy, at all non-critical temperatures. More specifically, assuming $n,m\to\infty$ with $n/m=\lambda+O(n^{-1-\delta})$ for some $\lambda,\delta>0$, they provided explicit formulas for the first two terms in the asymptotic expansion of the free energy for $\b\neq\b_c$, where the critical inverse temperature $\b_c$ is equal to $\sqrt{1+\la}/\la^{1/4}$. The formulas imply that fluctuation is Gaussian with order $n^{-1}$ for $\b<\b_c$ (high temperature), and is GOE Tracy--Widom of order $n^{-2/3}$ for $\b>\b_c$ (low temperature). 

See \cite{bipartiteAuffingerChen, DW21,BS22, Subag23, Subag23_crittemp} for high temperature results for more general Ising or spherical spin models. 



\subsection{Main theorem}

The goal of this paper is to compute the fluctuations of the free energy in a transitional window around the critical temperature for the bipartite (1,1)-spin SSK model. In particular, this includes detailed knowledge of the free energy at the critical temperature, providing another result on critical temperature among spin glass models, in addition to the independent results of Landon \cite{Landon_crit} and of Johnstone et al. \cite{JKOP2}.  

We state our main result in the following theorem.

\begin{theorem}\label{thm:main} Let $F_{n,m}(\b)$ denote the free energy of a bipartite SSK spin glass, given by \eqref{def:Fn}, where the species sizes $n,m$ satisfy  $n/m=\la+O(n^{-1})$, for some constant $\la\in(0,1]$, as $n,m\to\infty$.  When the inverse temperature is at the critical scaling, namely $\b=\b_c+bn^{-1/3}\sqrt{\log n}$ for fixed $b$ and $\b_c:=\sqrt{1+\la}/\la^{1/4}$, the limiting distribution of the free energy is given by the formula below and this convergence holds in distribution.
\beq
\frac{n+m}{\sqrt{\frac16 \log n}}\left(F_{n,m}(\b)-F(\b)+\frac1{12}\frac{\log n}{n+m}\right)\rightarrow \mathcal{N}(0,1)+\frac{\sqrt{6}(1+\la)^{\frac12}b_+}{\la^{\frac34}(1+\sqrt{\la})^{\frac23}}\TW_1
\eeq
where $\TW_1$ denotes the real Tracy--Widom distribution that is independent from the standard normal $\cN(0,1)$ and $b_+$ denotes the positive part of $b$.  The limiting free energy is given by
\beq
F(\b)=\begin{cases}
\frac{\beta^2}{2\beta_c^4}& \quad \text{for }\beta<\beta_c\\
f_\la+\frac{\la}{1+\la}A\left((1+\sqrt{\la})^{2},\frac{\b}{\sqrt{\la(1+\la)}}\right)-\frac12\log\b-\frac{\la}{2(1+\la)}C_\la & \quad \text{for }\beta\geq\beta_c
\end{cases}
\eeq
where 
\beq\label{eq:theoremquantities}\begin{split}
f_\la&=-\frac12+\frac{\la-1}{2(\la+1)}\log2+\frac{\la-1}{4(\la+1)}\log\la+\frac14\log(1+\la),\\
A(x,B)&=\sqrt{\alpha^2+xB^2}-\alpha\log\left(\frac{\alpha+\sqrt{\alpha^2+xB^2}}{2B}\right),\\
C_\la&=(1-\la^{-1})\log(1+\la^{1/2})+\log(\la^{1/2})+\la^{-1/2}.
\end{split}\eeq
\end{theorem}

\subsection{Overview of the proof methods}

One valuable tool in the analysis of the free energy for SSK and bipartite SSK models is a contour integral representation for the partition function ($Z_{n,m}$ in our model). A priori, the partition function of SSK is a surface integral on a high dimensional sphere (or two spheres in the bipartite case).  However, this can be rewritten in terms of contour integrals in the complex plane, which are significantly easier to analyze.  The contour integral representation for the SSK partition function was first observed by Kosterlitz, Thouless, and Jones \cite{kosterlitz1976spherical}.  The analogous representation for the spherical bipartite model, which we use in the current paper, was derived by Baik and Lee \cite{BaikLeeBipartite}.

Armed with this contour integral representation, our analysis can be broken into two broad stages: (1) use steepest descent analysis to obtain an asymptotic expansion for the free energy and (2) analyze the limiting fluctuations using tools from random matrix theory.  This general procedure has been followed in several recent papers on spherical spin glasses, including \cite{Landon_crit} and \cite{JKOP2} in their analysis of SSK at critical temperature.  While much of our analysis is inspired by the methods in these two papers, the bipartite setting introduces certain technical challenges beyond those that arise for unipartite SSK.

One challenge in the bipartite setting is that the representation for $Z_{n,m}$ is a double contour integral, rather than the single integral that arises for SSK.  This makes the process of contour deformation and steepest descent analysis more delicate, particularly on the low temperature side of the critical threshold, where the  contour passes very close to the (random) singularities of the integrand.  Another challenge in the bipartite setting is that the underlying random matrix is a Laguerre Orthogonal Ensemble (LOE) rather than the Gaussian Orthogonal Ensemble (GOE) that appears for SSK (more background on random matrices is in Section \ref{sec:prelim}).  While these ensembles have many similarities, certain analyses are more complicated for LOE.  

From the steepest descent analysis, we obtain an asymptotic expansion for the free energy near the critical temperature, which depends on a sum of the form $\sum_{i=1}^n\log(\g-\mu_i)$.  This is a logarithmic linear statistic of the eigenvalues $\{\mu_i\}_{i=1}^n$ of LOE.   The CLT for this quantity is well-known in random matrix theory in the case where $\g-d_+>c$ for some constant $c$ and $d_+$ being the upper edge of the matrix spectrum (see, e.g., \cite{BaiSilverstein, Lytova09, BW10}). However, this standard CLT for linear eigenvalue statistics does not address the case where $\g$ approaches $d_+$ as $n\to\infty$, which is precisely the scenario that arises when analyzing the free energy at critical temperature.  Thus, we need an ``edge CLT" to treat the case where $\g\to d_+$.  A similar challenge arises for the SSK model at critical temperature, where the log linear statistic depends on eigenvalues of GOE.  The edge CLT for this statistic in the GOE case can be found in \cite{lambertpaquette,JKOP1}, and these works provide a necessary ingredient for the analysis of SSK free energy at critical temperature.

When we began the current project, an analogous edge CLT for LOE did not exist in the literature.  To fill this gap, we proved the following theorem in a separate paper \cite{CWL_CLT}.

\begin{theorem}[Collins-Woodfin, Le \cite{CWL_CLT}]\label{thm:CLT}
Let $M_{n,m}$ be an LOE matrix with $n,m,\la,C_\la,d_+$ as above. Let $\gamma=d_++\sigma_n n^{-2/3}$ with $-\tau< \sigma_n\ll (\log n)^2$ for some $\tau>0$.  Then,
\beq
\frac{\sum_{i=1}^n\log|\g-\mu_i|-C_\la n - \frac{1}{\la^{1/2}(1+\la^{1/2})}\sigma_n n^{1/3}
+\frac{2}{3\la^{3/4}(1+\la^{1/2})^2}\sigma_n^{3/2} +\frac16\log n}{\sqrt{\frac23 \log n}}\to \mathcal{N}(0,1).
\eeq
\end{theorem}
The above result is essential in proving Theorem \ref{thm:main} as it is the source of the Gaussian term in the limiting distribution.

The last step of our proof is to show the asymptotic independence of the Gaussian and Tracy--Widom terms in the limiting distribution.  This involves a recurrence on the entries of the tridiagonal representation of LOE.  In the course of this analysis, we prove a result that may be of independent interest, namely that the largest eigenvalue of an $n\times n$ LOE matrix depends (asymptotically) on a minor of size $n^{1/3}\log^3n$.  This result is well known numerically (e.g. \cite{Edelman2013}), but we have not found an explicit proof of it in the literature.

\subsection{Organization}
In Section \ref{sec:prelim}, we provide a more detailed set-up of the problem along with various probability, spin glass, and random matrix theory results that will be used throughout the paper.  Sections \ref{sec:hightemp} and \ref{sec:lowtemp} contain our analysis of the free energy for $\beta=\beta_c+bn^{-1/3}\sqrt{\log n}$ in the cases of $b<0$ (high-critical temperature) and $b>0$ (low-critical temperature) respectively.  The case of $b=0$ is also addressed in Section \ref{sec:lowtemp}.  Finally, in Section \ref{sec:independence}, we prove the asymptotic independence of the Gaussian and Tracy--Widom terms in the main theorem.  Appendices \ref{sec:appendix} and \ref{sec:appendix2} provide proofs of some technical lemmas from Sections \ref{sec:prelim} and \ref{sec:independence} respectively.

\section{Set-up and preliminaries}\label{sec:prelim}
\subsection{Preliminaries for bipartite SSK model}

\subsubsection*{Double contour integral representation of free energy}

One of the key tools that enables us to precisely calculate the free energy and its fluctuations is a contour integral representation of the partition function.  A priori, $Z_{n,m}$ is given by the surface integral in \eqref{eq:Znm_apriori}.  The contour integral representation of $Z_{n,m}$ was derived by Baik and Lee \cite{BaikLeeBipartite}.  
For the bipartite model, we assume, without loss of generality, that $n\leq m$.  We use $S^{n-1}$ to denote the unit $n$-sphere (as opposed to $S_{n-1}$, which denotes the $n$-sphere of radius $\sqrt{n}$).  Then the partition function can be written as \cite{BaikLeeBipartite}
\beq
Z_{n,m}(\beta)=\frac{2^n}{|S^{m-1}||S^{n-1}|}\left(\frac{\pi^2(n+m)}{m^2n\b^2}\right)^{\frac{n+m-4}{4}}Q(n,\alpha_n,B_n)
\eeq
where
\beq
Q_n:=Q(n,\alpha_n,B_n)=-\int_{\gamma_1-\ii\infty}^{\gamma_1+\ii\infty}\int_{\gamma_2-\ii\infty}^{\gamma_2+\ii\infty}e^{nG(z_1,z_2)}\dd z_2\dd z_1
\eeq
and $G(z_1,z_2)$ is a random function depending on the eigenvalues $\mu_1\geq\mu_2\geq\cdots\geq\mu_n$ of $\frac1m JJ^T$.  The parameters $\gamma_1,\gamma_2$ can be any positive real numbers satisfying $4\gamma_1\gamma_2>\mu_1$.  The function $G$ is defined as
\beq\label{eq:G_def}
G(z_1,z_2):=B_n(z_1+z_2)-\frac{1}{2n}\sum_{i=1}^n\log(4z_1z_2-\mu_i)-\a_n\log z_1
\eeq
where
\beq
\a_n:=\frac{m-n}{2n}, \quad B_n:=\frac{m}{\sqrt{n(n+m)}}\b
\eeq
Using this contour integral representation of $Z_{n,m}(\b)$, the free energy of the bipartite SSK is 
\begin{equation}
F_{n,m}(\beta) = \frac{1}{n+m}\log Q(n,\alpha_n, B_n) + \frac{1}{n+m}\log\left(\frac{2^n}{|S^{n-1}||S^{m-1}|}\left(\frac{\pi^2(n+m)}{m^2n\beta^2}\right)^{\frac{n+m}{4}-1}\right).
\end{equation}		
By direct computation, the second term of the right hand side is 
$f_\la-\frac12\log\b+\frac{\la}{1+\la}\frac{\log n}{n}+O(n^{-1})$ as $n\to \infty$,
where $f_\la$ is as defined in \eqref{eq:theoremquantities}.
We obtain
\begin{equation}\label{eqn:Fn}
F_{n,m}(\beta) = \frac{1}{n+m}\log Q(n,\alpha_n, B_n) + f_\la-\frac12\log\b+\frac{\la}{1+\la}\frac{\log n}{n}+O(n^{-1})
\end{equation}
so the computation of the free energy boils down to computing the integral $Q_n$.  In order to compute this integral via steepest descent analysis, one needs to find a critical point of $G(z_1,z_2)$.  Baik and Lee show that there exists a critical point $(z_1,z_2)$ such that both coordinates are positive real and $4z_1z_2>\mu_1$.  We can choose the contours of the double integral to pass through this critical point, which has coordinates
\beq\label{eq:g1g2_def}
(\gamma_1,\gamma_2)=\left(\frac{\a_n+\sqrt{\a_n^2+\gamma B_n^2}}{2B_n},
\frac{-\a_n+\sqrt{\a_n^2+\gamma B_n^2}}{2B_n} \right)
\eeq
where $\gamma$ is the unique real number greater than $\mu_1$ satisfying
\beq\label{eq:gamma_implicit}
\frac1n\sum_{i=1}^n\frac{1}{\gamma-\mu_i}=\frac{B_n^2}{\a_n+\sqrt{\a_n^2+\gamma B^2}}.
\eeq
We see that $\gamma$ is implicitly a function of the eigenvalues of $\frac1m JJ^T$, which is a normalized Laguerre Orthogonal Ensemble (i.e. real Wishart matrix).  Later in this section we recount some important properties of this matrix ensemble that will be used throughout the paper.

\subsubsection*{Critical inverse temperature $\b_c$ and critical window}
 As stated above, the critical inverse temperature of the bipartite SSK model is $\b_c=\sqrt{1+\la}/\la^{1/4}$.  At this value of $\b$, one sees a transition in the behavior of the critical point $\gamma$.  We give a brief, heuristic description of the transition here and provide more details in the next two sections.

The equation \eqref{eq:gamma_implicit}, which is random and $n$-dependent, can be approximated by its deterministic, $n$-independent analog 
\beq\label{eq:gamma_implicit_infty}
\int_\RR\frac{1}{z-x}p_{\MP}(x)\dd(x)=\frac{B^2}{\a+\sqrt{\a^2+z B^2}}
\eeq
where $p_{\MP}$ denotes the Mar{\v{c}}enko--Pastur measure (see definition in equation \eqref{eq:def_MP} below) and $\a,B$ are given by
\beq
\a:= \frac{1-\la}{2\la},\quad B:=\frac{\b}{\sqrt{\la(1+\la)}}. 
\eeq
If the equation \eqref{eq:gamma_implicit_infty} is to be of any use, then it should provide a solution $z\in(d_+,\infty)$ that is close to the solution $\gamma$ of \eqref{eq:gamma_implicit} (with high probability and for all sufficiently large $n$).  Labeling the left and right sides of \eqref{eq:gamma_implicit_infty} as $L_\infty(z)$ and $R_\infty(z)$ respectively, Baik and Lee \cite{BaikLeeBipartite} observe that $\frac{L_\infty(z)}{R_\infty(z)}$ is a decreasing function of $z\in(d_+,\infty)$ with
\beq
\lim_{z\to\infty}\frac{L_\infty(z)}{R_\infty(z)}=0,\quad
\lim_{z\downarrow d_+}\frac{L_\infty(z)}{R_\infty(z)}=\frac{L_\infty(d_+)}{R_\infty(d_+)}.
\eeq
Hence, \eqref{eq:gamma_implicit_infty} has a solution $z\in(d_+,\infty)$ if and only if $L_\infty(d_+)>R_\infty(d_+)$.  We call this solution $\gt$.  By setting $L_\infty(d_+)=R_\infty(d_+)$ and solving for $\b$, one obtains the critical inverse temperature.  The implication of this is that, for $\b<\b_c$ (high temperature), $\g$ can be approximated by $\gt$, and this deterministic approximation turns out to be very accurate.  However, for $\b>\b_c$ (low temperature), \eqref{eq:gamma_implicit_infty} can't be used to approximate $\g$, since it has no solution in $(d_+,\infty)$.  Intuitively, this is due to the fact that, at low temperature, $\g$ is very close to the eigenvalue $\mu_1$ and may be above or below $d_+$, depending on the value of $\mu_1$.  A detailed analysis of $\g$ in these two cases is provided in Sections \ref{sec:hightemp} and \ref{sec:lowtemp}.

Finally, we comment on the scaling of the critical temperature window, $\b=\b_c+O(n^{-1/3}\sqrt{\log n}$).  One can conjecture this critical scaling from the theorem of Baik and Lee by matching the order of the variance of the free energy at high and low temperature.  For fixed $\b<\b_c$, the free energy has variance of order $\frac{1}{n^2 \log(\b_c-\b)}$ while, for fixed $\b>\b_c$, the free energy has variance of order $(\b-\b_c)^2n^{-4/3}$.  By formally equating these, we find that their order matches when $\b-\b_c=\Theta(n^{-1/3}\sqrt{\log n})$ and we conjecture that the variance of the free energy in this critical scaling should be of order $n^{-2}\log n$.  This conjecture turns out to be correct, as we will see in the subsequent sections.

\subsection{Probability and random matrix preliminaries}\label{sec:prelimRMT}
\subsubsection*{Notational conventions (probability and asymptotics)}
Below are several asymptotic notations that we use along with the definitions that we follow.  For any sequence $\{a_n\}$ and positive sequence $\{b_n\}$, we write
\begin{itemize}
\item $a_n=O(b_n)$ if there exists some constant $C$ such that $|a_n|\leq Cb_n$ for all $n$,
\item $a_n=\Omega(b_n)$ if there exists some constant $C$ such that $|a_n|\geq Cb_n$ for all $n$,
\item $a_n=\Theta(b_n)$ if there exist constants $C_1,C_2$ such that $C_1b_n\leq |a_n|\leq C_2b_n$ for all $n$\\ (or, equivalently, $a_n=O(b_n)$ and $a_n=\Omega(b_n)$),
\item $a_n\ll b_n$ if $\lim_{n\to\infty} a_n/b_n=0$,
\item $a_n\gg b_n$ if $\lim_{n\to\infty} b_n/a_n=0$.
\end{itemize}
In addition, we sometimes need to make asymptotic statements about the probability of events in a sequence $\{E_n\}$.  We say that $E_n$ occurs ``asymptotically almost surely" if $\PP(E_n)\to1$ as $n\to\infty$.  We say $E_n$ occurs ``with overwhelming probability" if, for all $D>0$, there exists $n_0$ such that $\PP(E_n)>1-n^{-D}$ for all $n>n_0$.

\subsubsection*{Laguerre Orthogonal Ensemble and Mar{\v{c}}enko--Pastur measure}
As we saw in the previous subsection, the eigenvalues of the matrix $\frac1m JJ^T$ will play an important role in our analysis.  This is a normalized Laguerre Orthogonal Ensemble and we provide an overview of some of its key properties here.  Mar{\v{c}}enko and Pastur \cite{Marcenko_1967} showed that the empirical spectral measure of LOE has the following convergence, as $n,m\to\infty$ with $n/m\to\la\leq1$,
	\beq\label{eq:def_MP}
	\frac1n \sum_{i=1}^n\delta_{\mu_i}(x)\to p_{\MP}(x)\dd x:= \frac{\sqrt{(d_+-x)(x-d_-)}}{2\pi\la x}\mathbf{1}_{[d_-,d_+]}(x)\dd x.
	\eeq
	The convergence is weakly in distribution and $d_\pm=(1\pm\la^{1/2})^2$ and $p_{\MP}(x)$ is referred to as the Mar{\v{c}}enko--Pastur measure.  
	In working with $p_{\MP}$, we sometimes need to use its Stieltjes transform
\beq
s_{\MP}(z):=\int_\RR\frac{1}{z-x}p_{\MP}(x)\dd x.
\eeq
We note that it is common to define the Stieltjes transform as the negative of what we use here.  However, our definition is consistent with that of \cite{BaikLeeBipartite} and is more logical in this context, since it results in a positive value of $s_{\MP}$ for our setting.

\subsubsection*{Tracy--Widom distribution}
The location of the largest eigenvalue is particularly important in our analysis.  The following result is well-known in random matrix theory. See, for example, \cite{Johnstone2001, Soshnikov2002} and Corrollary 1.2 of \cite{PillaiYin2014}.
\begin{lemma}\label{lem:TW_mu1}
	Let $\mu_1$ be the largest eigenvalue of $\frac1m M_{n,m}$, where $M_{n,m}$ is an $n\times n$ matrix from the Laguerre orthogonal ensemble. Then the following convergence in distribution holds.
	\begin{equation*}
		\frac{m\mu_1-(\sqrt{n}+\sqrt m)^2}{(\sqrt{n}+\sqrt m)\left((1/\sqrt n)+(1/\sqrt m)\right)^{1/3}}\to \TW_1.
	\end{equation*}
\end{lemma}
Under the condition $n/m\to\la\in(0,1]$, the following form of Lemma \ref{lem:TW_mu1} is useful in our paper.
\begin{equation}\label{eqn:mu1_rescaled}
\frac{n^{\frac23}(\mu_1-d_+)}{\la^{\frac12}(1+\la^{\frac12})^{\frac43}}\to \TW_1.
\end{equation}

\subsubsection*{Classical eigenvalue locations and rigidity}
A key tool in our analysis is to approximate the eigenvalues by their ``classical locations" (i.e. the quantiles of the Mar{\v{c}}enko--Pastur measure).  The classical locations $\{g_i\}$ are defined by the relation
	\beq
	\frac{i}{n}=\int_{g_i}^{d_+}p_{\MP}(x)\dd x.
	\eeq
	Using this definition, one can show that
	\begin{equation}\label{eqn:classical_loc}
			g_i=d_+-\left(\frac{3\pi \la^{3/4}d_+i}{2n}\right)^{2/3}+O\left(\frac{i^{4/3}}{n^{4/3}}\right),\quad i\leq n/2.  
		\end{equation}
	Thus, we expect that, for $i\ll n$, we will have $\mu_i\approx d_+-\left(\frac{3\pi \la^{3/4}d_+i}{2n}\right)^{2/3}$.  The concept of ``eigenvalue rigidity" means that eigenvalues are close to their classical locations with high probability.  More precisely, we define \textbf{eigenvalue rigidity} to be the event
	\[
	\bigcap_{1\leq i\leq n}\left\{|\mu_i-g_i|\leq \frac{n^\d}{n^{2/3}\min\{i^{1/3},(n+1-i)^{1/3}\}} \right\},
	\]
which holds with overwhelming probability.  This is proved in \cite{PillaiYin2014}(Theorem 3.3) in the case $\la\in(0,1)$. For $\la=1$, the result follows from Corollary 1.3 of \cite{OLT2014} and the relation $p_{\MP}(x)=p_{\mathrm{SC}}(\sqrt{x})$ between the Mar{\v{c}}enko--Pastur and semicircle distributions. 	
	
In addition to eigenvalue rigidity, we sometimes need more precise control of the larger eigenvalues. For this purpose, we introduce the following lemma, which is proved in Appendix \ref{sec:appendix}.  This lemma is inspired by a similar one proved in \cite{LandonSosoe} for GOE matrices  and used by Landon in his analysis of SSK at critical temperature \cite{Landon_crit}.
	\begin{lemma}\label{lem:Aj}
		Let $\{\mu_j\}_{j=1}^n$ be the eigenvalues of $\frac1m M_{n,m}$. For each $j$, define
		\begin{equation}\label{eqn:Aj}
			A_j=\left(\frac{3\pi \la^{3/4}d_+}{2}j\right)^{2/3}-n^{2/3}(d_+-\mu_j).
		\end{equation}Given $\e>0$, there exists $K$ such that for sufficiently large $n$,
		\begin{equation}\label{eqn:probAj}
			\bP\left(\bigcap_{K\leq j \leq n^{2/5}}\left\{\left|A_j\right|\leq\la j^{2/3}\right\}\right)\geq 1-\e.
		\end{equation}
		Furthermore, there exists $C,c>0$ such that 
		\begin{equation}
			\bE\left[\mathbbm{1}_{\{n^{2/3}(\mu_j-d_+)\leq-C\}}\left|A_j\right|\right]\leq \frac{c\log j}{j^{1/3}}, \quad \text{for } K\leq j\leq n^{2/5}.
		\end{equation}
	\end{lemma}

\subsubsection*{Tridiagonal representation of LOE}
In Section \ref{sec:independence}, when proving the asymptotic independence of the Gaussian and Tracy--Widom variables, we will need the tridiagonal representation of LOE. Dumitriu and Edelman \cite{DumitriuEdelman} show that the eigenvalue distribution of the unnormalized LOE matrix $M_{n,m}$ is the same as that of the $n\times n$ matrix $T_n=BB^T$ where $B$ is a bi-diagonal matrix of dimension $n\times n$.  In particular, 
\beq\label{eq:tridiagonal}
B=
\begin{bmatrix}
a_1&&&&\\
b_1&a_2&&&\\
&b_2&a_3&&\\
&&\ddots&\ddots&\\
&&&b_{n-1}&a_n
\end{bmatrix}
\quad\text{so}\quad
BB^T=
\begin{bmatrix}
a_1^2&a_1b_1&&&\\
a_1b_1&a_2^2+b_1^2&a_2b_2&&\\
&a_2b_2&a_3^2+b_2^2&&\\
&&&\ddots&a_{n-1}b_{n-1}\\
&&&a_{n-1}b_{n-1}&a_n^2+b_{n-1}^2
\end{bmatrix}
\eeq
where $\{a_i\},\{b_i\}$ are all independent random variables with distributions satisfying
\beq\label{eqn:aibi}
a_i^2\sim \chi^2 (m-n+i),\qquad
b_i^2\sim \chi^2( i).
\eeq

\subsection{Defining the event on which our results hold}
Our arguments throughout this paper rely upon certain conditions on the eigenvalues, which hold with probability close to 1.  To streamline the later proofs, we collect in this section various events involving the eigenvalues $\{\mu_i\}$ and provide probability bounds for each event.  Finally, we define $\event$ to be the intersection of these events, which holds with probability $1-\e$ for arbitrarily small choice of $\e$.
\begin{definition}
Let $\delta, s,t, r, R$ be positive numbers where $s<t$, $r<R$, and let $K$ be a positive integer. We define the events $\Fone,\Ftwo,\Fthree,\Ffour$ as follows. 
\begin{align}
\Fone&=\bigcap_{1\leq i\leq n}\left\{|\mu_i-g_i|\leq \frac{n^\d}{n^{2/3}\min\{i^{1/3},(n+1-i)^{1/3}\}} \right\},\\
\Ftwo&=\bigcap_{K\leq j\leq n^{2/5}}\left\{\left|n^{2/3}(\mu_j-d_+)+\left(\frac{3\pi\la^{3/4}d_+}{2}j\right)^{2/3}\right|\leq\frac{j^{2/3}}{10} \right\},\\
\Fthree&=\left\{n^{2/3}|d_+-\mu_1|\in[s,t] \right\},\qquad0<s<t,\\
\Ffour&=\left\{ r<n^{2/3}(\mu_1-\mu_2)<R\right\}. 
\end{align}
\end{definition}
\begin{remark}
The event $\Fone$ is the eigenvalue rigidity condition with respect to the ``classical location", and $\Ftwo$ is inspired by a similar event used in the context of Gaussian ensembles by Landon and Sosoe \cite{LandonSosoe}.
\end{remark}
\begin{lemma}[Event probability bounds]\label{lem:eventprob}The following statements hold.
\begin{itemize}
\item For any fixed $\delta>0$, the event $\Fone$ holds with overwhelming probability.  
\item For any $\e>0$ there exist positive constants $K, s, t, r, R$ depending on $\e$ but not on $n$ such that, for sufficiently large $n$,
\[
\PP[\Ftwo]\geq1-\tfrac\e4,\qquad\PP[\Fthree]\geq1-\tfrac\e4,\qquad\PP[\Ffour]\geq1-\tfrac\e4.
\]
\end{itemize}
\end{lemma}
\begin{proof}
The bounds on the first three events are clear. The eigenvalue rigidity condition $\Fone$ holds with overwhelming probability (see explanation in Subsection \ref{sec:prelimRMT}). The bound on event $\Ftwo$ follows directly from Lemma \ref{lem:Aj}, where we can take larger value of $K$ to replace $\e$ in the bound  by $\e/4$. Result on $\Fthree$ is a consequence of the Tracy--Widom convergence in Lemma \ref{lem:TW_mu1}. 

Finally, we consider $\Ffour$. The upper bound $n^{2/3}(\mu_1-\mu_2)\leq R$ holds with probability $1-\e/8$ for some $R>0$ via union bound (where $|\mu_1-d_+|$ is controlled using $\cF_{s,t}^{(3)}$ and $|\mu_2-d_+|$ is bounded similarly using Tracy--Widom convergence of $\mu_2$). For the lower bound on $n^{2/3}(\mu_1-\mu_2)$, note that the joint distribution of $\mu_1$ and $\mu_2$ (each rescaled as in \eqref{eqn:mu1_rescaled}) converges to the distribution given by the Tracy--Widom law (see, for example, \cite{Soshnikov2002}, \cite{PillaiYin2014}). This law describes the joint distribution of the largest two eigenvalues of an operator $\mathbf{H}_1$ whose spectrum is simple with probability one (see, for example, (4.5.9) and Theorem 4.5.42 of \cite{AndersonGuionnetZeitouni}), implying an $r>0$ such that $\bP(n^{2/3}(\mu_1-\mu_2)>r)\geq1-\e/8$ does exist for sufficiently large $n$.
\end{proof}
\begin{definition}
Given $\e>0$, we define $\event$ to be an event
\[
\event:=\Fone\cap\Ftwo\cap\Fthree\cap\Ffour
\]
where the parameters $\d,K,s,t,r, R$ are chosen to satisfy the probability bounds in Lemma \ref{lem:eventprob}.  Note that $K,s,t,r, R$ depend on $\e$, but $\d$ does not. The choice of these constants is not unique. However, for any given $\e>0$, we fix these values and define $\event$ accordingly.
\end{definition}
The following corollary follows directly from the above definition and Lemma \ref{lem:eventprob}.
\begin{corollary}
For any $\e>0$, $\bP[\event]\geq 1-\e$.
\end{corollary}
Computing the free energy in both the high and low temperature regimes involves analyzing linear statistics of eigenvalues of the form
$\sum_{i=1}^n\frac{1}{(z-\mu_i)^k}$, on the event defined above.  The key lemma that we use for handling these sums is the following.
\begin{lemma}\label{lem:diff_stieltjes_gK}
Let $z\in\CC$ with $\re(z)\geq d_+$. Let $\{\mu_i\}$ be the eigenvalues of $\frac1m M_{m,n}$.  Then, for any $\e>0$ and any positive integer $l$,
\beq
	\bE\left[\mathbbm{1}_{\event}
	\left|\frac1n\sum_{j=K}^n\frac1{(z-\mu_j)^l}-\int_{d_-}^{g_K} \frac{1}{(z-y)^l}p_{\MP}(y)dy\right|\right]=
	O\left(n^{\frac23 l-1}\cdot\min\left\{\left|\frac{\log(n^{2/3}|z-d_+|)}{(n^{2/3}|z-d_+|)^l}\right|,\;1\right\}\right).
\eeq
Here, $K$ is the constant depending on $\e$ in $\Ftwo$ and $\event$.
\end{lemma}
A proof of this lemma is included in Appendix \ref{sec:appendix}. The general approach is inspired by the method that Landon and Sosoe used in \cite{LandonSosoe} to bound similar eigenvalue statistics in the case of Gaussian orthogonal ensembles.  We prove a series of supporting lemmas, first for LUE, which allows us to make use of the determinantal properties.  We then extend our final result to LOE by way of the interrelationship between eigenvalues of unitary and orthogonal ensembles provided in \cite{ForresterRains}.

\section{High temperature}\label{sec:hightemp}
As mentioned in the previous section, the computation of the free energy reduces to the computation of the integral
\beq\label{eq:Qn}
Q_n=-\int_{\gamma_1-i\infty}^{\gamma_1+i\infty}\int_{\gamma_2-i\infty}^{\gamma_2+i\infty}e^{nG(z_1,z_2)}dz_2dz_1
\eeq
where $G(z_1,z_2)$ is defined in \eqref{eq:G_def}.  The general idea is that we should be able to compute this integral via steepest descent analysis by deforming the contours such that they pass through the critical point $(\g_1,\g_2)$, which is a function of $\g$ as defined in \eqref{eq:g1g2_def}-\eqref{eq:gamma_implicit}.  Baik and Lee \cite{BaikLeeBipartite} show that at fixed high temperature (i.e. constant $\b<\b_c$), the random variable $\g$ is well-approximated by $\gt$, the solution to \eqref{eq:gamma_implicit_infty}.  Furthermore, $|\g-\gt|$ is small enough that the integral computations can be carried out with $\gt$ and the error remains sufficiently small.

In the high temperature side of the critical window, we do not have fixed $\b<\b_c$ as in \cite{BaikLeeBipartite}, but rather $\b=\b_c+bn^{-1/3}\sqrt{\log n}$ for $b<0$.  The first task of this section is to show that, even in this scaling, $\gt$ remains a good approximation of $\g$. Namely, we need to compute the asymptotics of $\gt$ and obtain an upper bound on $|\g-\gt|$.

\subsection{Bounds on $G$, its derivatives, and its critical point}
We begin with an asymptotic expansion for $\gt$.

\begin{lemma}\label{lem:gt_d+}
For fixed $b<0$, the solution $\gt$ to \eqref{eq:gamma_implicit_infty} satisfies
\beqq
\gt=d_++\frac{4\la b^2}{1+\la}n^{-2/3}\log n+O(n^{-1}(\log n)^{3/2}).
\eeqq
\end{lemma}
\begin{proof}
From \cite{BaikLeeBipartite} (see (6.17)), we obtain the closed-form expression
\beq\label{eqn:gt_dplus}
\gt =(1+\la)\b^{-2}+1+\la+\frac{\la}{1+\la}\b^2.
\eeq
Observe that the right hand side, as a function of $\b$, is equal to $d_+$ at $
\b_c$. Thus, by expanding the function around $\b=\b_c+bn^{-1/3}\sqrt{\log n}$, we obtain
\beq\label{eq:gt-d+initialcomp}
\gt -d_+
=\frac{4\la}{1+\la}(\b-\b_c)^2-\frac{4\la^{5/4}}{(1+\la)^{3/2}}(\b-\b_c)^3+O((\b-\b_c)^4),
\eeq
and the lemma follows.
\end{proof}
In order to obtain a sufficiently tight bound for $|\g-\gt|$, we need bounds on various eigenvalue statistics and, in particular, we need to bound differences of the form
\beq
\frac1n\sum_{i=1}^n\frac{1}{(z-\mu_i)^k}-\int\frac{p_{\MP}(y)\dd y}{(z-y)^k},\quad k\geq1
\eeq
when $z$ is close to $\mu_1$.  Given the precision needed for computations in the critical window, the bound obtained using eigenvalue rigidity is not tight enough.  Instead, we make use of the following lemma.
\begin{lemma}\label{lem:diff_stieltjes}
Let $z\in\CC$ with $\re(z)\geq d_+$ and $|z-d_+|>cn^{-2/3}\log n$ for some $c>0$. Let $\{\mu_i\}$ be the eigenvalues of $\frac1m M_{m,n}$.  Then, for any $\e>0$ and any positive integer $l$,
\beq
\bE\left[\mathbbm{1}_{\event}\left|\frac1n\sum_{j=1}^n\frac1{(z-\mu_j)^l}-\int \frac{p_{\MP}(y)\dd y}{(z-y)^l}\right|\right]=
O\left(n^{\frac23 l-1}\frac{\log(n^{2/3}|z-d_+|)}{(n^{2/3}|z-d_+|)^l}\right).
\eeq
\end{lemma}
	\begin{proof}[Proof of Lemma \ref{lem:diff_stieltjes}]
	Given $\e>0$, let $K$ be the integer in the events $\Ftwo$ and $\event$. 
	Recall the classical locations  $g_i$, $i=0,\dots, n$ of the Mar{\v{c}}enko--Pastur measure. We start by writing $\frac1n\sum_{i=1}^n\frac1{(z-\mu_i)^l}-\int \frac{1}{(z-y)^l}p_{\MP}(y)dy$ as the sum
	\beq\begin{split}
	\quad S_1+S_2
	=\left(\frac1n\sum_{i=1}^K\frac1{(z-\mu_i)^l}-\int_{g_K}^{d_+} \frac{p_{\MP}(y)\dd y}{(z-y)^l}\right)
	+\left(\frac1n\sum_{i=K+1}^n\frac1{(z-\mu_j)^l}-\int_{d_-}^{g_K} \frac{p_{\MP}(y)\dd y}{(z-y)^l}\right).
	\end{split}\eeq
	For $i\leq K$, we observe that:
	\begin{itemize}
	\item On the event $\event$, $n^{2/3}(d_+-\mu_i)$ is uniformly bounded in $i$. Thus $|z-\mu_i|\geq|z-d_+|-|d_+-\mu_i|>\frac12|z-d_+|$ by the assumption on $z$.
	\item As $\re(z)>d_+$, we have $|z-y|\geq|z-d_+|$ for all real $y<d_+$.
	\end{itemize}		
Therefore,
		\beq
		\mathbbm{1}_{\event}|S_1|\leq \frac1n\sum_{i=1}^K\frac1{|z-\mu_i|^l}+\int_{g_K}^{d_+} \frac{1}{|z-y|^l}p_{\MP}(y)\dd y\leq \frac{3K}{n|z-d_+|^l}.
		\eeq
		We then bound $\mathbbm{1}_{\event}|S_2|$ using Lemma \ref{lem:diff_stieltjes_gK} to complete the proof of Lemma \ref{lem:diff_stieltjes}.
		\end{proof}
We obtain an upper bound for $\gamma-\gt$ in the following lemma. Together with Lemma \ref{lem:gt_d+}, it verifies that the order of $\gamma-\gt$ is strictly less than that of $\gt -d_+$.
\begin{lemma}
If $b<0$, then, on the event $\event$ for any given $\e>0$,
\beqq
|\g-\gt|=O\left(\frac{(\log\log n)^2}{n^{2/3}\sqrt{\log n}}\right).
\eeqq
\end{lemma}
\begin{proof}
	Recall that $\g$ and $\gt$ are solutions to the equations $L(x)=R(x)$ and $L_\infty(x)=R_\infty(x)$, respectively, where
	\beqq
	L(x)=\frac1n\sum_{i=1}^n\frac{1}{x-\mu_i(n)},\qquad
	R(x)=\frac{B_n^2}{\alpha_n+\sqrt{\alpha_n^2+xB_n^2}}
	\eeqq
	and
	\beqq
	L_\infty(x)=\int_\mathbb{R}\frac{p_{\MP}(y)\dd y}{x-y},\qquad
	R_\infty(x)=\frac{B^2}{\alpha+\sqrt{\alpha^2+xB^2}}.
	\eeqq
	Define $F(x)=R(x)/L(x)$ and let $F_\infty(x)$ be given similarly. Setting $\e_n=\frac{(\log\log n)^2}{n^{2/3}\sqrt{\log n}}$, we follow the method in \cite{BaikLeeBipartite} to prove $|\g-\gt |=O(\e_n)$ by showing $F(\gt -\e_n)<1<F(\gt +\e_n).$
	Since $F_\infty(\gt )=1$ and $F_\infty(\gt -\e_n)<1<F_\infty(\gt +\e_n)$, it suffices to show
	\beq\label{eq:Fclaim}
	|F(x)-F_\infty(x)|\ll |F_\infty'(\gt )|\e_n, \quad \text{for } x\in[\gt -\e_n,\gt +\e_n].
	\eeq
	Thus, we need a lower bound for $|F_\infty'(\gt )|$ and an upper bound for $|F(x)-F_\infty(x)|$. For the lower bound, begin with 
	\beqq
	F_\infty'(\gt )=\frac{R_\infty'(\gt )L_\infty(\gt )-L_\infty'(\gt )R_\infty(\gt )}{(L_\infty(\gt ))^2}.
	\eeqq
	Note that $L_\infty(\gt )$ and $R_\infty(\gt )$ are of order 1, and $R_\infty'(\gt )=O(1)$ using the fact that $\a,B,\gt $ are all of order 1.  
	We now show $|L_\infty'(\gt )|$ is of order at least $n^{1/3}(\log n)^{-1/2}$, which implies $|F_\infty'(\gt )|$ is as well. 
	
	Since we are interested in $L'_\infty(x)$ at $\gt$, where $\gt-d_+=\Theta(n^{-2/3}\log n)$ by Lemma \ref{lem:gt_d+}, we consider $L_\infty(d_++s)$ as a function of $s$ and its derivative, and later set $s$ to take value of order $n^{-2/3}\log n$. We have
	\beqq
	L_\infty(d_++s)
	=C\int_{d_-}^{d_+}\frac{\sqrt{(d_+-y)(y-d_-)}}{(d_++s-y)y}\dd y=\int_0^{d_+-d_-}\frac{C\sqrt{z(d_+-d_--z)}}{(z+s)(d_+-z)}\dd z,
	\eeqq
	where $C=\frac2{\pi}(\sqrt{d_+}-\sqrt{d_-})^{-2}$ and $z=d_+-y$.
	We then write
	\beq\label{eq:Linftyderivsplit}\begin{split}
		L_\infty'(d_++s)
		=\frac{d}{ds}\int_0^{\frac{d_+-d_-}{2}}\frac{C\sqrt{z(d_+-d_--z)}}{(z+s)(d_+-z)}\dd z
		+\frac{\dd}{\dd s}\int_{\frac{d_+-d_-}{2}}^{d_+-d_-}\frac{C\sqrt{z(d_+-d_--z)}}{(z+s)(d_+-z)}\dd z.
	\end{split}\eeq
	First, we consider the derivative of a simplified version of the first integral:
	\beq\begin{split}
		\frac{\dd}{\dd s}\int_0^{\frac{d_+-d_-}{2}}\frac{\sqrt{z}}{z+s}dz
		&=-s^{-1/2}\arctan\sqrt{\tfrac{(d_+-d_-)/2}{s}}
		+\sqrt{s}\frac{\sqrt{(d_+-d_-)/2}}{s^{3/2}}\frac{1}{1+\frac{(d_+-d_-)/2}{s}}\\
		&=-s^{-1/2}+O(1).
	\end{split}\eeq
	Now that we have the derivative of this simplified integral, recall that the actual integrand is $\frac{C\sqrt{z(d_+-d_--z)}}{(z+s)(d_+-z)}$ and make the following observations:
	\begin{itemize}
		\item For $z\in[0,\frac{d_+-d_-}{2}]$, there exist positive constants $C_1,C_2$ such that $C_1<\frac{C\sqrt{d_+-d_--z}}{d_+-z}<C_2$.
		\item For any $z>0$, the quantity $\frac{\sqrt{z}}{z+s}$ is a decreasing function of $s$ when $s>0$.
	\end{itemize}
	From these two facts and the above computation, we conclude that, for small $s$,
	\beq
	-C_2s^{-1/2}\leq
	\frac{\dd}{\dd s}\int_0^{\frac{d_+-d_-}{2}}\frac{C\sqrt{z(d_+-d_--z)}}{(z+s)(d_+-z)}\dd z
	\leq-C_1s^{-1/2}.
	\eeq
	
	Finally, the second bullet point implies the second integral on the right side of \eqref{eq:Linftyderivsplit} must be negative. Thus $L_\infty'(d_++s)<-C_1s^{-1/2}$, which implies $|L_\infty'(\gt )|$ is of order at least $n^{1/3}(\log n)^{-1/2}$. We obtain the lower bound
	\beq
	|F_\infty'(\gt )|=\Omega(n^{1/3}(\log n)^{-1/2}).
	\eeq
	
	We now show an upper bound of $|F(x)-F_\infty(x)|$ for $x\in[\gt -\e_n,\gt +\e_n]$. For such $x$, 
	\beqq
	F(x)-F_\infty(x)=\frac{(R(x)-R_\infty(x))L_\infty(x)+(L_\infty(x)-L(x))R_\infty(x)}{L(x)L_\infty(x)}
	\eeqq
	satisfies that the denominator, $L_\infty(x)$, and $R_\infty(x)$ all have order 1. Thus, it remains to bound the terms $R(x)-R_\infty(x)$ and $L_\infty(x)-L(x)$. As $\a_n-\a=O(n^{-1-\delta})$ and $B_n-B=O(n^{-1-\delta})$, we have
	\beqq
	R(x)-R_\infty(x)=\frac{\sqrt{\alpha_n^2+xB_n^2}-\alpha_n}{x}-\frac{\sqrt{\alpha^2+xB^2}-\alpha}{x}
	=O(n^{-1-\delta}).
	\eeqq
	Lastly, Lemma \ref{lem:diff_stieltjes} yields that 
	\beqq
	L(x)-L_\infty(x)=\frac1n \sum_{i=1}^n\frac{1}{x-\mu_i}-\int\frac{p_{\MP}(y)\dd y}{x-y}=O(n^{-1/3}(\log\log n)(\log n)^{-1}).
	\eeqq
	Thus, we have shown that for $x\in[\gt -\frac{(\log\log n)^2}{n^{2/3}\log n},\gt +\frac{(\log\log n)^2}{n^{2/3}\log n}]$,
	\beq
	|F(x)-F_\infty(x)|=O\left(\frac{\log\log n}{n^{1/3}\log n}\right),\qquad
	|F'_\infty(\gt)|=\Omega\left(n^{1/3}(\log n)^{-1/2}\right).
	\eeq
	This verifies the inequality \eqref{eq:Fclaim}, and the lemma follows.
\end{proof}
We now introduce a deterministic approximation $G_\infty$ of the function $G$, given by
\beq
G_\infty(z_1,z_2)=B(z_1+z_2)-\a \log z_1-\frac12 \int\log(4z_1z_2-x)p_{\MP}(x)\dd x.
\eeq
We observe that $(\gt_1,\gt_2)$ is the unique critical point of $G_\infty$ satisfying $4\gt_1\gt_2 \in (d_+, \infty)$. This follows from the similar reasoning to what we used for $(\g_1,\g_2)$. We obtain the following asymptotic expressions for the functions $G$, $G_\infty$ and their partial derivatives.   

\begin{lemma}\label{lem:Gpartials}\label{lem:Gpartials1}
Let $(z_1,z_2)$ satisfy $\re(4z_1z_2)\geq d_+$ and $|4z_1z_2-d_+|\geq cn^{-2/3}\log n$ for some fixed $c>0$.  Then, on the event $\event$, the following hold and are uniform in any compact region satisfying the constraints on $(z_1,z_2)$:
\begin{enumerate}[(i)]
\item\label{item:Gpartials1} For every multi-index $k=(k_1,k_2)$ (with $|k|:=k_1+k_2\geq1$), 
\beq
\partial^kG(z_1,z_2)-\partial G_\infty^k(z_1,z_2)=O\left(n^{\frac23 |k|-1}\frac{\log\log n}{(\log n)^{|k|}}
\right).
\eeq
\item\label{item:Gpartials2} For every multi-index $k$ with $|k|\geq1$, 
\beq\begin{split}
\partial^kG_\infty(z_1,z_2)&=O(n^{\frac23 |k|-1}(\log n)^{-|k|+\frac32})\\
\partial^kG(z_1,z_2)&=O(n^{\frac23 |k|-1}(\log n)^{-|k|+\frac32}).
\end{split}\eeq
\end{enumerate}
\end{lemma}

\begin{proof}
We recall
\beqq\begin{split}
G(z_1,z_2)&=B_n(z_1+z_2)-\a_n \log z_1-\frac1{2n}\sum_{j=1}^n\log(4z_1z_2-\mu_j),\\
G_\infty(z_1,z_2)&=B(z_1+z_2)-\a \log z_1-\frac12 \int\log(4z_1z_2-x)p_{\MP}(x)\dd x.\\
\end{split}\eeqq
Observe that over any fixed compact region of $\bC^2$, for every $|k|\geq1$,
\begin{itemize}
\item $\partial^kG_\infty(z_1,z_2)=O\left(\int(4z_1z_2-x)^{-k}p_{\MP}(x)\dd x\right)$, and	
\item the differences in the partials of $G$ and $G_\infty$ satisfy
\beq\label{eq:partialdiff}
\partial^kG(z_1,z_2)-\partial^kG_\infty(z_1,z_2)=O\left(\frac1n\sum_{i=1}^n\frac{1}{(4z_1z_2-\mu_i)^{|k|}}-\int\frac{p_{\MP}(x)\dd x}{(4z_1z_2-x)^{|k|}}
\right).
\eeq
\end{itemize}
Applying Lemma \ref{lem:diff_stieltjes} to \eqref{eq:partialdiff} gives us part \eqref{item:Gpartials1} of the lemma.  For part \eqref{item:Gpartials2}, we first obtain the bound for $\partial^kG_\infty$ by noting that
\beq\begin{split}
\left|\int(4z_1z_2-x)^{-|k|}p_{\MP}(x)\dd x\right|
&\leq \int\frac{1}{\max\{|4z_1z_2-x|,\;d_+-x\}^{|k|}}p_{\MP}(x)\dd x\\
&=O\left(\int_{n^{-2/3}\log n}^\infty\frac{\sqrt{y-n^{-2/3}\log n}}{y^{|k|}}\dd y\right)\\
&=O\left(\int_{n^{-2/3}\log n}^\infty y^{-|k|+\frac12}\dd y\right)
=O\left((n^{-2/3}\log n)^{-|k|+3/2}\right).
\end{split}\eeq
Then, the bound for $\partial^kG$ as in \eqref{item:Gpartials2} follows by part \eqref{item:Gpartials1} of the lemma and the bound obtained for $\partial^kG_\infty$.
\end{proof}

We prove some further properties of $G$ and $G_\infty$ in the following lemma.

\begin{lemma}\label{lem:Gcrit_hightemp}\label{lem:Gpartials2}
For the critical points $(\g_1,\g_2)$ and $(\gt_1,\gt_2)$ of $G$ and $G_\infty$, respectively, the following hold on event $\event$.
\begin{enumerate}[(i)]
\item\label{item:Gcrit_hightemp1} We have
\beqq
|\g_1-\gt_1|=O(n^{-2/3}(\log\log n)^2(\log n)^{-1/2}),\quad
|\g_2-\gt_2|=O(n^{-2/3}(\log\log n)^2(\log n)^{-1/2}).
\eeqq
\item\label{item:Gcrit_hightemp2} There is a positive constant $c$, independent of $n$, such that
\beqq
4\g_1\g_2-\mu_1>cn^{-2/3}\log n\quad 
4\g_1\g_2-d_+>cn^{-2/3}\log n.
\eeqq
\item\label{item:Gcrit_hightemp3} We have
\beqq
G(\g_1,\g_2)=G(\gt_1,\gt_2)+O(n^{-1}(\log n)^{-3/2}(\log\log n)^4)
\eeqq
and for and multi-index $k=(k_1,k_2)$ satisfying $|k|>0$,
\beqq
\partial^kG(\g_1,\g_2)=\partial^kG(\gt_1,\gt_2)+O\left(n^{\frac23 |k|-1}(\log n)^{-|k|}(\log\log n)^2
\right).
\eeqq
\end{enumerate}
\end{lemma}

\begin{proof}
Part \eqref{item:Gcrit_hightemp1} follows from the equations for $\g_1,\g_2,\gt_1,\gt_2$ along with the bound on $|\g-\gt|$.

Part \eqref{item:Gcrit_hightemp2} follows from part \eqref{item:Gcrit_hightemp1} along with the computation of $\gt-d_+$ and the fact that $|d_+-\mu_1|=O(n^{-2/3})$.

For Part \eqref{item:Gcrit_hightemp3}, using the bounds from Lemma \ref{lem:Gpartials1}\eqref{item:Gpartials2} and Lemma \ref{lem:Gpartials2}\eqref{item:Gcrit_hightemp1}, we get the Taylor expansion
\beqq\begin{split}
G(\gt_1,\gt_2)&=G(\g_1,\g_2)+\partial_1G(\g_1,\g_2)(\gt_1-\g_1) 
+\partial_2G(\g_1,\g_2)(\gt_2-\g_2)+O\left(\tfrac{n^{1/3}}{(\log n)^{1/2}}\cdot (\tfrac{(\log\log n)^2}{n^{2/3}(\log n)^{1/2}})^2\right)\\
&=G(\gt_1,\gt_2)+O(n^{-1}(\log n)^{-3/2}(\log\log n)^4).
\end{split}\eeqq
Similarly, for the partials, we get
\beqq\begin{split}
\partial^kG(\gt_1,\gt_2)&=\partial^kG(\g_1,\g_2)+O\left(n^{\frac23 (|k|+1)-1}(\log n)^{-(|k|+1)+\frac32 }\cdot
\tfrac{(\log\log n)^2}{n^{2/3}(\log n)^{1/2}}\right)\\
&=\partial^kG(\g_1,\g_2)+O\left(n^{\frac23 |k|-1}(\log n)^{-|k|}(\log\log n)^2
\right).
\end{split}\eeqq
\end{proof}

\subsection{Steepest descent analysis}
We now perform steepest analysis to compute the contour integral in the high temperature case. The method relies on the observation that the dominant contribution to the integral comes from within a small radius around the critical point of $G$. In this case, the radius is $r=n^{-2/3}(\log n)^{\frac14+\e}$ for some $\e>0$.

The intuition behind this choice of truncation radius is as follows:  Consider a Taylor expansion of $G_\infty$ where $z_1=\tg_1+it_1/r_n$ and $z_2=\tg_2+it_2/r_n$ with $r_n$ to be determined.  Let $m$ denote a multiindex for the derivative and let $|m|$ denote the length the multi-index.  We want to choose $r_n$ such that
\beq
\partial^{(m)}G_\infty(\tg_1,\tg_2)\cdot r_n^{|m|}=\begin{cases}\Theta( \frac1n)&|m|=2\\
o(\frac1n)&|m|\geq3.
\end{cases}
\eeq
Using the previous lemmas, this is satisfied exactly when $r=\Theta( n^{-2/3}(\log n)^{1/4})$.

\begin{lemma}\label{lem:tail_int}
	Let $\g_1=\g_1(n)$ and $\g_2=\g_2(n)$ be such that $(\g_1,\g_2)$ is the critical point of $G(z_1,z_2)$ satisfying $\g=4\g_1\g_2>\mu_1(n)$. Then, for any $0<\e<1/4$ and any $\Omega\subset \{(y_1,y_2)\in\bR^2:y_1^2+y_2^2\geq n^{-4/3}(\log n)^{1/2+2\e} \}$, on the event $\event$, there exists some $C>0$ such that 
	\[
	\int_{\Omega}\exp\left[n\re(G(\g_1+\ii y_1,\g_2+\ii y_2)-G(\g_1,\g_2)) \right]\dd y_2\dd y_1 =O(e^{-C(\log n)^\e}).
	\]
\end{lemma}
	\begin{proof}
	Since $\g-\mu_n$ is bounded in $n$, Lemma 3.9 of \cite{BaikLeeBipartite} implies that with high probability, the portion of the above integral over $\Omega\cap \{(y_1,y_2)\in\bR^2:y_1^2+y_2^2\geq n^{-1+2\e}\}$ is $O(e^{-n^\e})$. Thus, it remains to consider the subset of $\Omega$ where $y_1^2+y_2^2$ is between $n^{-4/3}(\log n)^{1/2+2\e}$ and $n^{-1+2\e}$. We denote this subset by $\widetilde{\Omega}$.
	
	The proof of Lemma 3.9 of \cite{BaikLeeBipartite} also shows that, for some constant $c_0>0$ and for any integer $K\geq 1$, 
	\begin{equation}\label{eqn:re_diff}
		\re(G(\g_1+ \ii y_1,\g_2+\ii y_2)-G(\g_1,\g_2)) \leq -\frac{1}{4n}\sum_{j=K}^{n}\log\left(1+\frac{c_0}{(\g-\mu_j)^2}(y_1^2+y_2^2)\right),
	\end{equation}
	for all $y_1,y_2\in \bR$. By Lemma \ref{lem:Aj}, for every $\e>0$, there exists $c,K>0$ such that, with probability at least $1-\e$,
	\[
	\g-d_+\leq cn^{-2/3}\log n  \quad \text{and}\quad d_+-\mu_j\leq \begin{cases}
		ci^{2/3}n^{-2/3}, &\quad K\leq j\leq n^{2/5},	\\
		c, &\quad j>n^{2/5}.
	\end{cases}
	\]
	Thus, with probability at least $1-\e$,
	\[
	\g-\mu_j \leq \begin{cases}
		cn^{-2/3}\log n, &\quad K \leq j \leq (\log n)^{3/2},\\
		cj^{2/3}n^{-2/3}, &\quad (\log n)^{3/2}\leq j\leq n^{2/5},\\
		c, &\quad j>n^{2/5}.
	\end{cases}
	\]
	Write $r^2=y_1^2+y_2^2$ using polar coordinates, then for $r\in[n^{-2/3}\log^{1/4+\e}n, n^{-1/2+\e}]$ and the above choice of $K$, the right hand side of \eqref{eqn:re_diff} has upper bound 
	\begin{align}
	-\frac1{4n}\left[(\log n)^{3/2}\log\left(1+\frac{c'n^{4/3}}{\log^2 n}r^2\right)+\sum_{j=(\log n)^{3/2} }^{n^{2/5}}\log\left(1+\frac{c'r^2}{(j/n)^{4/3}}\right)+\frac n2\log(1+c'r^2) \right].
	\end{align}
	We then use $r\geq n^{-2/3}(\log n)^{1/4+\e}$ for the first and last terms inside the brackets, and the fact $\log(1+x)\geq x/2$ for small $x$ to obtain a new bound
	\begin{align}
		-\frac{c'}{4n}\left[(\log n)^{2\e}+(\log n)^{-3/2+2\e} \sum_{j=\log^{3/2}n}^{n^{2/5}}j^{-4/3}+\frac n4 r^2\right]
		&\leq -\frac{c'r^2}{16}-\frac{c'(\log n)^{2\e}}{8n},
	\end{align}
	noting that the sum over $j$ is $O((\log n)^{-1/2})$. Therefore, the integral over $\widetilde{\Omega}$ is bounded by 
	\begin{equation}
		e^{-\frac{c'}{8}(\log n) ^{2\e}} \int_{ n^{-2/3}\log^{1/4+\e}n}^{n^{-1/2+\e}}e^{-\frac{c'}{16}r^2}r\dd r = O(e^{-C(\log n)^{2\e}}),
	\end{equation}
	for some $C>0$. This completes our proof.
\end{proof}

\begin{lemma}\label{lem:Qn}
If $\b=\b_c+bn^{-1/3}\sqrt{\log n}$ for fixed $b<0$, then the integral $Q_n$ in \eqref{eq:Qn} satisfies
\beqq
Q_n=e^{nG(\g_1,\g_2)}\frac{\pi}{n\sqrt{D(\g_1,\g_2)}}\left(1+O((\log n)^{-\frac32+6\e})\right),
\eeqq
where $\e>0$ is arbitrarily small and $D(\g_1,\g_2)$ is the discriminant
\beq
D(\g_1,\g_2):=\partial_1^2G(\g_1,\g_2)\cdot\partial_2^2G(\g_1,\g_2)-(\partial_1\partial_2G(\g_1,\g_2))^2.
\eeq
\end{lemma}
\begin{proof}
We make the change of variables
\beq
z_1=\g_1+\ii r_nt_1,\quad z_2=\g_2+\ii r_nt_2,
\eeq
where the scaling $r_n:=n^{-2/3}(\log n)^{1/4}$ is chosen such that the quadratic term in the Taylor expansion of $G$ near $(\g_1,\g_2)$ will be of order 1. With this change of variable, we have
\beq
Q_n=r_n^2e^{nG(\g_1,\g_2)}\int_{-\infty}^\infty\int_{-\infty}^\infty
\exp\left(n\Big(G(\g_1+\ii r_nt_1,\;\g_2+\ii r_nt_2)-G(\g_1,\g_2)\Big)\right)\dd t_2\dd t_1.
\eeq
Fix $0<\e<1/4$. We have shown in Lemma \ref{lem:tail_int} that this integral outside a region of radius $(\log n)^\e$ around the critical point is $O(e^{-c(\log n)^\e})$ for some constant $c>0$. We now consider the region where $|t_1|,|t_2|\leq (\log n)^\e$.  In this region,
\beq\begin{split}
G&(\g_1+\ii r_nt_1,\g_2+\ii r_nt_2)-G(\g_1,\g_2)\\
&=-\tfrac12 r_n^2\left(\partial_1^2G(\g_1,\g_2)t_1^2 +2\partial_1\partial_2G(\g_1,\g_2)t_1t_2+\partial_1^2G(\g_1,\g_2)t_1^2\right)\\
&\quad-\tfrac{\ii}{6}r_n^3\left(\partial_1^3G(\g_1,\g_2)t_1^3 +3\partial_1^2\partial_2G(\g_1,\g_2)t_1^2t_2 +3\partial_1\partial_2^2G(\g_1,\g_2)t_1t_2^2 +\partial_2^3G(\g_1,\g_2)t_2^3\right)\\
&\quad+O(\text{Taylor remainder})\\
&=:-r_n^2X_2(t_1,t_2)-\ii r_n^3X_3(t_1,t_2)+O(n^{-1}(\log n)^{-\frac32+4\e}).
\end{split}\eeq
Thus, the integral on the central region becomes
\begin{multline}
\int_{-(\log n)^\e}^{(\log n)^\e}\int_{-(\log n)^\e}^{(\log n)^\e}\exp\left(n\Big(G(\g_1+\ii r_nt_1,\;\g_2+\ii r_nt_2)-G(\g_1,\g_2)\Big)\right)\dd t_2\dd t_1\\
=\iint e^{-nr_n^2X_2(t_1,t_2)}\dd t_2\dd t_1 -\ii\iint nr_n^3X_3(t_1,t_2)e^{-nr_n^2X_2(t_1,t_2)}\dd t_2\dd t_1 +O\left((\log n)^{-\frac32+6\e}\right),
\end{multline}
where the second integral vanishes due to the fact that
\beqq
X_3(-t_1,-t_2)e^{-nr_n^2X_2(-t_1,-t_2)}=-X_3(t_1,t_2)e^{-nr_n^2X_2(t_1,t_2)}.
\eeqq
It remains to compute $\int_{-(\log n)^\e}^{(\log n)^\e}\int_{-(\log n)^\e}^{(\log n)^\e}e^{-nr_n^2X_2(t_1,t_2)}\dd t_2\dd t_1$, which we replace by the integral over $\RR^2$, incurring an error on the order of
\beq
\int_{(\log n)^\e}^\infty e^{-x^2}\dd x< e^{-(\log n)^{2\e}}\ll (\log n)^{-3/2}.
\eeq
Finally, applying Gaussian integration, we obtain the lemma.
\end{proof}

We observe from the lemma above that the integral $Q_n$ depends on $G(\g_1,\g_2)$ and $D(\g_1,\g_2)$, which we compute in the following lemma.
	\begin{lemma}\label{lem:D_tildeg}
		If $\beta=\beta_c+bn^{-1/3}\sqrt{\log n}$ for some fixed $b<0$, then
		\begin{align*}
		G(\g_1,\g_2)&=A(\gt,B)-\frac{1}{2n}\sum_{i = 1}^n\log(\tilde{\g}-\mu_i)+O(n^{-1})\\
		D(\g_1,\g_2)
		&=\frac{\beta_c}{\la^2b}n^{1/3}(\log n)^{-1/2}\left(1+O\left((\log\log n)^2(\log n)^{-3/2}\right)\right)
		\end{align*}
		where
		\beq
		A(x,B):=\sqrt{\alpha^2+xB^2}-\alpha\log\left(\frac{\alpha+\sqrt{\alpha^2+xB^2}}{2B}\right).
		\eeq
	\end{lemma}
\begin{proof}	  
The computation of $G(\g_1,\g_2)$ relies upon $G_\infty(\tilde{\g}_1,\tilde{\g}_2)$, which we write as 
		\beq
		G_\infty(\tilde{\g}_1,\tilde{\g}_2)=A(\gt,B)-\frac12H_{\MP}(\tilde{\g}), \quad  \quad  H_{\MP}(z):=\int_{\bR}\log(z-x)p_{\MP}(x)\dd x.
		\eeq
		Then, by Lemma \ref{lem:Gcrit_hightemp}(\ref{item:Gcrit_hightemp3}), 
		\beq\label{eq:Ggamma_comp}\begin{split}
			G(\g_1,\g_2) 
			&= G_\infty(\tilde{\g}_1,\tilde{\g}_2) + \left[G(\tilde{\g}_1,\tilde{\g}_2) - G_\infty(\tilde{\g}_1,\tilde{\g}_2)\right] + O\left(n^{-1}\frac{(\log\log n)^4}{(\log n)^{3/2}}\right)\\
			&=G_\infty(\tilde{\g}_1,\tilde{\g}_2) -\frac1{2n}\left[\sum_{i=1}^n\log(\tilde{\g}-\mu_i)-nH_{\MP}(\tilde{\g})\right] + O(n^{-1})\\
			&=A(\gt,B)-\frac{1}{2n}\sum_{i = 1}^n\log(\tilde{\g}-\mu_i)+O(n^{-1}).
		\end{split}\eeq
The same lemma and Lemma \ref{lem:Gpartials}\eqref{item:Gpartials2} together yield
		\begin{align*}
		D(\g_1,\g_2) 
		=D_\infty(\tilde{\g}_1,\tilde{\g}_2)+O\left(n^{1/3}\frac{(\log\log n)^2}{(\log n)^2}\right).
		\end{align*}	
Recall from \eqref{eqn:gt_dplus} that $\tilde{\g}
=\frac{1+\beta^2+\b_c^{-4}\beta^4}{(1+\la)^{-1}\beta^2}$, and $\b_c=\la^{-\frac14}(1+\la)^{1/2}$. We arrive at
\begin{equation}
\begin{split}
D_\infty(\tilde{\g}_1,\tilde{\g}_2)
&:=\partial_1^2G_\infty(\tilde{\g}_1,\tilde{\g}_2)\cdot\partial_2^2G_\infty(\tilde{\g}_1,\tilde{\g}_2)-(\partial_1\partial_2G_\infty(\tilde{\g}_1,\tilde{\g}_2))^2
=\frac{4\beta^4}{\la^2(\b_c^4-\beta^4)}.
\end{split}			
\end{equation}
Apply this to the expression $D(\g_1,\g_2)$ and perform Taylor expansion around $\b_c$, we obtain the lemma.		
\end{proof}

\subsection{High temperature free energy}
Finally, using the contour integral computations from the previous section, we obtain the following lemma for the limiting fluctuations of the free energy on the high temperature side of the critical temperature window.
\begin{lemma}\label{lem:Fn_high}
Suppose $\beta=\beta_c+bn^{-1/3}\sqrt{\log n}$ for some fixed $b<0$. We define $F(\beta)=\frac{\beta^2}{2\beta_c^4}$. Then the free energy satisfies 
\begin{equation}
\frac{m+n}{\sqrt{\frac16\log n}}\left(F_{n,m}(\beta)-F(\beta)+\frac{1}{12}\frac{\log n}{n+m}\right) \to \mathcal{N}(0,1).
\end{equation}	
\end{lemma}
	
\begin{proof}
We will show that 
\beq\label{eq:toprovehightempFE}
F_{n,m}(\beta)-\frac{\beta^2}{2\beta_c^4}+\frac{1}{12}\frac{\log n}{n+m}-\frac{\sqrt{\frac16 \log n}}{m+n}T_{0n}=O\left(\frac{\log\log n}{n}\right),
\eeq
where
\begin{equation}\label{eq:T0n_def}
-T_{0n}:=\frac{\sum_{i = 1}^n\log(\tilde{\g}-\mu_i)-C_\la n - \frac{1}{\sqrt{\la}(1+\sqrt{\la})} n(\tilde{\g}-d_+)
+\frac{2}{3\la^{3/4}(1+\sqrt{\la})^2}n(\tilde{\g}-d_+)^{3/2} +\frac16\log n}{\sqrt{\frac23 \log n}}
\end{equation}
with  $C_\la:=(1-\la^{-1})\log(1+\la^{1/2})+\log(\la^{1/2})+\la^{-1/2}$ and, by \cite{CWL_CLT}, $T_{0n}$ converges in distribution to a standard normal.  We now compute the left hand side of \eqref{eq:toprovehightempFE} in terms of the parameters $\beta$ and $\lambda$.
From \eqref{eqn:Fn}, we start by computing
\begin{equation}\label{eqn:logQn}
\frac{1}{n+m}\log Q(n,\alpha_n, B_n)= \frac n{n+m} G(\g_1,\g_2)+\frac1{2(n+m)}\log\left(\frac{\pi^2}{D(\g_1,\g_2)}\right)-\frac{\log n}{n+m}+o(n^{-1}),
\end{equation}	
using Lemma \ref{lem:Qn}. By Lemma \ref{lem:D_tildeg}, the second term satisfies
\begin{equation}
	\frac1{2(n+m)}\log\left(\frac{\pi^2}{D(\g_1,\g_2)}\right)
	=-\frac16\frac{\log n}{n+m}+O(\frac{\log\log n}{n}).
\end{equation}	
Thus, using the computation of $G(\g_1,\g_2)$ from \eqref{eq:Ggamma_comp}, \eqref{eqn:logQn} simplifies to 
\begin{equation}\label{eqn:logQn2}
	\frac{1}{n+m}\log Q(n,\alpha_n, B_n)=\frac{n}{n+m}A(\gt,B)-\frac{1}{2(n+m)}\sum_{i = 1}^n\log(\tilde{\g}-\mu_i)-\frac76\frac{\log n}{n+m}+O(\frac{\log\log n}{n}).
\end{equation}	
Recall that $\alpha=\frac12(\la^{-1}-1)$ and $B=\frac{1}{\sqrt{\la(1+\la)}}\beta$ for the bipartite SSK model, and $\gt$ is given in \eqref{eqn:gt_dplus}. This implies $\sqrt{\alpha^2+\tilde{\g}B^2}=\frac{\la+1}{2\la}+\frac{\b^2}{1+\la}$, and 
\begin{equation}\label{eqn:A1}
\frac{n}{n+m}A(\gt,B)=\frac12+\frac{\la\beta^2}{(1+\la)^2}+\frac{1-\la}{2(1+\la)}\log\left(\frac{2\beta\sqrt{\la(1+\la)}}{1+\la+\beta^2\la}\right)+O(n^{-1}).
\end{equation}
Combining \eqref{eqn:Fn}, \eqref{eqn:logQn2} and \eqref{eqn:A1}, we have
\begin{equation}\label{eqn:Fn_high}
	\begin{split}
	F_{n,m}(\beta)
	=& -\frac{1}{2(n+m)}\sum_{i = 1}^n\log(\tilde{\g}-\mu_i)
	+\frac{\la\beta^2}{(1+\la)^2}-\frac{1-\la}{2(1+\la)}\log(1+\la+\beta^2\la)\\
	&-\frac{\la}{1+\la}\log\beta+\frac{1}{2(\la+1)}\log(1+\la)
	 -\frac16\frac{\log n}{n+m}+O\left(\frac{\log\log n}{n}\right).
	\end{split}	
\end{equation}
In order to prove equation \eqref{eq:toprovehightempFE}, we need express each $\beta$-dependent term as a Taylor expansion around $\beta_c$.  More specifically, we define
\beq
\Delta_\b:=\b_c-\b=O(n^{-1/3}\sqrt{\log n}).
\eeq
Using this and the fact that $\beta_c=\frac{\sqrt{1+\la}}{\la^{1/4}}$, we get
\beq\begin{split}
\beta^2=&\frac{1+\la}{\sqrt{\la}}-2\beta_c\Delta_\beta+\Delta_\beta^2\\
\log\beta=&\frac12\log(1+\la)-\frac14\log\la-\frac{1}{\beta_c}\Delta_\beta-\frac{1}{2\beta_c^2}\Delta_\beta^2-\frac{1}{3\beta_c^3}\Delta_\beta^3+O(\Delta_\beta^4)\\
\log(1+\la+\beta^2\la)=&\log((1+\la)(1+\sqrt{\la}))-\frac{2\beta_c\la}{(1+\la)(1+\sqrt{\la})}\Delta_\beta \\
 &+\frac{\la(1+\la+\beta_c^2\la-2\beta_c^2)}{(1+\la)^2(1+\sqrt{\la})^2}\Delta_\beta^2 +\frac{2\beta_c\la^2(1+\la-\frac13\beta_c^2\la)}{(1+\la)^3(1+\sqrt{\la})^3}\Delta_\beta^3+O(\Delta_\beta^4)
\end{split}\eeq
Furthermore, using equation \eqref{eq:gt-d+initialcomp} we have 
\beq
\gt-d_+=\frac{4(1+\la)}{\b_c^4}\Delta_\beta^2+\frac{4(1+\la)}{\b_c^5}\Delta_\beta^3+O(\Delta_\beta^4).
\eeq
Plugging these asymptotics into equations \eqref{eq:T0n_def} and \eqref{eqn:Fn_high},
we verify \eqref{eq:toprovehightempFE}, and the lemma follows.
\end{proof}

\section{Low temperature}\label{sec:lowtemp}
We now determine the asymptotics of the random double integral $Q_n=-\int_{\g_1-\ii\infty}^{\g_1+\ii\infty}\int_{\g_2-\ii\infty}^{\g_2+\ii\infty}e^{nG(z_1,z_2)}\dd z_2\dd z_1$
when $\beta=\beta_c+bn^{-\frac13}\sqrt{\log n}$ for fixed $b\geq0$. 

Recall that in the regime $\beta<\beta_c$, both for fixed $
\beta$ as in \cite{BaikLeeBipartite} and for $\b$ in the previous Section \ref{sec:hightemp}, the critical point $(\g_1,\g_2)$ of the function $G$ is approximated by $(\gt_1,\gt_2)$, the critical point satisfying $4\gt_1\gt_2>d_+$ of a deterministic approximation $G_\infty$ of $G$. In the case $\beta>\beta_c$, a  critical point of $G_\infty$ satisfying this inequality does not exist, and we cannot approximate the product $\g=4\g_1\g_2$ by a deterministic number. In fact, the product $\g$ gets close to the branch point $\mu_1$ from above, which requires more delicate analysis.


We address this issue by focusing on $G$ near the point $(\muf,\mus)$, given by
\begin{equation}\label{def:muf_mus}
	\muf=\frac{\alpha_n+\sqrt{\alpha_n^2+\mu_1B_n^2}}{2B_n}, \quad \mus=\frac{-\alpha_n+\sqrt{\alpha_n^2+\mu_1B_n^2}}{2B_n},	
\end{equation}
instead of $(\g_1,\g_2)$. We see that $4\muf\mus=\mu_1$, and $G(z_1,z_2)$ at $(\muf,\mus)$ is undefined due to the term $\frac1n\log(4z_1z_2-\mu_1)$. However, the non-singular part given below will play an important role.  
\begin{equation}
\Gmu:=B_n(\muf+\mus)-\alpha_n\log \muf-\frac1{2n}\sum_{j=2}^n\log(\mu_1-\mu_j)
\end{equation}	
In our computation of $\hG$ as well as the contour integral, we need to work with sums of the form $\frac1n \sum_{i=2}^n\frac{1}{(\mu_1-\mu_i)^l}$ for $l\geq1$.  More specifically, we need the following lemma.
\begin{lemma}\label{lem:sum_from_index2}
For LOE eigenvalues, on the event $\event$, we have
\beqq
\frac1n\sum_{i=2}^n\frac{1}{\mu_1-\mu_i}-\frac{1}{\la^{1/2}(1+\la^{1/2})}=O(n^{-1/3}) \quad \text{and} \quad \frac1n\sum_{i=2}^n\frac{1}{(\mu_1-\mu_i)^l}=O(n^{\frac23 l-1}), \quad \text{for }l\geq 2.
\eeqq
\end{lemma}
\begin{proof}
 It suffices to prove the following statements:
\begin{enumerate}[(i)]
\item\label{item:specialsums_step1} For any $l\geq1$, on the event $\event$, 
\beq
\left|\frac1n\sum_{i=k}^n\frac{1}{(\mu_1-\mu_i)^l}-\int_{d_-}^{g_k}\frac{p_{\MP}(x)}{(d_+-x)^l}\dd x \right|=O(n^{\frac23 l-1}).
\eeq
\item\label{item:specialsums_step2} For any $l\geq1$ and any fixed $k$, on the event $\event$,
\beq
\frac1n\sum_{i=2}^k\frac{1}{(\mu_1-\mu_i)^l}=O(n^{\frac23 l-1}).
\eeq
\item\label{item:specialsums_step3} For the $l=1$ case, 
\beq
\int_{g_k}^{d_+}\frac{p_{\MP}(x)}{d_+-x}\dd x=O(n^{-\frac13}).
\eeq
\item\label{item:specialsums_step4} For the $l\geq2$ case,
\beq
\int_{d_-}^{g_k}\frac{p_{\MP}(x)}{(d_+-x)^l}\dd x=O(n^{\frac23 l-1})
\eeq
\end{enumerate}
Verifying \eqref{item:specialsums_step2} is straightforward after imposing the assumption $\mu_1-\mu_i>cn^{-2/3}$ for some $c>0$, which follows from event $\Ffour$. Statements \eqref{item:specialsums_step3} and \eqref{item:specialsums_step4} follow from the definitions of $p_{\MP}$ and $g_k$. 

We now turn to \eqref{item:specialsums_step1}. It follows from Lemma \ref{lem:diff_stieltjes_gK} that, on the event $\Ftwo$,
 \beqq
 \frac1n\sum_{i=K}^n\frac{1}{(d_+-\mu_i)^l}-\int_{d_-}^{g_K}\frac{p_{\MP}(x)}{(d_+-x)^l}\dd x=O(n^{\frac23 l-1}).
 \eeqq
Thus it remains only to show that
\beq\label{eq:stieltjes_d_vs_mu}
\frac1n\sum_{i=K}^n\left(\frac{1}{(\mu_1-\mu_i)^l}-\frac{1}{(d_+-\mu_i)^l}\right)=O(n^{\frac23 l-1}).
\eeq
This bound holds on the event $\Ftwo\cap\Fthree$, which can be seen by observing that
\beqq\begin{split}
\left|\frac{1}{(\mu_1-\mu_i)^l}-\frac{1}{(d_+-\mu_i)^l}\right|
&=\left|\frac{(d_+-\mu_1)\sum_{j=0}^{l-1}(d_+-\mu_i)^j(\mu_1-\mu_i)^{l-j-1}}{(\mu_1-\mu_i)^l(d_+-\mu_i)^l}\right|
\leq \frac{l|d_+-\mu_1|}{\min\{|d_+-\mu_i|,\;|\mu_1-\mu_i|\}^{l+1}},
\end{split}\eeqq
and thus
\beqq\begin{split}
\left|\frac1n\sum_{i=K}^n\left(\frac{1}{(\mu_1-\mu_i)^l}-\frac{1}{(d_+-\mu_i)^l}\right)\right|
&=O\left( \frac1n\sum_{i=K}^n\frac{l\cdot n^{-2/3}}{(d_+-\mu_i)^{l+1}}\right)
\\&
=O\left(n^{-5/3}\int_K^n l\left(\frac{x}{n}\right)^{-\frac23(l+1)}\dd x\right)
=O(n^{\frac23 l-1}).
\end{split}\eeqq
\end{proof}

\subsection{Computation of $\hG(\muone,\mutwo)$}
\begin{lemma}\label{lem:Ghat}
\begin{equation*}
\hG=A(d_+,B)-\frac{\log n}{3n}-\frac{1}{2n}\sum_{i=1}^n\log|d_+-\mu_i|+\frac{bn^{-\frac13}\sqrt{\log n}}{\la^{\frac14}(1+\la)^{\frac12}d_+}(\mu_1-d_+)+O(n^{-1}),
\end{equation*}
where
\begin{equation}\label{def:c0_c1}
A(x,B):=\sqrt{\alpha^2+xB^2}-\alpha\log\left(\frac{\alpha+\sqrt{\alpha^2+xB^2}}{B}\right).
\end{equation}
\end{lemma}
\begin{remark}
The expression of $\hG$ given by Lemma \ref{lem:Ghat} contains two distinct random variables, $\sum_{i=1}^{n}\log|d_+-\mu_i|$ and $\mu_1-d_+$. Under appropriate translation and scaling, they are the quantities that give rise to the Gaussian and Tracy--Widom terms, respectively, in the convergence of free energy as stated in Theorem \ref{thm:main}. The translation and scaling needed for these two random variables are, respectively, $T_{1n}$ and $T_{2n}$, given by
\beq\label{eqn:T1nT2n}
T_{1n}=\frac{C_\la n-\frac16\log n-\sum_{i=1}^{n}\log|d_+-\mu_i|}{\sqrt{\frac23\log n}}, \quad T_{2n}=\frac{n^{2/3}(\mu_1-d_+)}{\sqrt{\la}(1+\sqrt{\la})^{4/3}},
\eeq
where $C_\la$ is as in \eqref{eq:theoremquantities}. The expression of $\hG$ then reads
\beq\label{eqn:hG_T1nT2n}
\hG=A(d_+,B)-\frac12C_\la-\frac{\log n}{4n}+\left(\frac{1}{\sqrt{6}}T_{1n}+\frac{\la^{\frac14}b}{(1+\la^{\frac12})^{\frac23}(1+\la)^{\frac12}}T_{2n}\right)\frac{\sqrt{\log n}}{n}+O(n^{-1}).
\eeq

\end{remark}

\begin{proof}[Proof of Lemma \ref{lem:Ghat}]
By definition, 
\beq\label{eq:Ghat_def_recalled}\begin{split}
\hG&=B_n(\muone+\mutwo)-\a_n\log(\muone)-\frac{1}{2n}\sum_{i=2}^n\log(\mu_1-\mu_i)\\
&=\sqrt{\a_n^2+\mu_1B_n^2}-\a_n\log\left(\frac{\a_n+\sqrt{\a_n^2+\mu_1B_n^2}}{2B_n} \right)-\frac{1}{2n}\sum_{i=2}^n\log(\mu_1-\mu_i).
\end{split}\eeq
Replacing $\a_n,B_n$ by $\a,B$, respectively (incurring an error of $n^{-1-\d}$) and applying Taylor expansion with respect to $\mu_1$ near $d_+$, we obtain
\beq\begin{split}
\sqrt{\a_n^2+\mu_1B_n^2}-\a_n\log\left(\frac{\a_n+\sqrt{\a_n^2+\mu_1B_n^2}}{2B_n} \right)
&=A(d_+,B)
+\frac{B^2(\mu_1-d_+)}{2(\a+\sqrt{\a^2+d_+B^2})}
+O(n^{-1-\d}).
\end{split}
\eeq
Note we have dropped the quadratic term in the Taylor expansion, which is $O(n^{-4/3})$.  It remains to compute the summation in \eqref{eq:Ghat_def_recalled}, which can be rewritten as
\beq
\sum_{i=2}^n\log(\mu_1-\mu_i)=\sum_{i=2}^n\log|d_+-\mu_i|-\frac{n(d_+-\mu_1)}{\lambda^{\frac12}(1+\la^{\frac12})}+E_1+E_2
\eeq
where we define
\beq
E_1=n(d_+-\mu_1)\left(\frac{1}{\lambda^{\frac12}(1+\la^{\frac12})}-\frac1n\sum_{i=2}^n\frac{1}{\mu_1-\mu_i}
\right),\quad E_2=\sum_{i=2}^n\left(\frac{d_+-\mu_1}{\mu_1-\mu_i}-\log\left|1+\frac{d_+-\mu_1}{\mu_1-\mu_i}\right|\right).
\eeq
We now show $E_1+E_2=O(1)$, following an argument similar to that of Johnstone et al in \cite{JKOP2}.  The bound $E_1=O(1)$ follows from Lemma \ref{lem:sum_from_index2}. To bound $E_2$, observe that, on the event we are considering, there exist $k, C$ such that $\mu_1\leq d_++Cn^{-2/3}$ and $\mu_k\leq d_+-Cn^{-2/3}$. For any fixed $i$, we also have $d_+-\mu_1=\Theta(n^{-2/3})$ and $\mu_1-\mu_i=\Theta(n^{-2/3})$. This implies 
\beqq
\sum_{i=2}^{k-1}\left(\frac{d_+-\mu_1}{\mu_1-\mu_i}-\log\left|1+\frac{d_+-\mu_1}{\mu_1-\mu_i}\right|\right)=O(1).
\eeqq
To bound the sum over the indices above $k$, we observe that, for $i\geq k$, we have $\frac{d_+-\mu_1}{\mu_1-\mu_i}\geq-\frac12$ and, for any $x\geq-\frac12$, there is $C_1$ such that $|\log(1+x)-x|\leq C_1x^2$.  This gives us
\beqq
\sum_{i=k}^n\left(\frac{d_+-\mu_1}{\mu_1-\mu_i}-\log\left|1+\frac{d_+-\mu_1}{\mu_1-\mu_i}\right|\right)=O(1).
\eeqq
Finally, combining the results above, and observing that $\frac{1}{2n}\log|d_+-\mu_1|=-\frac{\log n}{3n}+O(n^{-1})$, we get
\beq\label{eqn:hG_c2}
\hG(\muone,\mutwo)=A(d_+,B)-\frac{\log n}{3n}-\frac{1}{2n}\sum_{i=1}^n\log|d_+-\mu_i|+c_2(B)(\mu_1-d_+)+O(n^{-1}),\eeq
where 
\beq c_2(B)=\frac{B^2}{2(\a+\sqrt{\a^2+d_+B^2})}-\frac{1}{2\la^{1/2}(1+\la^{1/2})}.
\eeq
Recall that $B_c$ is defined to be the quantity satisfying
\beq
\frac{\sqrt{\a^2+d_+B_c^2}-\a}{d_+}=\int\frac{p_{\MP}(x)}{d_+-x}\dd x=\frac{1}{\la^{\frac12}(1+\la^{\frac12})}.
\eeq
Using this definition along with a Taylor expansion of $c_2$ near $B=B_c=\la^{-\frac34}$, we get
\beqq\begin{split}
c_2(B)&=\frac{B_c}{2\sqrt{\a^2+d_+B_c^2}}(B-B_c)+O((B-B_c)^2)\\
&=\frac{bn^{-\frac13}\sqrt{\log n}}{\la^{\frac14}(1+\la)^{\frac12}d_+}+O(n^{-2/3}\log n).
\end{split}\eeqq
Apply this to \eqref{eqn:hG_c2}, we obtain the lemma.
\end{proof}

\subsection{Contour integral analysis}
We now derive the asymptotics of the rescaled double integral 
\begin{equation}\label{def:Sn}
	S_n:=\exp(-n\Gmu)Q_n=\int_{-\infty}^\infty \int_{-\infty}^\infty \exp[n(G(\g_1+\ii y_1,\g_2+\ii y_2)-\Gmu)]\dd y_2\dd y_1.
\end{equation}
The analysis holds on the following probability event $\cF_\e$ for arbitrarily small $\e>0$.
\begin{lemma}
	For each $\e>0$, there exist positive numbers $r, s, t$ and $C$, depending on $\e$, such that the event $\cF_\e$ given by
	\beqq
	\cF_\e=\left\{\left|\sum_{j=2}^n\frac{1}{n^{\frac23}(\mu_1-\mu_j)}-s_{\MP}(d_+)\right|\leq C\right\}\cap\left\{\sum_{j=2}^n\frac{1}{n^{\frac43}(\mu_1-\mu_j)^2}\leq C\right\} \Fthree \cap \Ffour
	\eeqq
	satisfies $\bP(\cF_\e)>1-\e$. 
\end{lemma}
We note that the definition of $\cF_\e$ is not unique as it depends on the choice of $s,t, r, R$ and $C$. For any given $\e>0$, we fix the values $s,t,r, R, C$ and define $\cF_\e$ accordingly.
\begin{proof}
	First, for some $C>0$, each of the two events that involve $\frac1{n^{2/3}(\mu_1-\mu_j)}$, with this $C$ as upper bound, holds with probability at least $1-\e/4$ by Lemma \ref{lem:sum_from_index2}. Meanwhile, by Lemma \ref{lem:eventprob}, we can find $0<s<t$ and $0<r<R$ such that each of the events $\Fthree$ and $\Ffour$ holds with probability at least $1-\e/4$.
\end{proof}

Since the integral representation of the partition function only requires $\g_1,\g_2>0$ such that $4\g_1\g_2>\mu_1$, we set $\g_1=\muf$ and $\g_2=\mus+n^{-1}$ in the low temperature case. The shift $n^{-1}$ in $\g_2$ is due to the deformation $\hg_2$, given in \eqref{def:gt_2}, that we later apply to  the integral in the $y_2$ variable. The order $n^{-1}$ is needed to cancel out a term of order $n$ of the function in the exponent (see, for example, \eqref{eqn:re_diffG}). Thus, 
\begin{equation}
	S_n=\int_{-\infty}^\infty \int_{-\infty}^\infty \exp[n(G(\muf+\ii y_1,\mus+n^{-1}+\ii y_2)-\Gmu)]\dd y_2\dd y_1.
\end{equation}
In the remainder of the subsection, we prove the following lemma, for fixed $\e>0$ sufficiently small (e.g. $0<\e<\frac{1}{100}$). 
\begin{lemma}\label{lem:Sn}
	On the event $\cF_\e$,
	\beqq
		S_n= \begin{cases}
		e^{O(1)}n^{-\frac56}\left(b\sqrt{\log n}\right)^{-\frac12}, &\quad b>0,\\
		e^{O(\log\log n)}n^{-\frac56}, &\quad b=0.
		\end{cases}
	\eeqq
\end{lemma}
By Lemma 3.9 of \cite{BaikLeeBipartite}, the part of the double integral $S_n$ with $|y_1|>n^{-\frac12+\e}$ is $O(e^{-n^\e})$ with high probability. For $|y_1|<n^{-\frac12+\e}$, we modify the $z_2$-integral by replacing the vertical contour $z_2=\g_2+\ii y_2$, $y_2\in\bR$ with the contour $z_2=\hg_2+\ii y_2$, $y_2\in\bR$, where $\hg_2$ is defined for each $y_1$ by
\begin{equation}\label{def:gt_2}
	\hg_2(y_1)=\frac{\muf(\mus+n^{-1})}{\muf+\ii y_1}.
\end{equation}
The new contour is a modification of the one introduced by Baik and Lee in \cite{BaikLeeBipartite}. Similarly to the case in \cite{BaikLeeBipartite}, we observe that the change in product $z_1z_2$ for $(z_1,z_2)$ near $(\muf,\mus)$, but not the individual changes in $z_1$, $z_2$ with $z_1z_2$ being fixed, greatly impacts the change in $G(z_1,z_2)$, since the main contribution for the latter comes from the term $\frac1{4z_1z_2-\mu_1}$. This suggests behavior of $G(\muf+\ii y_1,\hg_2)$ should be similar to that of $G(\muf,\mus+n^{-1})$ for the current range of $y_1$. 

Note that this deformation for each $z_1=\muf+\ii y_1$ is valid. Indeed, if $(z_1,z_2)$ is a point on the branch cut of the logarithmic function in $G$, then $4z_1z_2-\mu_1$ is real and non-positive. That is, for some $r\geq0$, 
\[
\re z_2=\re \left(\frac{\mu_1-r}{4(\muf+\ii y_1)}\right)=\re\left(\frac{\mu_1-r}{4\muf(\mus+n^{-1})}\hg_2\right)<\re\hg_2.
\]
This implies that the deformed contour does not cross the branch cut. Thus, the part of $S_n$ with $|y_1|<n^{-1/2+\e}$ is equal to 
\[
\int_{-n^{-1/2+\e}}^{n^{-1/2+\e}}\int_{-\infty}^\infty \exp[n(G(\muf+\ii y_1,\hg_2+\ii y_2)-\Gmu)]\dd y_2\dd y_1.
\]
We now carry out the analysis of this double integral, first by truncating the $y_2-$integral. For given $y_1,y_2\in\bR$,
\begin{equation}\label{eqn:diffGhat}
	\begin{split}
	&\quad G(\muf+\ii y_1,\hg_2+\ii y_2)-\Gmu\\
	&= B_n\left(\ii(y_1+y_2)+\frac{\mus+n^{-1}}{1+\ii\frac{y_1}{\muf}}-\mus\right)-\alpha_n\log\left(1+\frac{\ii y_1}{\muf}\right)\\
	&\quad-\frac1{2n}\sum_{j=2}^n\log\left(1+\frac{4\muf n^{-1}-4y_1y_2}{\mu_1-\mu_j}+\ii\frac{4\muf y_2}{\mu_1-\mu_j}\right)-\frac1{2n}\log(4\muf n^{-1}-4y_1y_2+\ii 4\muf y_2).	
	\end{split}
\end{equation}

Our truncation procedure, which relies on bounding $|G(\muf+\ii y_1,\hg_2+\ii y_2)-\Gmu|$, aligns rather closely with the arguments in \cite{BaikLeeBipartite}, where the difference $|G(\g_1+\ii y_1,\hg_2+\ii y_2)-G(\g_1,\g_2)|$ is the focus there. After truncating in the $y_1$ variable, the contribution from the part $|y_2|>n^{-\frac12+\e}$ is as follows.

\begin{lemma}\label{lem:trunc1}
	The following bound holds for the truncated integral.
	\begin{equation}\label{eqn:trunc1}
		\int_{|y_1|\leq n^{-\frac12+\e}}\int_{|y_2|>n^{-\frac12+\e}}\exp[n(G(\muf+\ii y_1,\hg_2+\ii y_2)-\Gmu)]\dd y_2\dd y_1 = O(n^{-1}).
	\end{equation}
\end{lemma}
\begin{proof}
	From \eqref{eqn:diffGhat}, 
	\begin{equation}\label{eqn:re_diffG}
		\begin{split}
			&\quad\re\left[n\left(G(\muf+\ii y_1,\hg_2+\ii y_2)-\Gmu\right)\right]\\
			&=\frac{B_n(\muf)^2-nB_n\mus y_1^2}{(\muf)^2+y_1^2}-\frac{\alpha_n n}{2}\log\left(1+\left(\frac{y_1}{\muf}\right)^2\right)-\frac1{4}\log\left(\left(\frac{4\muf}{n}-4y_1 y_2\right)^2+(4\muf y_2)^2\right)\\
			&\quad-\frac1{4}\sum_{j=2}^n\log\left(\left(1+\frac{4\muf n^{-1}-4y_1 y_2}{\mu_1-\mu_j}\right)^2+\frac{(4\muf y_2)^2}{(\mu_1-\mu_j)^2}\right).
		\end{split}
	\end{equation}
	Applying Taylor expansion in terms of $y_1$ around 0 to the first two terms on the right hand side of \eqref{eqn:re_diffG}, then for some $c>0$, the first line has upper bound
	\[
	c_0-cny_1^2-\frac1{2}\log\left(4\muf |y_2|\right), \quad \text{uniformly in } |y_1|\leq n^{-\frac12+\e}.
	\]
	For the sum of log, by consider the cases $y_1y_2>0$ and $y_1y_2<0$ as in \cite{BaikLeeBipartite}, there exists $c'>0$ such that for all $j\in\{2,3,\dots, n\}$, for all $|y_1|<n^{-\frac12+\e}$ and $|y_2|>n^{-\frac12+\e}$,
	\[
	\left(1+\frac{4\muf n^{-1}-4y_1 y_2}{\mu_1-\mu_j}\right)^2+\frac{(4\muf y_2)^2}{(\mu_1-\mu_j)^2}\geq 1+c'y_2^2.
	\]
	Therefore, 
	\[
	\re\left[n\left(G(\muf+\ii y_1,\hg_2+\ii y_2)-\Gmu\right)\right]\leq c_0-cny_1^2-\frac1{2}\log\left(4\muf |y_2|\right)-\frac n4\log(1+c'y_2^2),
	\]
	and the left hand side of \eqref{eqn:trunc1} has upper bound
	\[
	\int_{|y_1|\leq n^{-\frac12+\e}}\int_{|y_2|>n^{-\frac12+\e}}e^{c_0-cny_1^2}e^{-\frac n4\log(1+c'y_2^2)}(4\muf |y_2|)^{-\frac12}\dd y_2\dd y_1,
	\]
	which is a product of a $y_1$-integral and a $y_2$-integral. Each individual integral is $O(n^{-\frac12})$, so we obtain the lemma. 
\end{proof}
The computation of $S_n$ is now reduced to that of the same integral, over the subset $|y_1|\leq n^{-\frac12+\e}$ and $|y_2|<n^{-\frac12+\e}$. However, we need to truncate the $y_2$-integral further.
\begin{lemma}
	For this further truncation, we have the following bound.
	\beqq
	\int_{|y_1|\leq n^{-\frac12+\e}}\int_{n^{-\frac23+2\e}<|y_2|<n^{-\frac12+\e}}\exp[n(G(\muf+\ii y_1,\hg_2+\ii y_2)-\Gmu)]\dd y_2\dd y_1 = O(e^{-n^{4\e}}).
	\eeqq
\end{lemma}
\begin{proof}
	Computations similar to the proof of Lemma \eqref{lem:trunc1} gives
	\begin{equation}
		\begin{split}\label{eqn:diffG_trunc2}
		\re\left[n\left(G(\muf+\ii y_1,\hg_2+\ii y_2)-\Gmu\right)\right]
		&\leq c_0 -\frac1{4}\sum_{j=2}^n\log\left(\left(1+\frac{4\muf n^{-1}-4y_1 y_2}{\mu_1-\mu_j}\right)^2+\frac{(4\muf y_2)^2}{(\mu_1-\mu_j)^2}\right).
		\end{split}	
	\end{equation}
	Observe that $n^{-\frac23}\ll \mu_1-\mu_{n^{4\e}}\ll n^{-\frac23+2\e}$. Thus, for $2\leq j\leq n^{4\e}$, $\left(\frac{4\muf y_2}{\mu_1-\mu_j}\right)^2 \geq (4\muf)^2$ and we obtain
	\[
	-\frac14\sum_{j=1}^{n^{4\e}}\log\left(\left(1+\frac{4\muf n^{-1}-4y_1 y_2}{\mu_1-\mu_j}\right)^2+\frac{(4\muf y_2)^2}{(\mu_1-\mu_j)^2}\right)\leq -\frac12\log(\muf)n^{4\e}.
	\]
	
	For $j>n^{4\e}$, $\mu_1-\mu_j\geq \mu_1-\mu_{n^{4\e}}\gg n^{-\frac23}$. Since $|y_1|$, $|y_2|\leq n^{-\frac12+\e}$, we have $\mu_1-\mu_j \gg |4\muf n^{-1}-4y_1y_2|$. Using $\log(1-x)\geq -2x$ for $x\in(0,1)$, then for some constant $C,C'>0$, the sum with indices $j>n^{4\e}$ on the right hand side of \eqref{eqn:diffG_trunc2} has upper bound
	\beqq
	\begin{split} -\frac12\sum_{j=n^{4\e}+1}^{n}\log\left(1+\frac{4\muf n^{-1}-4y_1 y_2}{\mu_1-\mu_j}\right)
	&\leq 4|\muf n^{-1}-y_1y_2|\sum_{j=n^{4\e}+1}^{n}\frac1{\mu_1-\mu_j}\\
	&\leq Cn|\muf n^{-1}-y_1y_2|\leq C'n^{2\e}.
	\end{split}
	\eeqq
	Here, the second inequality holds with probability at least $1-\e$ by Lemma \ref{lem:sum_from_index2}. Thus, we obtain the uniform bound
	\[
	\re\left[n\left(G(\muf+\ii y_1,\hg_2+\ii y_2)-\Gmu\right)\right] \leq c_0-Cn^{4\e}
	\]
	for some constant $C>0$. This implies the lemma.
\end{proof}

Therefore, we have shown that,
\begin{equation}\label{eqn:S_error}
	S_n = \int_{ -n^{-\frac12+\e}}^{n^{-\frac12+\e}}\int_{-n^{-\frac23+2\e}}^{n^{-\frac23+2\e}}\exp[n(G(\muf+\ii y_1,\hg_2+\ii y_2)-\Gmu)]\dd y_2\dd y_1 + O(n^{-1}).
\end{equation}

We proceed to compute the double integral in \eqref{eqn:S_error}. For $|y_1|\leq n^{-\frac12+\e}$ and $|y_2|<n^{-\frac23+2\e}$, by Taylor series and the definitions of $\muf$ and $\mus$ in \eqref{def:muf_mus}, the second line of \eqref{eqn:diffGhat} for $G(\muf+\ii y_1,\hg_2+\ii y_2)-\Gmu$ is 
\[
B_n(n^{-1}+\ii y_2)-\ii\frac{B_n n^{-1}}{\muf}y_1-\frac{B_n(\frac{\muf+\mus}{2}+n^{-1})}{(\muf)^2}y_1^2+O(y_1^3),
\]
while the last line, after factorizing the arguments of logarithm functions, becomes
\begin{multline}
	-\frac1{2n}\sum_{j=2}^n\log\left(1+\frac{4\muf n^{-1}+4\ii\muf y_2}{\mu_1-\mu_j}\right)-\frac1{2n}\log(4\muf n^{-1}+4\ii\muf y_2)\\
	-\frac{1}{2n}\sum_{j=2}^n\log\left(1-\frac{4y_1y_2}{\mu_1-\mu_j+4\muf n^{-1}+4\ii\muf y_2}\right)-\frac{1}{2n}\log\left(1-\frac{4y_1y_2}{4\muf n^{-1}+4\ii\muf y_2}\right).
\end{multline}
Combine the above two displays, we obtain
\begin{equation}\label{eqn:exp_diffGmu}
	\begin{split}
		&\quad\exp[n(G(\muf+\ii y_1,\hg_2+\ii y_2)-\Gmu)]\\
		&=\exp\left[-\frac{\ii B_n }{\muf}y_1-\frac{B_n(\frac{\muf+\mus}{2}+n^{-1})}{(\muf)^2}ny_1^2\right]\\
		&\quad \cdot \exp\left[B_nn(n^{-1}+\ii y_2)-\frac12\log(4\muf n^{-1}+4\ii\muf y_2)-\frac12\sum_{j=2}^n\log(1+\frac{4\muf n^{-1}+4\ii\muf y_2}{\mu_1-\mu_j})\right]\\
		&\quad \cdot \exp\left[-\frac12\sum_{j=1}^n\log\left(1-\frac{4y_1y_2}{\mu_1-\mu_j+4\muf n^{-1}+4\ii\muf y_2}\right)+O(ny_1^3)\right].
	\end{split}
\end{equation}

Let $H(y_1,y_2)$ denote the product of the first two exponential factors on the right hand side of \eqref{eqn:exp_diffGmu}, and $L(y_1,y_2)$ be the last factor. That is,
\[
\exp[n(G(\muf+\ii y_1,\hg_2+\ii y_2)-\Gmu)]=H(y_1,y_2)L(y_1,y_2).
\]  There is a constant $c>0$ such that $\frac{B_n(\frac{\muf+\mus}{2}+n^{-1})}{(\muf)^2}>c$, so
\begin{equation}\label{eqn:H}
	\begin{split}
		|H(y_1,y_2)|
		&\leq \exp\left[B_n-\frac{B_n(\frac{\muf+\mus}{2}+n^{-1})}{(\muf)^2}ny_1^2-\frac12\re\sum_{j=2}^n\log(1+\frac{4\muf n^{-1}+4\ii\muf y_2}{\mu_1-\mu_j})\right]\\
		&\leq \exp\left[B_n-cny_1^2-\frac12\log\left(\frac{4\muf |y_2|}{\mu_1-\mu_2}\right)\right]\\
		&\leq C(\mu_1-\mu_2)^{\frac12}|y_2|^{-\frac12}
		e^{-cny_1^2} ,
	\end{split}
\end{equation}
for some constant $C>0$. On the other hand, by Lemma \ref{lem:sum_from_index2}, there exists constant $C>0$ such that 
\begin{equation}\label{eqn:sum_ell}
	\sum_{j=2}^{n}\frac1{|\mu_1-\mu_j+4\muf n^{-1}+4i\muf y_2|^\ell}\leq \sum_{j=2}^{n}\frac1{(\mu_1-\mu_j)^\ell}\leq 
	\begin{cases}
	Cn^{1+\e}, &\quad \ell=1,\\
	Cn^{\frac{2\ell}3+\e}, &\quad \ell=2,3,\dots
	\end{cases}
\end{equation}
At the same time,
\[
\left|\frac{4y_1y_2}{4\muf n^{-1}+4\ii\muf y_2}\right| \leq \frac{|y_1|}{\muf}=O(n^{-\frac12+\e}).
\]
Thus, applying Taylor series, we have
\begin{equation}\label{eqn:I2term}
	\begin{split}
		L(y_1,y_2)=1&+\sum_{k=1}^{n}\frac{2y_1y_2}{\mu_1-\mu_j+4\muf n^{-1}+4\ii\muf y_2}+ O(n^{-\frac12+3\e}).
	\end{split}
\end{equation}
Observe that 
\beqq
	|\mu_1-\mu_j+4\muf n^{-1}+4\ii\muf y_2|\geq \begin{cases}
	\mu_1-\mu_j, &\quad j=2,3,\dots, n,\\
	4\muf |y_2|, &\quad j=1.
	\end{cases}
\eeqq
Applying \eqref{eqn:sum_ell} with $\ell=1$, we obtain 
\begin{equation}\label{eqn:L_bound}
	|L(y_1,y_2)-1|\leq Cn|y_1y_2| 
	+C'n^{-\frac12+3\e}.
\end{equation}
We now write
\begin{equation}\label{eqn:sumI1_I2}
	\int_{ -n^{-\frac12+\e}}^{n^{-\frac12+\e}}\int_{-n^{-\frac23+2\e}}^{n^{-\frac23+2\e}}\exp[n(G(\muf+\ii y_1,\hg_2+\ii y_2)-\Gmu)]\dd y_2\dd y_1=I_1+I_2,
\end{equation}
where $I_2$ is given by
\begin{equation}
	I_2=\int_{ -n^{-\frac12+\e}}^{n^{-\frac12+\e}}\int_{-n^{-\frac23+2\e}}^{n^{-\frac23+2\e}} H(y_1,y_2)(L(y_1,y_2)-1) \dd y_2\dd y_1.
\end{equation}
By \eqref{eqn:H} and \eqref{eqn:L_bound}, there are constant $C_j>0, j=1,2,3$ such that
\begin{equation}
	\begin{split}
		|I_2|&\leq C_1\int_{ -n^{-\frac12+\e}}^{n^{-\frac12+\e}}\int_{-n^{-\frac23+2\e}}^{n^{-\frac23+2\e}} |H(y_1,y_2)|\left(n|y_1y_2|+n^{-\frac12+3\e}\right)\dd y_2\dd y_1\\
		&\leq C_2n(\mu_1-\mu_2)^{\frac12} \int_{ -n^{-\frac12+\e}}^{n^{-\frac12+\e}}\int_{-n^{-\frac23+2\e}}^{n^{-\frac23+2\e}} e^{-cny_1^2}\left(|y_1||y_2|^{\frac12}+n^{-\frac32+3\e}|y_2|^{-\frac12}\right) \dd y_2\dd y_1\\
		&\leq C_3 n^{-1+4\e}(\mu_1-\mu_2)^{\frac12}.
	\end{split}
\end{equation}
Together with \eqref{eqn:S_error} and \eqref{eqn:sumI1_I2}, this implies that on the event $\mu_1-\mu_2\leq n^{-2/3+\e}$,
\beqq
	S_n=I_1+O(n^{-1}).
\eeqq
Note that
\begin{equation}
	I_1=\int_{ -n^{-\frac12+\e}}^{n^{-\frac12+\e}}\int_{-n^{-\frac23+2\e}}^{n^{-\frac23+2\e}} H(y_1,y_2)\dd y_2\dd y_1
\end{equation}
is equal to the product of two single integrals $I_{11}$ and $I_{12}$ as follows.  First,
\begin{equation*}
	\begin{split}
		I_{11}
		&=\int_{ -n^{-\frac12+\e}}^{n^{-\frac12+\e}}\exp\left[-\frac{\ii B_n }{\muf}y_1-\frac{B_n(\frac{\muf+\mus}{2}+n^{-1})}{(\muf)^2}ny_1^2\right]\dd y_1\\
		&= n^{-\frac12}\int_{ -n^{\e}}^{n^{\e}}e^{-c_1x^2}\cos\left(\frac{c_2}{\sqrt{n}}x\right)\dd x, \quad (c_1,c_2):=\left(\frac{B_n(\frac{\muf+\mus}{2}+n^{-1})}{(\muf)^2},\frac{B_n}{\muf}\right).
	\end{split}
\end{equation*}
Using Taylor's series of cosine, we obtain that for some $C>0$,
\begin{equation}\label{eqn:I11}
	I_{11}=Cn^{-\frac12}\left(1+O(n^{-1+2\e})\right).
\end{equation}
Second, we have
\begin{equation}
	I_{12}
	=\int_{-n^{-\frac23+2\e}}^{n^{-\frac23+2\e}}\exp\left[n(G(\muf,\mus+n^{-1}+\ii y)-\Gmu)\right]\dd y.
\end{equation}
We first check that $I_{12}$ is close to the integral over the whole real line
\beq
K_n:=\int_{-\infty}^{\infty}\exp\left[n(G(\muf,\mus+n^{-1}+\ii y)-\Gmu)\right]\dd y.
\eeq
By \eqref{eqn:exp_diffGmu}, for all $y\in\bR$,
\beqq
\begin{split}
	\re\left[n\left(G(\muf,\mus+\frac1n+\ii y)-\Gmu\right)\right]&\leq c_0-
	\frac14\log\left((4\muf n^{-1})^2+(4\muf y_2)^2\right)\\ &\quad-\frac14\sum_{j=2}^n\log\left(\left(1+\frac{4\muf/n }{\mu_1-\mu_j}\right)^2+\left(\frac{4\muf y}{\mu_1-\mu_j}\right)^2\right).
\end{split}
\eeqq
In the case $n^{-\frac23+2\e}<|y|<n$, we use $-
\frac14\log\left((4\muf n^{-1})^2+(4\muf y_2)^2\right)\leq \frac12\log n$, and bound
\[
\sum_{j=2}^n\log\left(\left(1+\frac{4\muf /n}{\mu_1-\mu_j}\right)^2+\left(\frac{4\muf y}{\mu_1-\mu_j}\right)^2\right)\geq 2\sum_{j=2}^{n^{2\e}}\log\left(\frac{4\muf |y|}{\mu_1-\mu_j}\right)\geq Cn^{2\e}
\]
using the fact that $\mu_1-\mu_{n^{2\e}}\ll n^{-\frac23+2\e}$ with high probability. For $|y|>n$, we drop the negative term $-
\frac14\log\left((4\muf n^{-1})^2+(4\muf y_2)^2\right)$, while, for some $c>0$, 
\[
\sum_{j=2}^n\log\left(\left(1+\frac{4\muf /n}{\mu_1-\mu_j}\right)^2+\left(\frac{4\muf y}{\mu_1-\mu_j}\right)^2\right)\geq n \log(1+cy^2)\geq n\log(c|y|).
\]
Therefore, for some $C',C''>0$, it holds with high probability that 
\begin{equation}\label{eqn:diffK_I12}
	|K_n-I_{12}| \leq C'\left(n^{1/2}e^{-Cn^{-2\e}}+\int_0^\infty (cy)^{-\frac n4}\right)\dd y \leq C''e^{-c'n^{2\e}}.
\end{equation}

We determine in Subsection \ref{sec:Kn} that, on the event $\cF_\e$,
\begin{equation}\label{eqn:Kn}
	K_n=\begin{cases}
		e^{O(1)}n^{-\frac13}\left(b\sqrt{\log n} \right)^{-\frac12}, &\quad b>0,\\
		e^{O(\log\log n)}n^{-\frac13},&\quad b=0.
	\end{cases}
\end{equation}
Assuming \eqref{eqn:Kn} is true, then using \eqref{eqn:diffK_I12} and the fact that $S_n=I_{11}\cdot I_{12}+O(n^{-1})$, we obtain Lemma \ref{lem:Sn} .

\subsubsection{Proof of \eqref{eqn:Kn} when $b>0$}\label{sec:Kn}
For brevity, we introduce the following two notations to be used throughout the Subsection:
\begin{equation}
	a_+ =\frac{\mus+\mu_2^{(2)}}{2}= \frac{\mu_1}{8\muf}+\frac{\mu_2}{8\mu_2^{(1)}},
\end{equation}
where  $\mu_2^{(1)}:=\frac{\alpha_n+\sqrt{\alpha_n^2+\mu_2B_n^2}}{2B_n}$ and  $\mu_2^{(2)}:=\frac{-\alpha_n+\sqrt{\alpha_n^2+\mu_2B_n^2}}{2B_n}$.

We now show that the integral $K_n$, on the event $\cF_\e$, satisfies \eqref{eqn:Kn}, first under the assumption $b>0$. By Cauchy theorem, for every $r\in(0,n^{-1}]$,
\[
\ii K_n = \int_\Gamma \exp\left[n(G(\muf,z)-\Gmu)\right]\dd z,
\]
where $\Gamma=\Gamma_1\cup\Gamma_{2}^{\pm}\cup\Gamma_{3}^{\pm}$ is the vertical keyhole-like contour as in Figure \ref{fig:keyhole-contour}. In particular, given a function $\phi_r: \bR_+\to[0,\pi]$ of $r$ such that $\phi_r\to 0$ as $r\downarrow0$, we let $\Gamma_1$ be the arc $\{\mus+re^{i\theta}: \theta\in[-\pi+\phi_r, \pi-\phi_r]\}$, $\Gamma_{2}^{\pm}=\{x\pm r\sin\phi_r: x\in [a_+,\mus-r\cos\phi_r]\}$, and $\Gamma_{3}^{\pm}$ be the rays $\{a_+\pm iy: y\in [r\sin\phi_r, \infty)\}$. 
Then, for fixed $n$,
\begin{equation}\label{eqn:iKn}
	\ii K_n = \lim\limits_{r\downarrow 0} \int_\Gamma \exp\left[n(G(\muf,z)-\Gmu)\right]\dd z.
\end{equation}

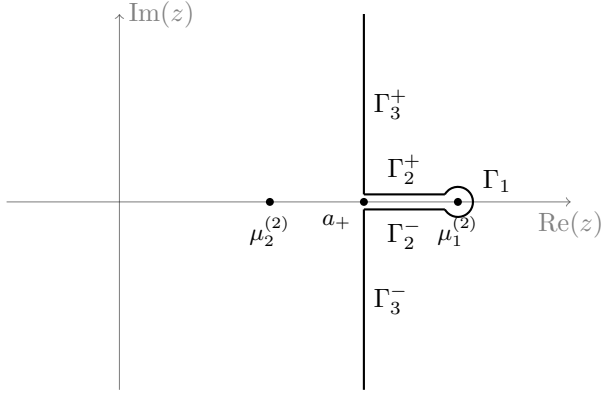
\begin{figure}
	\centering
	\begin{tikzpicture}[scale = 1]
		
		\draw[opacity=0.5, ->] (-3, 0) -- (4.5, 0) node[anchor=north] {$\mathrm{Re}(z)$};
		\draw[opacity=0.5, ->] (-1.5, -2.5) -- (-1.5, 2.5) node[anchor=west] {$\mathrm{Im}(z)$};
		
		\node[circle,fill=black,inner sep=0pt,minimum size=3pt,label=below:{\small\(\mus\)}] (a) at (3,0) {};
		
		\node[circle,fill=black,inner sep=0pt,minimum size=3pt,label=below:{\small\(\mu_2^{(2)}\)}] (a) at (0.5,0) {};
		
		\node[circle,fill=black,inner sep=0pt,minimum size=3pt,label=south west:{\small\(a_+\)}] (a) at (1.75,0) {};
		
		\draw [arrow] (2.82679491924,-0.1) coordinate (low) arc (-150:150:0.2) coordinate (top) node[midway,above right] {\(\Gamma_1\)};
		\draw [arrow] (top) -- (1.75, 0.1) coordinate (topmu) node[midway,above] {\(\Gamma_2^+\)};
		\draw [arrow] (1.75, -0.1) coordinate (lowmu) -- (low) node[midway,below] {\(\Gamma_2^-\)};
		\draw [arrow] (topmu) -- (1.75, 2.5) node[midway,right] {\(\Gamma_3^+\)};
		\draw[arrow] (1.75, -2.5) -- (lowmu) node[midway,right] {\(\Gamma_3^-\)};
	\end{tikzpicture}
	\caption{Keyhole-like contour of integration $\Gamma$.}
	\label{fig:keyhole-contour}
\end{figure}

For $\Gamma_1$, using the fact that $\log(x+\ii t) \to \log|x|+\ii\pi$ as $t\downarrow0$ for $x<0$, and $dz=\ii re^{\ii\theta}d\theta$ where $\theta$ takes values in $[-\pi-\phi_r,\pi+\phi_r]$ as described above, one can verify using Fubini's that for each fixed $n$, the integral over $\Gamma_1$ converges to 0 as $r\to 0$. 

We show in Lemma \ref{lem:G2pm} that, in the limit $r\downarrow 0$, the contribution from $\Gamma_2^+\cup\Gamma_2^-$ part of the contour satisfies the asymptotics \eqref{eqn:Kn} in both cases $b>0$ and $b=0$. In Lemma \ref{lem:G3pm}, we confirm that for any keyhole radius $r\in(0,1/n]$, with probability arbitrarily close to 1, the contribution from $\Gamma_3^+\cup\Gamma_3^-$ is little-o of that of $\Gamma_2^+\cup\Gamma_2^-$ when $b>0$. Together, the lemmas establish \eqref{eqn:Kn} when $b>0$.

\begin{lemma}\label{lem:G2pm}
	On the event $\cF_\e$, it holds that
	\begin{equation}\label{eqn:Gamma2}
		\lim\limits_{r\downarrow0}\int_{\Gamma_2^+\cup\Gamma_2^-}\exp\left[n(G(\muf,z)-\Gmu)\right]\dd z = \begin{cases}
			\ii e^{O(1)}n^{-\frac13}\left(b\sqrt{\log n}\right)^{-\frac12}, &\quad b>0,\\
			\ii e^{O(1)}n^{-\frac13},&\quad b=0.
		\end{cases}
	\end{equation}
\end{lemma}
\begin{proof}
	Recall that, if $z \in \Gamma_2^{\pm}$, then $z=x\pm \ii r\sin \phi_r$ where $x\in[a_+,\mus-r\cos \phi_r]$. Set $s=\mus-x$, we have
	\beqq
	\begin{split}
		n(G(\muf,z)-\Gmu)
		&=-nB_n(\mus-x)\pm \ii n B_n r\sin\phi_r-\frac12\sum_{j=2}^n\log\left(1-\frac{4\muf(\mus-x)\mp \ii 4\muf r\sin\phi_r}{\mu_1-\mu_j}\right)\\
		&\quad\quad-\frac12\log(-4\muf(\mus-x)\pm \ii 4\muf r\sin\phi_r)\\
		&\stackrel{r\downarrow0}{\to} -B_nns-\frac12\sum_{j=2}^n\log\left(1-\frac{4\muf(\mus-x)}{\mu_1-\mu_j}\right)-\frac12\log(4\muf(\mus-x))\mp \ii\frac\pi 2.
	\end{split}
	\eeqq
	Let $A$ be the left hand side of \eqref{eqn:Gamma2}. We then obtain 
	\begin{equation}\label{eqn:Gamma2_simplified}
		A=\frac{2\ii}{\sqrt{4\muf}}\int_0^{\mus-a_+}\exp\left(-B_nns-\frac12\sum_{j=2}^n\log(1-\frac{4\muf s}{\mu_1-\mu_j})\right)\frac{\dd s}{\sqrt{s}}.
	\end{equation}
	Observe that $\frac{4\muf s}{\mu_1-\mu_j} \in [0,\frac12]$ for all $s\in[0,\mus-a_+]$ and all $j$. As $0<-\log(1-x)-x\leq x^2$ for $x\in[0,\frac12]$, there exists $\zeta\in[0,1]$ such that 
	\begin{equation}\label{eqn:exponent_Gamma2}
		\begin{split}
			-B_nns-\frac12\sum_{j=2}^n\log(1-\frac{4\muf s}{\mu_1-\mu_j})
			&=-B_nns+2\muf s\sum_{j=2}^n\frac1{\mu_1-\mu_j}+\frac{\zeta(4\muf s)^2}2\sum_{j=2}^n\frac1{(\mu_1-\mu_j)^2}.
		\end{split}
	\end{equation}
Define $y:=n^{2/3}s \in[0,n^{2/3}(\mus-a_+)]$, and let  $\omega_{1n}$, $\omega_{2n}$ be random variables given by 
	\begin{equation}\label{eqn:omega}
		\sum_{j=2}^n\frac1{n^{\frac23}(\mu_1-\mu_j)} = s_{\MP}(d_+)n^{\frac13}+\omega_{1n}, \quad \sum_{j=2}^n\frac1{\left(n^{\frac23}(\mu_1-\mu_j)\right)^2} = \omega_{2n}.
	\end{equation}
Then, \eqref{eqn:exponent_Gamma2} simplifies to 
\beq
-B_nns-\frac12\sum_{j=2}^n\log(1-\frac{4\muf s}{\mu_1-\mu_j})=n^{\frac13}y\left(-B_n + 2\muf s_{\MP}(d_+)\right)+\left[y(2\muf\omega_{1n}+2(\muf)^2\zeta \omega_{2n} y^2\right],
\eeq
where the term inside the square brackets is $O(1)$, uniformly for $y\in[0,n^{2/3}(\mus-a_+)]$. Observe also 
	\beqq
	-B_n + 2\muf s_{\MP}(d_+)= -B_n+\frac{\a_n+\sqrt{\a_n^2+\mu_1B_n^2}}{B_n}s_{\MP}(d_+),
	\eeqq
	where $B_n-B_c=\Theta(\b-\b_c)$ and $B_c$ satisfies $\sqrt{\a^2+d_+B_c^2}=\a+d_+s_{\MP}(d_+)$. Therefore, applying Taylor expansion to the above expression with respect to $B_n$ near $B_c$ and $\mu_1$ near $d_+$, using  $\mu_1-d_+=O(n^{-2/3})$ on the event $\cF_\e$, we obtain
	\begin{equation}\label{eqn:Bn_cancel}
		\begin{split}
			-B_n + 2\muf s_{\MP}(d_+)
			&=-\frac{2s_{\MP}(d_+)\la^{\frac12}}{\sqrt{1+\la}}(\beta-\beta_c)+O\left((\beta-\beta_c)^2\right).
		\end{split}
	\end{equation}
	Thus, on the event $\cF_\e$,
	\beq
	-B_nns-\frac12\sum_{j=2}^n\log(1-\frac{4\muf s}{\mu_1-\mu_j})=-\frac{2s_{\MP}(d_+)\la^{\frac12}b\sqrt{\log n}}{\sqrt{1+\la}}y +O(1),
	\eeq
	and we arrive at
	\begin{equation*}
		A=\frac{\ii e^{O(1)}}{n^{\frac13}}\int_0^{n^{\frac23}(\mus-a_+)}\exp\left(-\frac{2s_{\MP}(d_+)\la^{\frac12}b\sqrt{\log n}}{\sqrt{1+\la}}y\right)\frac{\dd y}{\sqrt{y}}=\begin{cases}
			\ii e^{O(1)}n^{-\frac13}b^{-\frac12}(\log n)^{-\frac14}, &\quad b>0,\\
			\ii e^{O(1)}n^{-\frac13}, &\quad b=0.
		\end{cases}
	\end{equation*}
	This completes the proof of the lemma.
\end{proof}

\begin{lemma}\label{lem:G3pm}
	Let $\theta_n=\frac{n^{2/3}(\mus-\mu_2^{(2)})}{2}$. For $b\geq 0$ and for every $0<r<n^{-1}$, on the event $\cF_\e$, 	
	\begin{equation*}
		\left|\int_{\Gamma_3^+\cup\Gamma_3^-}\exp\left[n(G(\muf,z)-\Gmu)\right]\dd z\right|\leq n^{-\frac13}\exp\left(-\frac{2s_{\MP}(d_+)\sqrt{\la}\theta_n}{\sqrt{1+\la}}b\sqrt{\log n}+O(1)\right).
	\end{equation*}
\end{lemma}
\begin{proof}
	Since $G(\muf,\overline{z})=\overline{G(\muf, z)}$ for all $z\in\bC$, it suffices to bound the integral over $\Gamma_3^+$. We define
	\[
	G_+(\muf,a_+)=\lim\limits_{t\downarrow0}G(\muf, a_++\ii t), \quad \widetilde{G}(t)=G(\muf, a_++\ii t)-G_+(\muf,a_+).
	\]
	Then, for $z\in\Gamma_3^+$, 
	\[
	n(G(\muf,z)-\Gmu) = n(G_+(\muf,a_+)-\Gmu))+ n\widetilde{G}(t),
	\]
	and we have
	\begin{equation}\label{eqn:int_G3}
		\begin{split}
			\left|\int_{\Gamma_3^+}\exp\left[n(G(\muf,z)-\Gmu)\right]\dd z\right|
			&\leq \left|e^{n(G_+(\muf,a_+)-\Gmu)}\right|\int_0^\infty e^{n\re \widetilde{G}(t)} \dd t.
		\end{split}
	\end{equation}
	For fixed $k>2$,
	\[
	n\re\widetilde{G}(t)=-\frac14\sum_{j=1}^{n}\log\left(1+\left(\frac{4\muf t}{4\muf a_+-\mu_j}\right)^2\right) \leq-\frac14\sum_{j=2}^{n}\log\left(1+\left(\frac{4\muf t}{\mu_1-\mu_j}\right)^2\right)\leq-\frac{k}{4}\log\left(1+\xi^{-2}n^{\frac43}t^2\right),
	\]
	where $\xi:=\frac{n^{2/3}}{4\muf}|\mu_1-\mu_{k+1}|$ is $O(1)$ on the event $\cF_\e$. Thus,
	\begin{equation}
		\int_0^\infty e^{n\re\widetilde{G}(t)}\dd t\leq\int_0^\infty(1+\xi^{-2}n^{\frac43}t^2)^{-\frac k4}\dd t=\exp\left(-\frac23 \log n+O(1)\right).
	\end{equation}
	At the same time, on the event $\cF_\e$, $ \theta_n=\Theta(n^{2/3}(\mu_1-\mu_2))=\Theta(1)$. Thus,  similar to the proof of Lemma \ref{lem:G2pm}, we obtain that 
	\begin{equation}
		\begin{split}
			n(G_+(\muf,a_+)-\Gmu)
			&=-B_nn^{\frac13}\theta_n-\frac12\sum_{j=2}^n\log\left(1-\frac{4\muf\theta_n}{n^{\frac23}(\mu_1-\mu_j)}\right)-\frac12\log(4\muf n^{-2/3}\theta_n)-\frac{\ii\pi}2\\
			&=-n^{\frac13}\theta_n\left(B_n-2\muf s_{\MP}(d_+)-\omega_{1n}n^{-\frac13}\right)+(\zeta \theta_n)^2\omega_{2n}+\frac{\log n}3+O(1)\\
			&=-\frac{2\sqrt{\la}s_{\MP}(d_+)\theta_n}{\sqrt{1+\la}}b\sqrt{\log n}+\frac{\log n}3+O(1),
		\end{split}		
	\end{equation}
	on the event $\cF_\e$. Applying the above two displays to \eqref{eqn:int_G3}, we obtain the lemma.
\end{proof}

\subsubsection{Proof of \eqref{eqn:Kn} when $b=0$}

Observe that when $b=0$, Lemmas \ref{lem:G2pm} and \ref{lem:G3pm} using keyhole contour shows that, with probability $1-\e$ for arbitrary small $\e>0$, the contribution from the vertical and horizontal parts of the contour are both $n^{-\frac13}e^{O(1)}$. This provides the upper bound for $K_n$. As some cancellation between the two contributions can occur, further analysis is required for the lower bound. 
In this section, we use the steepest descent contour of $G(\muf, z)$ crossing the real line above $\mus$ to obtain the needed lower bound
\[
K_n \geq n^{-\frac13}e^{O(\log\log n)}.
\]
The argument is inspired by the one provided by Johnstone et al in \cite{JKOP2}.

\begin{lemma}
	There exists a unique saddle point of $G(\muf,z)$ on $z\in(\mus, \infty)$.
\end{lemma}
\begin{proof}
	Observe that 
	\[
	\partial_2 G(\muf,z)=B_n-\frac{1}{2n}\sum_{j=1}^{n}\frac{4\muf}{4\muf z-\mu_j}
	\]
	is an increasing function of $z$ on the interval $(\mus, \infty)$, and that
	\[
	\lim\limits_{z\downarrow\mus}\partial_2G(\muf,z)=-\infty, \quad \lim\limits_{z\to\infty}\partial_2G(\muf, z)=B_n>0.
	\]
	Thus, there is a unique solution $z_c\in (\mus, \infty)$ to the equation $\partial_2G(\muf,z)=0$. Moreover, $\partial_2^2G(\muf,z)>0$ for all $z>\mus$. Thus, $z_c$ is a saddle point of $\re[G(\muf, z)]$.
\end{proof}

Let $\Gs$ be the steepest descent contour of $G(\muf, z)$ crossing $z_c$. For $z=x+\ii y\in\Gs$, 
\begin{equation*}
	0=\im[G(\muf,z)]=B_ny-\frac1{2n}\sum_{j=1}^n\arg(4\muf x-\mu_i+\ii 4\muf y),
\end{equation*}
which implies $\Gs$ is symmetric with respect to the $x-$axis. Moreover, for fixed $y>0$, $\arg(4\muf x-\mu_i+\ii4\muf y)$ is strictly decreasing in $x$. This suggests there is at most one solution $x$ to $\im[G(\muf, x+\ii y)]=0$ for any $y>0$. The same applies to $y<0$ by symmetry. We then parameterize $\Gs=\{\Gs(t):0<t<1\}$ such that $\im\Gs(t)$ is increasing in $t$. 

As $B_n|y|\uparrow\frac\pi 2$, $x\to-\infty$ so $\Gs(0^+)=-\infty-\ii\frac{\pi}{2B_n}$ and $\Gs(1^-)=-\infty+\ii\frac{\pi}{2B_n}$. We obtain $\re\Gs(t)$ is bounded above, and $K_n$ as in \eqref{eqn:iKn} satisfies
\[
\ii K_n=\int_{\Gs}\exp\left[n(G(\muf, z)-\widehat{G})\right]\dd z.
\]
We now consider points on the contour $\Gs$ with real part $\mus$. 
\begin{lemma}\label{lem:y0}
	The function
	\[
	f(y):=\im[G(\muf, \mus+\ii y)]=B_ny-\frac{\pi}{4n}-\frac{1}{2n}\sum_{j=2}^n\arctan\left(\frac{4\muf y}{\mu_1-\mu_j}\right)
	\]
	has a unique positive root $y_0$. Furthermore, for any sequence $a_n\to \infty$, $a_n=O(n^\delta)$ for any $\delta>0$, 
	\begin{equation}\label{eqn:y0_bounds}
		n^{-2/3}a_n^{-1}\leq y_0\leq n^{-2/3}a_n, \quad \text{asymptotically almost surely.}
	\end{equation}
\end{lemma}
\begin{proof}
	Existence and uniqueness of $y_0>0$ follows from the fact that $f(y)$ is continuous, convex function on $[0,\infty)$ with $f(0)=-\frac\pi{4n}$ and $\lim\limits_{y\to\infty}f(y)=\infty$. 	
	
	Let $y_-$, $y_+$ denote the bounds $a_n^{-1}n^{-2/3}$ and $a_nn^{-2/3}$, respectively. We now verify \eqref{eqn:y0_bounds} by showing that a.a.s., $f(y_-)<0<f(y_+)$. First, using $\arctan(x)\geq x-x^2/4$ for $x\geq 0$ and Lemma \ref{lem:sum_from_index2}, then with probability $1-\e$ for arbitrary $\e>0$, 
	\begin{align*}
		f(y_-)
		&\leq y_-\left(B_n-\frac{2\muf}{n}\sum_{j=2}^n\frac{1}{\mu_1-\mu_j}\right)-\frac\pi{4n}+\frac{(4\muf y_-)^2}{8n}\sum_{j=2}^n\frac1{(\mu_1-\mu_j)^2}\\
		&=y_-\left(B_n-2\muf s_{\MP}(d_+)+O(n^{-1/3})\right)-\frac\pi{4n}+y_-^2\cdot O(n^{1/3})\\
		&=-\frac\pi{4n}+o(n^{-1})<0.
	\end{align*}
	In the last equality, $B_n-2\muf s_{\MP}(d_+)=O(n^{-1-\tau})$ due to rigidity of $\mu_1$ and the fact $B_n=B_c+O(n^{-1-\tau})$ for any $\tau>0$. The second part of the proof relies on the following statistics regarding the eigenvalues of a matrix from the Laguerre orthogonal ensemble. Let
	\begin{align*}
		j_0&=\#\{j:\mu_j>d_+-\frac13a_n n^{-2/3}\}, \\ j^*&=\#\{j:\mu_j>\mu_1-\left(1+\frac\pi2\right)^{-1}a_nn^{-2/3}\}.
	\end{align*}	
	
	By Chebyshev's inequality and \eqref{eqn:Var_Ns}, for some $c>0$, it holds a.a.s. that $j_0\geq ca_n^{3/2}$. Combine with the observation that a.a.s., $\mu_1-d_+=\Theta(n^{-2/3})\ll a_nn^{-2/3}$, we obtain
	\begin{equation}\label{eqn:jstar}
		j^*\geq j_0\geq ca_n^{3/2} \quad \text{a.a.s.}
	\end{equation}
	Since $\arctan(x)\leq x-1$ for $x>1+\frac\pi{2}$ and $j^*=\max\{j:\frac{y_+}{\mu_1-\mu_j}>1+\frac\pi{2}\}$, we have
	\begin{equation*}
		\arctan\left(\frac{4\muf y_+}{\mu_1-\mu_j}\right)\leq \frac{4\muf y_+}{\mu_1-\mu_j}-\mathbbm{1}_{\{j\leq j^*\}}.
	\end{equation*}
	Lemma \ref{lem:sum_from_index2} and the above display imply that a.a.s.,
	\begin{align*}
		f(y_+)
		&=B_ny_+-\frac1{2n}\sum_{j=2}^n\arctan\left(\frac{4\muf y_+}{\mu_1-\mu_j}\right)-\frac\pi{4n}\\
		&\geq B_ny_+-\frac1{2n}\sum_{j=2}^n\frac{4\muf y_+}{\mu_1-\mu_j}+\frac{j^*}{2n}-\frac\pi{4n}\\
		&\geq y_+\cdot O(n^{-1/3})+\frac{ca_n^{3/2}}{2n}-\frac\pi{4n},
	\end{align*}
	which is strictly positive as $y_+=a_nn^{-2/3}$. We obtain the lemma. 
\end{proof}
Let $z_0=\mus+\ii y_0$, and consider the subset 
\[
\Gamma_0=\{z\in\Gs: |\im z|\leq y_0\},
\]
which is a connected curve with endpoints $z_0, \overline{z_0}$ by the parameterization. We have now obtained the needed tools to bound $K_n$ as follows. 

Observe that $G(\muf, z)-\widehat{G}$ is real on $\Gs$ and is monotone decreasing as $z$ moves away from the point $z_c$ along $\Gs$. Also, $\frac{\dd y}{\dd t}>0$ from the parameterization. Therefore,
\begin{equation}
	\begin{split}
		K_n
		=\frac1{\ii}\int_{\Gs}\exp\left[n(G(\muf, z)-\widehat{G})\right]\dd z
		&\geq \int_{-y_0}^{y_0}\exp\left[n\re(G(\muf, z(y))-\widehat{G})\right]\dd y\\
		&\geq 2y_0\exp\left[n\re(G(\muf, z_0)-\widehat{G})\right].
	\end{split}
\end{equation}
Here,
\begin{equation}
	\begin{split}
		\log y_0+n\re(G(\muf, z_0)-\widehat{G})
		&= \log y_0-\frac12\log(4\muf y_0)-\frac14\sum_{j=2}^n\log\left(1+\frac{(4\muf y_0)^2}{(\mu_1-\mu_j)^2}\right)\\
		&\geq \frac12\log y_0-\frac{(4\muf y_0)^2}{4}\sum_{j=2}^n\frac1{(\mu_1-\mu_j)^2}\\
		&\geq -\frac13\log n +O(\log\log n).
	\end{split}
\end{equation}
The last inequality holds a.a.s., using Lemma \ref{lem:y0} with $a_n^2=\log\log n$ and the fact  $\sum_{j=2}^n\frac1{(\mu_1-\mu_j)^2}$ is $O(n^{4/3})$ under the event $\cF_\e$. This completes the proof of the lower bound of $K_n$.

\subsection{Low temperature free energy}
Finally, using the contour integral computations from the previous section, we obtain the following lemma for the limiting fluctuations of the free energy on the low temperature side of the critical temperature window.
\begin{lemma}
If $\beta=\beta_c+bn^{-1/3}\sqrt{\log n}$ for some fixed $b\geq0$, then the free energy satisfies 
\beqq
\frac{m+n}{\sqrt{\frac16\log n}}\left(F_{n,m}(\beta)-F(\beta)+\frac{1}{12}\frac{\log n}{n+m}\right) \to \mathcal{N}(0,1)+\frac{\sqrt{6}\la^{\frac14}b}{(1+\la)^{\frac12}(1+\la^{\frac12})^{\frac23}}\TW_1,
\eeqq
where 	
\begin{equation}
F(\b)=f_\la+\frac{\la}{1+\la}A(d_+,B)-\frac12\log\b-\frac{\la}{2(1+\la)}C_\la.
\end{equation}
\end{lemma}

\begin{proof}
By \eqref{def:Sn}, 
\beqq
\frac{1}{n+m}\log Q_n= \frac{n}{n+m}\widehat{G}+\frac{1}{n+m}\log S_n.
\eeqq
Note that $\frac{1}{n+m}\log S_n =-\frac{5}{6}\frac{\log n}{n+m}+O(n^{-1}\log\log n)$ by Lemma \ref{lem:Sn}, while the quantity $\hG$ is computed in Lemma \ref{lem:Ghat}. Combining them, we get
\begin{multline}
\frac{1}{n+m}\log Q_n=\frac{\la}{1+\la}A(d_+,B) -\frac76\frac{\log n}{n+m}-\frac1{2(n+m)}\sum_{i = 1}^n\log|d_+-\mu_i|+\frac{\la^{\frac34}bn^{-\frac13}\sqrt{\log n}}{(1+\la)^{\frac32}d_+}(\mu_1-d_+)+O(\tfrac{\log\log n}{n}). 
\end{multline}
Apply this to \eqref{eqn:Fn}, we obtain
\beqq
\begin{split}
	F_{m,n}(\b)&=f_\la+\frac{\la}{1+\la}A(d_+,B)-\frac12\log\b-\frac16\frac{\log n}{n+m}
	\\&\quad 
	-\frac1{2(n+m)}\sum_{i = 1}^n\log|d_+-\mu_i|+\frac{\la^{\frac34}bn^{-\frac13}\sqrt{\log n}}{(1+\la)^{\frac32}d_+}(\mu_1-d_+)+O(\tfrac{\log\log n}{n}).
\end{split}
\eeqq
In terms of variables $T_{1n}$ and $T_{2n}$ as in \eqref{eqn:T1nT2n}, we get
\begin{equation}
\begin{split}
F_{m,n}(\b)&=f_\la+\frac{\la}{1+\la}A(d_+,B)-\frac12\log\b-\frac{\la}{2(1+\la)}C_\la-\frac1{12}\frac{\log n}{n+m}\\
&\quad +\frac{\sqrt{\frac16\log n}}{n+m}\left(T_{1n}+\frac{\sqrt{6}\la^{\frac14}b}{(1+\la)^{\frac12}(1+\la^{\frac12})^{\frac23}}T_{2n} \right)+O(\tfrac{\log\log n}{n}).
\end{split}	
\end{equation}
The theorem then follows since $T_{1n}\stackrel{d}{\to}\cN(0,1)$ by Theorem \ref{thm:CLT}, and $T_{2n}\stackrel{d}{\to}\TW_1$ by Lemma \ref{lem:TW_mu1}.
\end{proof}

The fact that the Gaussian and Tracy--Widom limits are independent in shown in the next section.

\section{Independence of Gaussian and Tracy--Widom variables (low temperature)}\label{sec:independence}

Recall the quantities 
\beq\begin{split}
&T_{1n}:=\frac{C_\la n-\frac16 \log n-\sum_{i=1}^n\log|d_+-\mu_i|}{\sqrt{\frac23\log n}},\qquad \quad
T_{2n}:=\frac{n^{2/3}(\mu_1-d_+)}{\sqrt{\la}(1+\sqrt{\la})^{4/3}},\\
&C_\la=(1-\la^{-1})\log(1+\la^{\frac12})+\log(\la^{\frac12})+\la^{-\frac12}
\end{split}\eeq
The goal of this section is to show that, given an LOE matrix $M_{n,m}$ (which we assume without loss of generality to be in tridiagonal form), with probability arbitrarily close to one,
\begin{itemize}
\item $T_{1n}=\frac{Z_n}{\sqrt{\frac23 \log n}}+o(1)$ for $Z_n$ depending only on the upper left minor of size $n-2n^{1/3}(\log n)^3$ \\of the matrix $M_{n,m}$, and
\item $T_{2n}=Y_n+o(1)$ for $Y_n$ depending only on the lower right minor of size $2n^{1/3}(\log n)^3$ of the matrix.
\end{itemize}

Our proofs draw on ideas from the paper \cite{JKOP2}, which proves a similar result in the case of Wigner ensembles.  We also make use of results from \cite{CWL_CLT}, which  studies the asymptotics of the quantity $\sum_{i=1}^n\log|\gamma-\mu_i|$ for $\g\geq d_+$ by analyzing a recurrence on the determinants of the minors of $M_{n,m}$.   In order to demonstrate the asymptotic independence of $T_{1n}$ and $T_{2n}$, we need not only the main theorem of \cite{CWL_CLT}, but also many of the intermediate lemmas which involve recurrences on the matrix entries.  For this purpose, we briefly summarize the set-up from that paper along with the key notations that are used.

Recall from \eqref{eq:tridiagonal} that the tridiagonal representation of $M_{n,m}$ depends on $\chi$-squared random variables $\{a_i^2\}$, $\{b_i^2\}$. The paper \cite{CWL_CLT} works with centered and rescaled versions of these, denoted by $\a_i$ and $\b_i$ respectively, which are defined as
\beq\label{eq:alpha_beta_def}
\a_i=\frac{a_i^2-(m-n+i)}{|\rho_i^+|},
\qquad\b_i=\frac{b_{i-1}^2-(i-1)}{|\rho_i^+|}.
\eeq
Here, the scaling factor $\rho_i^+$ is one of the characteristic roots of the recurrence on determinants of the minors of $M_{n,m}$.  This turns out to be a convenient rescaling since it prevents the iterates from blowing up.  More precisely,
\beq\label{eq:rho_def}
\rho_i^\pm:=-\frac12\left(\gamma m-(m-n+2i-1)
\pm\sqrt{(\gamma m-(m-n+2i-1))^2-4(m-n+i-1)(i-1)}
\right).
\eeq
Throughout the proofs, we will also use the notations
\beq\label{eq:tau_delta_def}
\tau_i=\frac{m-n+i}{|\rho_i^+|},
\qquad\d_i=\frac{i-1}{|\rho_i^+|}.
\eeq

\subsection{Proof for $T_{1n}$}
\begin{lemma}
There exists a random random $Z_n$, depending only on the upper left minor of size $n-2n^{1/3}(\log n)^3$ of the matrix $M_{n,m}$ such that
\beqq
T_{1n}=\frac{Z_n}{\sqrt{\frac23\log n}}+o(1).
\eeqq
\end{lemma}
\begin{proof}
We begin our analysis of $T_{1n}$ by remarking that it is tricky to analyze the distribution of $\sum_{i=1}^n\log|d_+-\mu_i|$ directly because of how close $d_+$ is to the eigenvalues $\{\mu_i\}$.  For this reason, \cite{CWL_CLT} uses the technique of first analyzing the sum $\sum_{i=1}^n\log|\g-\mu_i|$ for
\beq\label{eqn:gamma}
\g=d_++\sigma_n n^{-2/3},
\eeq
then analyzing the original sum by comparison to the shifted one.  We employ a similar technique here.  More precisely, we take
\beq
\sigma_n=\bar{\sigma}_n:=\left(\log\log n\right)^3.
\eeq
 From line (7.3) of \cite{CWL_CLT}, we have
\beq\label{eq:CLT_shiftedsum}
\sum_{i=1}^n\log|d_+-\mu_i|
=\sum_{i=1}^n\log|d_++\bar{\sigma}_nn^{-2/3}-\mu_i|
-C_1\bar{\sigma}_n n^{1/3}+C_2\bar{\sigma}_n^{3/2}+o(\sqrt{\log n}).
\eeq
where
\beq
C_1=\frac{1}{\la^{1/2}(1+\la^{1/2})},\qquad C_2=\frac{2}{3\la^{3/4}(1+\la^{1/2})^2}.
\eeq
Furthermore, from Lemma 3.1 and Section 4 of \cite{CWL_CLT}, we can rewrite the sum on the righthand side of \eqref{eq:CLT_shiftedsum} as
\beq\label{eq:sum_intermsof_Li}
\sum_{i=1}^n\log|d_++\bar{\sigma}_nn^{-2/3}-\mu_i|
=C_\la n-\sum_{i=3}^n L_i -\frac16\log n+C_1\bar{\sigma}_n n^{1/3}-C_2\bar{\sigma}_n^{3/2}+o(\sqrt{\log n})
\eeq
where $C_1,C_2$ are the same constants from \eqref{eq:CLT_shiftedsum} and $L_i$ is given by the recursive formula
\beq\label{eq:Li_def}
L_i:=\xi_i+\omega_i L_{i-1} \text{ for }i\geq4,\qquad L_3:=\xi_3.
\eeq
with
\beq\label{eq:xi_def}
\xi_i:=\a_i+\b_i(1+\tau_{i-1})+\a_{i-1}\d_i, \qquad
\omega_i:=\tau_{i-1}\d_i.
\eeq
Thus, combining \eqref{eq:CLT_shiftedsum} and \eqref{eq:sum_intermsof_Li} with the definition of $T_{1n}$, we get
\beq
T_{1n}=\frac{\sum_{i=3}^nL_i}{\sqrt{\frac23\log n}}+o(1).
\eeq
It remains to show that $\sum_{i=3}^n L_i=Z_n+o(\sqrt{\log n})$ for some $Z_n$ depending only on the upper left minor of $M_{n,m}$ of size $n-2n^{1/3}(\log n)^3$.  From the recursive definition of $L_i$, we have, for any $j\geq 4$,
\beqq
\sum_{i=3}^nL_i=\sum_{i=3}^{n}\xi_i+\w_i\xi_{i-1}+\cdots+\w_i\cdots\w_4\xi_3 
=\sum_{i=3}^{n}g_{i+1}\xi_i
\eeqq
where $g_i = 1+\w_i+\w_i\w_{i+1}+\dots +\w_i\dots\w_n$ for $3\leq i \leq n$.  Now we would like to compare this sum to a similar sum, truncated at index $i=n-2n^{1/3}(\log n)^3$ and show that their difference is small, with probability arbitrarily close to 1. As this will involve computing the variance of the difference between the sums, we would like to eliminate the dependence between consecutive terms in the sum by rewriting 
\beqq
\sum_{i = 3}^n L_i = \sum_{i = 3}^n g_{i+1}X_i + \sum_{i = 3}^n \a_i - g_3\a_2.
\eeqq
where 
\begin{equation}\label{defn:Xi}
X_i = (1 + \tau_{i-1})(\delta_i \alpha_{i-1} + \beta_i), \quad 3 \leq i \leq n.
\end{equation}
Now we define
\beq
Z_n=\sum_{i=3}^{\lfloor n-2n^{1/3}(\log n)^3\rfloor}g_{i+1}X_i.
\eeq
This gives us
\beq\label{eq:Li-Zn}
\sum_{i=3}^nL_i-Z_n=\sum_{i=\lceil n-2n^{1/3}(\log n)^3\rceil}^ng_{i+1}X_i+\sum_{i=3}^n\a_i-g_3\a_2.
\eeq
It follows from line (5.21) of \cite{CWL_CLT} that $\sum_{i=3}^n\alpha_i-g_3\a_2=o(\sqrt{\log n})$ with probability $1-n^{-1/2}$.  Finally, we bound the variance of the remaining sum on the right hand side of \eqref{eq:Li-Zn}.  Since $\{X_i\}$ are pairwise independent and $\{g_i\}$ are deterministic, we have 
\beqq
\EE\Big[\Big(\sum_{i=\lceil n-2n^{1/3}(\log n)^3\rceil}^ng_{i+1}X_i\Big)^2\Big]=
\sum_{i=\lceil n-2n^{1/3}(\log n)^3\rceil}^ng_{i+1}^2\EE X_i^2
\eeqq
From (4.42) of \cite{CWL_CLT}, we have $\EE X_i^2=O(n^{-1})$ uniformly in $i$.  Combining Lemma 5.1 and Corollary 2.9 of \cite{CWL_CLT}, we have
\beq
g_i=\begin{cases}O(n^{1/2}(n-i)^{-1/2}) & i\leq n-n^{1/3}\sigma_n\\
O(n^{1/3}\sigma_n^{-1/2}) & i\geq n-n^{1/3}\sigma_n.
\end{cases}
\eeq
Thus, we can bound the sum as follows:
\beqq\begin{split}
\sum_{i=\lceil n-2n^{1/3}(\log n)^3\rceil}^ng_{i+1}^2\EE X_i^2
&\leq\sum_{i=\lceil n-2n^{1/3}\sigma_n(\log n)^3\rceil}^{\lfloor n-n^{1/3}\sigma_n\rfloor}\frac{n}{n-i}\cdot\frac{C}{n}
+\sum_{i=\lceil n-n^{1/3}\sigma_n\rceil}^n\frac{n^{2/3}}{\sigma_n}\cdot\frac{C}{n}\\
&=O(\log\log n)+O(1).
\end{split}\eeqq
This completes the proof of the lemma concerning $T_{1n}$.
\end{proof}

\subsection{Proof for $T_{2n}$}
We now verify that, $T_{2n}=Y_n+o(1)$, for some random variable $Y_n$ depending only on the bottom-right minor of size $2n^{\frac13}(\log n)^3$ of the matrix $M_{n,m}$ (in fact, we get a much tighter tail bound than $o(1)$). 
Recall that $T_{2n}$ is a shifted re-scaling of the largest eigenvalue $\mu_1$, and it converges to the Tracy--Widom distribution. Thus, $Y_n$, if it exists, must converge to the same limit, while only depending on the bottom corner of $M_{n,m}$. The following lemma shows that the largest eigenvalue of the minor described above, with the same transformation as in $T_{2n}$, is a good choice for $Y_n$.

\begin{lemma}\label{lem:diff_mu}
	Let $\widetilde{\mu}_1$ be the largest eigenvalue of the bottom-right minor of $M_{n,m}$ of size $p>2n^{\frac13}(\log n)^3$. Then, for any $D>0$ and $\e>0$, with probability at least $1-\e$, 
	\[
	|\mu_1-\widetilde{\mu}_1| = O(n^{-D}).
	\]
	Furthermore, by setting $Y_n=\frac{n^{2/3}(\widetilde{\mu}_1-d_+)}{\sqrt{\la}(1+\sqrt{\la})^{4/3}}$ and taking $D>\frac23$ arbitrarily large, we have
	\[
	T_{2n}=Y_n+O(n^{-D+2/3}).
	\]
\end{lemma}
The key ingredient to bounding the difference $\mu_1-\widetilde{\mu}_1$ lies in controlling the first $n-2n^{1/3}(\log n)^3$ components of an eigenvector corresponding to $\mu_1$. In particular, we need the following result.
\begin{lemma}\label{lem:vi}
	If $\bv=(v_1,\dots,v_n)^T$ is a principal eigenvector of $M_{n,m}$, then for any $\e>0$ and $d>0$, with probability at least $1-\e$, we have
	\[
	\max_{j\leq n-2n^{\frac13}(\log n)^3}\frac{|v_j|}{\|\bv\|} < n^{-d}.
	\]	
\end{lemma}
Lemma \ref{lem:vi} itself relies on the following two auxiliary Lemmas \ref{lem:Fi} and \ref{lem:ratio}, both of which depend on the random entries in the tridiagonal matrix form. We include their proofs in the Appendix \ref{sec:appendix2}. 
\begin{lemma}\label{lem:Fi}
	Let $\mu_1$ be the largest eigenvalue of $M_{n,m}$. Let $\{F_j\}_{j=1}^{n-1}$ be the sequence given by
	\beqq
	F_1=-1+\frac{\mu_1 m-a_1^2}{|\rho_1^+|}, \quad F_j=-1+ \frac{\mu_1m-(a_j^2+b_{j-1}^2)}{|\rho_j^+|}+\frac{(a_{j-1}b_{j-1})^2}{|\rho_j^+||\rho_{j-1}^+|}\cdot\frac{1}{1+F_{j-1}} \text{ for } j=2,\dots, n-1.
	\eeqq
	Here, $\rho_j^+$ is given by \eqref{eq:rho_def} with $\gamma=d_+$. 
	Then, for every $\e>0$, with probability at least $1-\e$,
	\begin{equation}\label{eqn:maxRi}
		\max_{j\leq n-n^{\frac13}(\log n)^3}|F_j| = o(n^{-\frac13}).
	\end{equation}	
\end{lemma}
\begin{lemma}\label{lem:ratio}
Given $\e>0$, then for sufficiently large $n$ and $a_i,b_i$ as defined in \eqref{eqn:aibi}, we have 
\beq
\bP\left(\left|\max_{j\leq n-n^{1/3}(\log n)^3}a_jb_j-\sqrt{mn}\right|\leq (e\log n)^2n^{1/2}\right)\geq 1-\e.
\eeq
\end{lemma}
\begin{proof}[Proof of Lemma \ref{lem:vi}]
	From the tridiagonal representation \eqref{eq:tridiagonal} and the notations presented at the beginning of Section \ref{sec:independence}, we obtain the system of linear equations
	\begin{equation*}
		\begin{cases}
			\left(\frac{a_1^2}m-\mu_1\right)v_1+\frac{a_1b_1}mv_2=0,&\quad\\
			\frac{a_{j-1}b_{j-1}}mv_{j-1}+\left(\frac{a^2_j+b^2_{j-1}}m-\mu_1\right)v_j+\frac{a_jb_j}mv_{j+1}=0, &\quad j=2,\dots, n-1.
		\end{cases}
	\end{equation*}
	With probability 1, $a_j>0$ and $b_j>0$ for $j=1,\dots, n-1$. This implies $v_1\neq 0$ (otherwise, $\bv$ is the zero vector). In fact, as functions of positive, continuous random variables $a_1, \dots, a_{j-1},b_1,\dots, b_{j-1}$, it holds with probability 1 that $v_j\neq0$ for each $j$. Thus, we rescale $\bv$ to have $v_1=1$ and obtain
	\begin{equation}
		v_2=\frac{\mu_1m-a_1^2}{a_1b_1}, \quad v_{j+1}=\frac{\mu_1 m-(a_j^2+b_{j-1}^2)}{a_jb_j}v_j-\frac{a_{j-1}b_{j-1}}{a_jb_j}v_{j-1}, \quad j=2,\dots, n-1. 
	\end{equation}
	We introduce the following quantity
	\beq
	F_j=\frac{v_{j+1}}{v_j}\cdot\frac{a_jb_j}{|\rho_j^+|}-1, \quad \text{for } j=1, \dots, n-1.
	\eeq 
	Here, $\rho_j^+$ is given in \eqref{eq:rho_def} with $\g=d_+$.
	Set $k = \lceil n^{\frac13} \rceil$, and let $j\leq n-2n^{1/3}(\log n)^3$. Observe that 
	\begin{equation}\label{eqn:vi}
		\frac{|v_j|}{\|\bv\|}\leq \left|\frac{v_j}{v_{j+k}}\right|=\prod_{l=j}^{j+k-1}(1+F_l)^{-1}\prod_{l=j}^{j+k-1}\frac{(a_lb_l)/m}{|\rho_l^+|/m}.
	\end{equation}
	Since $\{F_l\}_{l=1}^{n-1}$ satisfies the hypothesis of Lemma \ref{lem:Fi} and each $l\in[j,j+k-1]$ satisfies $l\leq n -n^{\frac13}(\log n)^3$, it follows that, with probability $1-\e/2$, we have $\prod_{l=j}^{j+k-1}(1+F_l)^{-1} = 1+o(1)$.
	
	We then consider the product $\prod_{l=j}^{j+k-1}\frac{(a_lb_l)/m}{|\rho_l^+|/m}$. As $|\rho_l^+|$ is decreasing in $l$ by \eqref{eq:rho_def},
	\begin{equation}\label{eqn:ratio_rho}
		\begin{split}
			\frac{|\rho_l^+|}m
			&\geq \frac{\left|\rho_{n-n^{1/3}(\log n)^3}^+\right|}m\\
			&=\frac{d_+ m-(m+n-2n^{1/3}(\log n)^3-1)}{2m}\left(1+\sqrt{1-\frac{4(m-n^{1/3}(\log n)^3-1)(n-n^{1/3}(\log n)^3-1)}{(d_+ m-(m+n-2n^{1/3}(\log n)^3-1))^2}}\right).
		\end{split}
	\end{equation}
	Using $d_+ m= m+n+2\sqrt{mn}$, the first factor on the right hand side of \eqref{eqn:ratio_rho} is $\sqrt{\la}(1+O(n^{-\frac23}(\log n)^3))$, while the expression under the square root is $\Theta(n^{-\frac23}(\log n)^3)$. Therefore, there is a constant $c>0$ such that
	\[
	\frac{|\rho_l^+|}m \geq \sqrt{\la}+cn^{-\frac13}(\log n)^{\frac32} \quad \text{ for all } l\leq n-n^{\frac13}(\log n)^3.
	\] 
	Combining this with Lemma \ref{lem:ratio}, we obtain that, for some $c'>0$, with probability $1-\e/2$,
	\beq
	\prod_{l=j}^{j+k-1}\frac{(a_lb_l)/m}{|\rho_l^+|/m}\leq (1-c'n^{-\frac13}(\log n)^{\frac32}+o(n^{-\frac13}))^k.
	\eeq
	Therefore, with probability $1-\e$,
	\begin{align*}
		\max_{j\leq n-2n^{\frac13}(\log n)^3}\frac{|v_j|}{\|\bv\|}\leq \max_{j\leq n-2n^{\frac13}(\log n)^3}\left|\frac{v_j}{v_{j+k}}\right|=\exp\left(-c'(\log n)^{3/2}+o(1)\right).
	\end{align*}
	The above quantity is $O(n^{-\log^{1/2} n+o(1)})$, smaller than any $n^{-d}$ for sufficiently large $n$. This completes the proof of Lemma \ref{lem:vi}.
\end{proof}

We now have the necessary tools to prove Lemma \ref{lem:diff_mu} and conclude our argument of asymptotic independence.
\begin{proof}[Proof of Lemma \ref{lem:diff_mu}]
	We observe that $\widetilde{\mu}_1$ is equal to the largest eigenvalue of $M^{(p)}_{n,m}$ where
	\begin{equation}
		mM^{(p)}_{n,m}= \begin{bmatrix}
			0&0&&&&&\\
			0&\ddots&\ddots&&&&\\
			&\ddots&0&0&&&\\
			&&0&a^2_{n-p+1}+b^2_{n-p}&a_{n-p+1}b_{n-p+1}&&\\
			&&&a_{n-p+1}b_{n-p+1}&\ddots&\ddots&\\
			&&&&\ddots&\ddots&a_{n-1}b_{n-1}\\
			&&&&&a_{n-1}b_{n-1}&a^2_{n}+b^2_{n-1}\\
		\end{bmatrix}.
	\end{equation}
	This implies $\mu_1\geq\widetilde{\mu}_1$. We now verify the upper bound on $\mu_1-\widetilde{\mu}_1$.
	
	Set $\bv=(v_1,\dots,v_n)^T$ to be a normalized principal eigenvector, i.e. $\bv$ is a unit vector satisfying $\bv^TM_{n,m}\bv=\mu_1$. Since $\widetilde{\mu}_1\geq \bv^TM^{(p)}_{n,m}\bv$, it follows that, for $\bv_{:n-p}:=(v_1,\dots,v_{n-p})^T$ and 
	\begin{equation}
		(m-p)M_{n-p,m-p}= 
		\begin{bmatrix}
			a_1^2&a_1b_1&&&\\
			a_1b_1&a_2^2+b_1^2&a_2b_2&&\\
			&a_2b_2&a_3^2+b_2^2&\ddots&\\
			&&\ddots&\ddots&a_{n-l-1}b_{n-l-1}\\
			&&&a_{n-p-1}b_{n-p-1}&a_{n-p}^2+b_{n-p-1}^2
		\end{bmatrix},
	\end{equation}
	we have
	\begin{equation}\label{eqn:diff_mu}
		\begin{split}
			\mu_1-\widetilde{\mu}_1 
			&\leq \bv^T\left(M_{n,m}-M^{(p)}_{n,m}\right)\bv\\
			&=\frac{m-p}{m}\bv_{:n-p}^TM_{n-p,m-p}\bv_{:n-p}+\frac{2a_{n-p}b_{n-p}}{m}v_{n-p+1}v_{n-p}.
		\end{split}
	\end{equation}
	
	As $p>2n^{\frac13}(\log n)^3$, Lemma \ref{lem:vi} implies that for any $d>0$ and $\e>0$, then with probability $1-\e/3$, $\|\bv_{:n-p}\|^2=O(n^{-2d+1})$ and $\max{|v_{n-p+1}|, |v_{n-p}}|=O(n^{-d})$. Furthermore, $\frac{2}{m}a_{n-p}b_{n-p}=O(1)$ by Lemma \ref{lem:ratio}, and $\|M_{n-p,m-p}\|=O(1)$ (due to being a rescaled LOE matrix), and each of these $O(1)$ bounds holds with probability $1-\e/3$. Therefore, \eqref{eqn:diff_mu} implies $\mu_1-\widetilde{\mu}_1 = O(n^{-2d+1})$ with probability $1-\e$. Setting $d=\frac12(D+1)$, we obtain the first statement of Lemma \ref{lem:diff_mu}. The second one then follows immediately from the observation  $T_{2n}-Y_n=\Theta(n^{2/3}(\mu_1-\widetilde{\mu}_1))$.
\end{proof}

\begin{appendix}
\section{Appendix: Section 2 proofs}\label{sec:appendix} 
	In this appendix, we provide a proof for Lemma \ref{lem:Aj}, and then apply it to prove Lemma \ref{lem:diff_stieltjes_gK}.
\subsection{Proof of Lemma \ref{lem:Aj}}
	Before beginning the main proof, we need the following preliminary results.
\begin{lemma}\label{lem:counting}
	Let $\{\mu_j\}_{j=1}^n$ be eigenvalues of a scaled LUE or LOE matrix $\frac1m M_{n,m}$. Assume $s$ is such that $s>C$ for some $C>0$ and $s=o(n^{2/3})$ as $n\to\infty$. The following statements hold for $\cN_s:=\#\{i: \mu_i \in [d_+-sn^{-2/3},\infty) \}$.
	\begin{align}
		\bE\cN_s&=\frac{2}{3\pi\la^{3/4}d_+}s^{3/2}+O(s^{5/2}n^{-2/3}).\label{eqn:bE_Ns}\\
		\Var(\cN_s)&=\frac{3}{4\pi^2}\log(s)(1+o(1)). \label{eqn:Var_Ns}
	\end{align}
\end{lemma}	

The lemma is the analog of Proposition 6.5 from \cite{LandonSosoe}, which bounds the expectation and variance of the counting function in the case of GOE matrices. There, the result was obtainted by applying the corresponding result for GUE matrices by Gustavsson \cite{Gustavsson}, and the relation between eigenvalues of Gaussian orthogonal and unitary ensembles in \cite{ForresterRains}. The proof of \cite{LandonSosoe} works in our case, up to translating from Gaussian to Laguerre ensembles. For completeness, we reproduce it here, first proving for LUE matrices using a result in \cite{Su06}, then extend to LOE matrices using the following result. 
\begin{theorem}[Theorem 5.2 of \cite{ForresterRains}]\label{thm:ForresterRains}
	For independent eigenvalue point processes $\text{LOE}_{n,m}$, $\text{LOE}_{n+1,m+1}$, 
	\[
	\text{even}(\text{LOE}_{n,m}\cup\text{LOE}_{n+1,m+1})=\text{LUE}_{n,m},
	\]
	where the notation $\text{even}(\cdot)$ denotes the set containing only the even numbered elements among the ordered list of elements in the original set.
\end{theorem}

\begin{proof}[Proof of Lemma \ref{lem:counting}] 
	In the case of LUE matrix, the lemma follows from the results of Su in \cite{Su06}. Namely, the first inequality holds by Lemma 1 of \cite{Su06}, which states that 
	\begin{align*}
	\bE\#\{j:\mu_j\in [t_n,\infty)\}
	&=n\int_{t_n}^{\beta_{n,m}}p_{\MP}(x)dx\\
	&= \frac{\sqrt{\beta_{n,m}-\alpha_{n,m}}}{3\pi \beta_{n,m}}n(\beta_{n,m}-t_n)^{3/2}+ O(n(\beta_{n,m}-t_n)^{5/2}).
	\end{align*}
	As the matrix in \cite{Su06} is scaled by $1/n$ instead of $1/m$ as in this paper, our interval of interest $[d_+-sn^{-2/3}, \infty)$ corresponds to $t_n=\beta_{n,m}-\frac{s}\la n^{-2/3}$ in \cite{Su06}. 
	Meanwhile, the inequality for variance directly follows from Lemma 4 there.
	
	We now consider the case of LOE matrix. Let $M^{(1)}_{n,m}$, $M^{(1)}_{n+1,m+1}$ be independent LOE matrices of size $n\times m$ and $(n+1)\times(m+1)$, respectively, and let $M^{(2)}_{n,m}$ be a LUE matrix of size $n\times m$. Set $X^{(2)}_{n,m}$ to be the number of eigenvalues of $M^{(2)}_{n,m}$ that are at least $m\left(d_+-sn^{-\frac23}\right)$. We define $X^{(1)}_{n,m}$ and $X^{(1)}_{n+1,m+1}$ similarly, for the two LOE matrices.
	Theorem \ref{thm:ForresterRains} implies that there is a random variable $Y$ and a random variable $Z\in[0,1]$ such that 
	\[
	X^{(2)}_{n,m}\stackrel{d}{=}Y, \quad Y-Z=\frac12\left(X^{(1)}_{n,m}+X^{(1)}_{n+1,m+1}\right).
	\]
	The estimates \eqref{eqn:bE_Ns} and \eqref{eqn:Var_Ns} holds for $Y$ by the previous paragraph. The estimate \eqref{eqn:Var_Ns} for $Y$, together with boundedness of $Z$ and the fact $X^{(1)}_{n,m}$ and $X^{(1)}_{n+1,m+1}$ are independent, imply that \eqref{eqn:Var_Ns} holds for the $X^{(1)}$'s as well. Now, 
	\begin{equation}\label{eqn:bE}
		\bE[X^{(2)}_{n,m}]=\frac12\left(\bE[X^{(1)}_{n,m}]+\bE[X^{(1)}_{n+1,m+1}]\right)+c, \quad \text{for some } c\in[0,1].
	\end{equation}
	From the tridiagonal form of Laguerre ensembles, the top left $n\times n$ minor of $M^{(1)}_{n+1,m+1}$ has the same distribution as $M^{(1)}_{n,m}$. The eigenvalues of this minor interlace those of $M^{(1)}_{n+1,m+1}$, which implies there is a random variable $\tilde{X}^{(1)}_{n,m}$ with the same distribution as $X^{(1)}_{n,m}$ and satisfies
	\[
	|\tilde{X}^{(1)}_{n,m}-X^{(1)}_{n+1,m+1}|\leq 1.
	\]
	We then obtain \eqref{eqn:bE_Ns} for $X^{(1)}_{n,m}$ and $X^{(1)}_{n+1,m+1}$, using \eqref{eqn:bE_Ns} for $X^{(2)}_{n,m}$, \eqref{eqn:bE} and the above inequality. 
\end{proof}
 We now have the needed tools to prove Lemma \ref{lem:Aj}.

\begin{proof}[Proof of Lemma \ref{lem:Aj}]
For $j=1,\dots, n^{2/5}$ and $t>0$, by definition,
	\begin{equation}\label{eqn:PAj_from_PNT}
		\begin{split}
			\bP(A_j\geq t)=\bP\left(\mu_j\geq d_+-\left(\left(C^\star j\right)^{2/3}-t\right)n^{-2/3}\right)=	\bP(\cN_T\geq j),	
		\end{split}
	\end{equation}
	where $C^\star=\tfrac32 \pi\la^{3/4}d_+$ and $T=T(j,t):=\left(C^\star j\right)^{2/3}-t$. If $\bE\cN_T<j$, then 
	\beq\label{eqn:NT_Markov}
	\bP(\cN_T\geq j)\leq\bP(|\cN_T-\bE\cN_T|\geq j-\bE\cN_T)\leq \frac{\Var\cN_T}{( j-\bE\cN_T)^2}.
	\eeq
In order to make use of this inequality, we need to know what values of $t$ (depending on $j$) satisfy $\bE\cN_T<j$. By Lemma \ref{lem:counting}, there exist $K,c_0>0$ such that, for any $c_1>0$ and any sufficiently large $n$, if $K\leq j\leq n^{2/5}$ and $0<t<(C^\star j)^{2/3}-c_1$, then
\beq\label{eqn:j-EN_bound}\begin{split}
j-\EE\cN_T&\geq j-\tfrac{1}{C^\star}((C^\star j)^{2/3}-t)^{3/2}-c_0j^{5/3}n^{-2/3}\\
&\geq j-j\left(1-\frac{t}{(C^\star j)^{2/3}}\right)^{3/2}-c_0\; \geq\; \frac{tj^{1/3}}{(C^\star)^{2/3}}-c_0.
\end{split}\eeq
In particular, this means that $\EE\cN_{T}<j$ is satisfied (along with the conditions of Lemma \ref{lem:counting}) when $c_0(C^\star)^{2/3} j^{-1/3}<t<(C^\star j)^{2/3}-c_1$ and $K\leq j\leq n^{2/5}$ (note that one should choose $K>c_0$).
Thus, for $t,j$ satisfying these conditions, we combine \eqref{eqn:PAj_from_PNT}-\eqref{eqn:j-EN_bound} with the variance bound from Lemma \ref{lem:counting} to conclude that, for some $c_2>0$ and sufficiently large $n$,
\beq
\bP(A_j\geq t)\leq \frac{c_2 \log j}{((C^\star)^{-2/3}tj^{1/3}-c_0)^2}.
\eeq
Next, taking $T'=\left(C^\star j\right)^{2/3}+t$ we can follow the same argument to bound $\PP(A_j<-t)$.  This time, we find that $\bE\cN_{T'}\geq j$ is satisfied (along with the conditions of Lemma \ref{lem:counting}) when $c_0(C^\star)^{2/3} j^{-1/3}<t\ll n^{2/3}$ and $K\leq j\leq n^{2/5}$.  Then, for $t,j$ satisfying these conditions,  and for some $c_3>0$ with sufficiently large $n$, 
	\beq
	\begin{split}
		\bP(A_j\leq -t)
		&=\bP(\mu_j <d_+-T'n^{-2/3})\\
		&\leq \bP(|\cN_{T'}-\bE\cN_{T'}|> \bE\cN_{T'}-j)\leq c_3\frac{\log j+\log(1+t)}{((C^\star)^{-2/3}tj^{1/3}-c_0)^2}.
	\end{split}
	\eeq
	Thus, for $j,t$ satisfying $K\leq j\leq n^{2/5}$ and $c_0(C^\star)^{2/3} j^{-1/3}<t<(C^\star j)^{2/3}-c_1$, we have	
\beqq
	\bP(|A_j|\geq t)=O\left(\frac{\log j+\log(1+t)}{((C^\star)^{-2/3}tj^{1/3}-c_0)^2}\right).
	\eeqq
Taking $t=\la j^{2/3}$, then for all $k\geq K$,
	\beqq
	\bP\left(\bigcup_{k\leq j\leq n^{2/5}}\left\{|A_j|\geq \la j^{2/3}\right\}\right)=O\left(\sum_{j=k}^{n^{2/5}}\frac{\log j}{j^2}\right)=O\left(\frac{\log k}{k}\right).
	\eeqq	
This bound holds uniformly for $K\leq k\leq n^{2/5}$.  Taking $k\to\infty$ (for example $k=n^{1/5}$), we obtain \eqref{eqn:probAj}.
	
It remains to prove the second part of the lemma.  Set $t^*=c_0(C^\star)^{2/3}j^{-1/3}$.  For $K\leq j\leq n^{2/5}$, we have 
	\begin{align*}
		\bE\left[\mathbbm{1}_{\{n^{2/3}(\mu_j-d_+)\leq-C\}}\left|A_j\right|\right] 
		&\leq\int_0^{\infty}\bP(A_j\geq t)\dd t+\int_0^{\infty}\bP(-A_j\leq -t)\dd t\\
		&\leq \left(t^*+\int_{t^*}^{(C^\star j)^{2/3}-C}\bP(A_j>t)\dd t+0\right)+\left(t^*+\int_{t^*}^{n^{\frac23-\delta}}\bP(-A_j\leq -t)\dd t+o(n^{-1})\right)\\
		&\leq 2t^*+C' \int_{t^*}^{\infty}\frac{\log j +\log(1+t)}{((C^\star)^{-2/3}tj^{1/3}-c_0)^2}\dd t\\
		&\leq 2t^*+C''\frac{\log j }{j^{1/3}} = O\left(\frac{\log j}{j^{1/3}}\right),
	\end{align*} 
where, in the second line, we obtained $\int_{(C^\star j)^{2/3}-C}^\infty\PP(A_j\geq t)\dd t=0$ from the indicator in the expectation, and $\int_{n^{\frac23-\delta}}^\infty\PP(-A_j\leq-t)\dd t=o(n^{-1})$ from eigenvalue rigidity.
	\end{proof}

\subsection{Proof of Lemma \ref{lem:diff_stieltjes_gK}}
We observe that 
\beq
S_2:=\frac1n\sum_{i=K+1}^n\frac1{(z-\mu_j)^l}-\int_{d_-}^{g_K} \frac{1}{(z-y)^l}p_{\MP}(y)\dd y
=\sum_{i>K}\int_{g_i}^{g_{i-1}}\frac{(z-y)^l-(z-\mu_i)^l}{(z-\mu_i)^l(z-y)^l}p_{\MP}(y)\dd y.
\eeq
The modulus of this sum satisfies
		\beqq\label{eq:S2_bound}\begin{split}
		|S_2|
		&\leq\sum_{i>K}\int_{g_i}^{g_{i-1}}\frac{l\max\{|z-y|,|z-\mu_i|\}^{l-1} |\mu_i-y|}{|z-\mu_i|^l|z-y|^l}p_{\MP}(y)\dd y\\
		&\leq\sum_{i>K}\int_{g_i}^{g_{i-1}}\frac{l\cdot |(\mu_i-g_i)+(g_i-y)|}{\min\{|z-\mu_i|,|z-y|\}^{l+1}}p_{\MP}(y)\dd y.
		\end{split}\eeqq	
		We now split the sum as $S_{21}+S_{22}$, summing over $K\leq i\leq n^{2/5}$ and $i>n^{2/5}$, respectively. First, consider $K\leq i\leq n^{2/5}$. By Lemma \ref{lem:Aj}, given $\e>0$, on the event $\event$, there exists $c>0$ such that, for sufficiently large $n$, $n^{2/3}(d_+-\mu_i)\geq ci^{2/3}$ uniformly for all $i$ in this range. Combining with the facts that $\re z\geq d_+$ and $d_+\geq \mu_i$ for $i\geq K$ on $\event$, we have
		\begin{equation}\label{eqn:x-eval}
			n^{2/3}|z-\mu_i|\geq \max\{n^{2/3}|z-d_+|,\; ci^{2/3}\}.
		\end{equation}
		Meanwhile, there exists $C>0$, independent of $n$, such that $C^{-1}i^{2/3}\leq n^{2/3}(d_+-g_i)\leq Ci^{2/3}$ for all $i$ (see, for example, \cite{BaikLeeSSK}). Thus, \eqref{eqn:x-eval} also holds for $n^{2/3}|z-y|$, uniformly for $y\in(g_i,g_{i-1})$. 
		For the numerator, we have $n^{2/3}(y-g_i) \leq n^{2/3}(g_{i-1}-g_i)\leq ci^{-1/3}$, using
		\beqq
		\frac{1}{n} = \int_{g_i}^{g_{i-1}}p_{\MP}(y)\dd y\geq c\sqrt{d_+-g_i}(g_{i-1}-g_i).
		\eeqq
		By \eqref{eqn:classical_loc}, $n^{2/3}(\mu_i-g_i)= A_i+O\left(\frac{i^{4/3}}{n^{2/3}}\right)$, where $A_i$ is given in \eqref{eqn:Aj}.  
		The term $i^{-1/3}$ is of larger order than $n^{-2/3}i^{4/3}$ when $K\leq i \leq n^{2/5}$, and they have the same order when $i=\Theta(n^{2/5})$. Thus,
		\beq
			\mathbbm{1}_{\event}\frac{l\cdot|(\mu_i-g_i)+(g_i-y)|}{\min\{|z-\mu_i|,|z-y|\}^{l+1}}
			\leq Cln^{\frac23 l}\frac{i^{-1/3}+|A_i|}{i^{\frac23(l+1)}+(n^{\frac23}|z-d_+|)^{l+1}}, \quad K\leq i\leq n^{2/5}.
		\eeq
		By Lemma \ref{lem:Aj}, 
		\beq\label{eq:S12_bound}\begin{split}
			\bE\left[\mathbbm{1}_{\event}|S_{21}|\right]
			&\leq Cln^{\frac23 l-1}\sum_{K\leq i\leq n^{2/5}}\frac{i^{-1/3}+\bE\left[\mathbbm{1}_{\event}|A_i|\right]}{i^{\frac23(l+1)}+(n^{\frac23}|z-d_+|)^{l+1}}
			\\&
			\leq C'ln^{\frac23 l-1}\sum_{K\leq i\leq n^{2/5}}\frac{i^{-1/3}\log i}{i^{\frac23(l+1)}+(n^{\frac23}|z-d_+|)^{l+1}}.
			\end{split}\eeq
Next, we consider two separate cases and conclude that, for some $C''>0,$
\beqq
\bE\left[\mathbbm{1}_{\event}|S_{21}|\right]\leq\begin{cases} C''n^{\frac23 l-1}\frac{\log(n^{2/3}|z-d_+|)}{(n^{2/3}|z-d_+|)^{l}} 
& K^{2/3}<n^{2/3}|z-d_+|,\\
C''n^{\frac23 l-1} & K^{2/3}\geq n^{2/3}|z-d_+|.\end{cases}
\eeqq
The bound in the first case is obtained by evaluating the right hand side of \eqref{eq:S12_bound} separately for $i^{2/3}<n^{2/3}|z-d_+|$ and $i^{2/3}>n^{2/3}|z-d_+|$. The bound in the second case follows from the convergence of $\sum_{i=K}^\infty i^{-\frac23 l-1}\log i$ for all $l\geq1$.  Thus, we obtain
		\beq
		\bE\left[\mathbbm{1}_{\event}|S_{21}|\right]=O\left(n^{\frac23 l-1}\cdot\min\left\{\left|\frac{\log(n^{2/3}|z-d_+|)}{(n^{2/3}|z-d_+|)^l}\right|,\;1\right\}\right).
		\eeq
		
		Lastly, for $S_{22}$, we bound the numerator (which is $l\cdot|\mu_i-y|$) using rigidity and bound $n^{2/3}|z-y|\geq ci^{2/3}$ by \eqref{eqn:x-eval}, while $|z-\mu_i| \geq\max\{ |z-d_+|,\;d_+-\mu_i\}$, where, with high probability,
		\beqq
		d_+-\mu_i \geq \begin{cases}
			c>0, &\quad i>n/2,\\
			ci^{2/3}n^{-2/3}, &\quad n^{2/5}<i<n/2, \text{ using rigidity with $\delta<\frac2{15}$ and \eqref{eqn:classical_loc}}.
		\end{cases}
		\eeqq
		We obtain
		\beq\begin{split}
			\mathbbm{1}_{\event}|S_{22}|
			&\leq Cln^{\frac23 l-1+\delta}\sum_{i>n^{2/5}}\frac{1}{i^{\frac23(l+1)}\min\{i^{1/3}, (n+1-i)^{1/3}\}}\\
			&\leq C'ln^{\frac23 l-1+\delta}\left(\sum_{i= n^{2/5}}^{n/2}i^{-\frac23 l-1}+\sum_{i>n/2}\frac{1}{n^{\frac23 (l+1)}(n+1-i)^{1/3}}\right)= O(n^{\frac23 l-1}\cdot n^{-4l/15+\delta}),
		\end{split}\eeq
		which is $o(\bE\left[\mathbbm{1}_{\event}|S_{21}|\right])$, provided $\delta<4l/15$.
		This completes our proof of Lemma \ref{lem:diff_stieltjes_gK}.

\section{Appendix: Section 5 proofs}\label{sec:appendix2}
In this section, we provide our proofs of Lemmas \ref{lem:Fi} and \ref{lem:ratio}. The proof of Lemma \ref{lem:Fi} requires asymptotic bounds on $|\rho_j^\pm|$ when $\gamma=d_+$ and a few related quantities, which we state in the following two lemmas. Similar results were developed for the case $\g>d_+$ in Lemmas 2.7 and 2.8 in \cite{CWL_CLT}.

\begin{lemma}\label{lem:rho}
	The following asymptotic bounds hold, uniformly in $i\geq 2$ (where $i$ can be fixed or $n$-dependent):
	\begin{enumerate}[(i)]
		\item $|\rho_i^+|=\Theta(n)$, $|\rho_i^-|=O(n)$,
		\item $|\rho_i^+|-|\rho_i^-|=\Theta(n^{1/2}(n-i+1)^{1/2})$,
		\item $|\rho_i^-|-|\rho_{i-1}^-|=O((\frac{n}{n-i+1})^{1/2})$ and $|\rho_{i-1}^+|-|\rho_i^+|=O((\frac{n}{n-i+1})^{1/2})$.
	\end{enumerate}
\end{lemma}
\begin{proof}
	To show (i) for $|\rho_i^-|$, observe that $|\rho_i^-|$ is increasing in $i$, and
	\begin{align*}
		|\rho_n^-|=\frac12\left(2\sqrt{mn}+1-2\sqrt{\sqrt{mn}+n+m-\tfrac34}+O(1)\right)=O(n).
	\end{align*}
	Similarly, part (i) for $|\rho_i^+|$ holds since $|\rho_i^+|$ is decreasing in $i$, $|\rho_2^+|<2\sqrt{mn}+2n=\Theta(n)$, and
	\beqq
	|\rho_n^+|>\frac12\left(d_+m-(m+n-1)\right)=\frac12\left(2\sqrt{mn}+1\right)=\Theta(n).
	\eeqq
For part (ii) we have
	\begin{align*}
	|\rho_i^+|-|\rho_i^-|&=\sqrt{(2\sqrt{mn}-1+2(n-i+1))^2-4(m-(n-i+1))(n-(n-i+1))}\\
	&=2\sqrt{\sqrt{mn}+1+(m+n+2\sqrt{mn}-1)(n-i+1)}=\Theta(n^{1/2}(n-i+1)^{1/2}).
	\end{align*}
Next, we verify (iii) by showing that $|\rho_i^-|-|\rho_{i-1}^-|+ |\rho_{i-1}^+|-|\rho_i^+|=O((\frac{n}{n-i+1})^{1/2})$. Indeed, the left hand side can be written as 
	\begin{align*}
		\left(|\rho_{i-1}^+|-|\rho_{i-1}^-|\right)- \left(|\rho_i^+|-|\rho_i^-|\right)=\frac{\left(|\rho_{i-1}^+|-|\rho_{i-1}^-|\right)^2-\left(|\rho_i^+|-|\rho_i^-|\right)^2}{|\rho_{i-1}^+|-|\rho_{i-1}^-|+ |\rho_i^+|-|\rho_i^-|},
	\end{align*}
	where numerator of the last ratio simplifies to $4d_+m-4=\Theta(n)$, and the denominator is $\Theta(n^{1/2}(n-i+1)^{1/2})$ by part (ii).  
	\end{proof}

\begin{lemma}\label{lem:1-wi}
	There exist constants $0<C_1<C_2$ such that, for sufficiently large $n$, and $2\leq i\leq n$,
	\beqq
	C_1\left(\frac{n-i+1}{n}\right)^{1/2}\leq 1-\w_i\leq C_2\left(\frac{n-i+1}{n}\right)^{1/2}.
	\eeqq
\end{lemma}
\begin{proof}
We recall that $\w_i=\frac{|\rho_i^-|}{|\rho_{i-1}^+|}$.  Using the bounds $\frac{|\rho_{i-1}^-|}{|\rho_{i-1}^+|}<\w_i<\frac{|\rho_i^-|}{|\rho_i^+|}$ we obtain
\beqq
\frac{|\rho_i^+|-|\rho_i^-|}{|\rho_i^+|}<1-\w_i<\frac{|\rho_{i-1}^+|-|\rho_{i-1}^-|}{|\rho_{i-1}^+|}.
\eeqq
Using Lemma \ref{lem:rho}, the left and right sides of this inequality are both $\Theta((\frac{n-i+1}{n})^{1/2})$, uniformly in $i$, which gives the desired bounds.
\end{proof}

\subsection{Proof of Lemma \ref{lem:Fi}}
Using $\frac{1}{1+F_{i-1}}=1-\frac{F_{i-1}}{1+F_{i-1}}=1-F_{i-1}+\frac{F_{i-1}^2}{1+F_{i-1}}$ and the notations in \eqref{eq:alpha_beta_def} and \eqref{eq:tau_delta_def}, we have
$F_1=\frac{\mu_1-d_+}{|\rho_1^+|/m}-\a_1$,
and for $j=2,\dots, n-1$,
\begin{equation*}
	F_j = -1+\frac{\mu_1}{|\rho_j^+|/m}-\left(\a_j+\b_j+\tau_j+\d_j\right)-(\a_{j-1}+\tau_{j-1})(\b_j+\d_j)\left(1-F_{j-1}+\frac{F_{j-1}^2}{1+F_{j-1}}\right).
\end{equation*}
As $1+\tau_j+\d_j=\frac{d_+ m}{|\rho_j^+|}-\frac{|\rho_j^-|}{|\rho_j^+|}$, we re-arrange the terms to have 
\begin{equation}
	F_j = \eta_j- \xi_j +\w_j F_{j-1}+\phi_j, 
\end{equation}
where we define
\begin{align}
	\eta_j&=\frac{\mu_1-d_+}{|\rho_j^+|/m}
	,\\
	\phi_j &=-\w_j+\frac{|\rho_j^-|}{|\rho_j^+|}-\a_{j-1}\b_j +(\a_{j-1}\b_j+\a_{j-1}\d_j+\tau_{j-1}\b_j)\frac{F_{j-1}}{1+F_{j-1}}-\w_j\frac{F_{j-1}^2}{1+F_{j-1}}\label{eqn:vare_i},
\end{align}
and $\xi_j$ is given in \eqref{eq:xi_def}. Note that, by Lemma \ref{lem:rho},
\begin{equation}\label{eqn:wi_rho}
	0<\w_j-\frac{|\rho_j^-|}{|\rho_j^+|}= \frac{|\rho_j^-|}{|\rho_j^+|}\frac{|\rho_{j-1}^+|-|\rho_j^+|}{|\rho_{j-1}^+|}= O(n^{-\frac12}(n-j+1)^{-\frac12}).
\end{equation}
Expanding the recurrence iteratively, we get 
\begin{equation}\label{eqn:Fi}
	\begin{split}
		F_j
		=\w_j\dots\w_2F_1&+\left(\eta_j+\w_j\eta_{j-1}+\dots+\w_j\dots\w_3\eta_2\right)\\
		&-\left(\xi_j+\w_j\xi_{j-1}+\dots+\w_j\dots\w_3\xi_2 \right)\\
		&+\left(\phi_j+\w_j\phi_{j-1}+\dots+\w_j\dots\w_3\phi_2\right).
	\end{split}
\end{equation}

On the event $\cF^{(3)}_{s,t}$, which holds with probability $1-\e/6$ for some $s,t$ depending on $\e$, $|\mu_1-d_+|\leq tn^{-\frac23}$. As $|\rho_i^+|/m=\Theta(1)$ for all $i\leq n$, we obtain 
\[
\max_{j\leq n }|\eta_j|=O\left(n^{-\frac23}\right).
\]
We recall that $\a_j,\b_j$ are the centered and scaled version of $\chi$-squared random variables $a_j^2, b_{j-1}^2$, respectively and as such, they can be bounded using concentration of sub-gamma random variables (see, e.g. Theorem 2.3 of \cite{BLM_concentration}).  In particular, there exists some constant $c$ such that, for all $j\leq n$ and for all $t>0$, 
\beq
\PP(|\a_j|>c(\sqrt{\tfrac{t}{n}}+\tfrac tn))\leq2e^{-t}
\eeq
and likewise for each $\b_j$, so we conclude that, for any $\e$, with probability at least $1-\e/6$,
\begin{equation}\label{eqn:max_alpha_beta}
	\max\{|\a_j|,|\b_j|: j\leq n\}\leq cn^{-\frac12}\sqrt{\log n}.
\end{equation}
Thus, for some constant $C_1>0$, with probability $1-\e/3$,
\begin{equation}\label{eqn:prod_F1}
	|\w_j\dots\w_2F_1|\leq |F_1|=|\eta_1-\a_1|\leq C_1n^{-\frac12}.
\end{equation}
As $\w_j$ is increasing in $j$, 
\beqq
1+\w_j+\w_j\w_{j-1}+\dots+\w_j\dots\w_3\leq 1+\w_j+\w_j^2+\dots=\frac1{1-\w_j}.
\eeqq
By Lemma \ref{lem:1-wi}, $1-\w_j=\Theta\left(\left(\frac{n-j}{n}\right)^{\frac12}\right)$.  Thus, setting $j_0:=\lfloor n-n^{\frac13}(\log n)^3 \rfloor$, we observe that, for some constant $C_2$, with probability $1-\e/3$,
\begin{equation}\label{eqn:max_sum_eta}
	\max_{j\leq j_0}|\eta_j+\w_j\eta_{j-1}+\dots+\w_j\dots\w_3\eta_2|\leq \max_{j\leq j_0}\left(|\eta_j|\frac1{1-\w_j}\right)\leq C_2n^{-\frac13}(\log n)^{-\frac32}.
\end{equation}
Having bounded the first line of \eqref{eqn:Fi}, we turn to the second line and recall the definition of $L_j$ in \eqref{eq:Li_def}. We have
\[
\xi_j+\w_j\xi_{j-1}+\dots+\w_j\dots\w_3\xi_2 = L_j+\w_j\dots\w_3\xi_2.
\]
Note that $\max_{j\leq n }|\xi_j|=O(n^{-\frac12}\sqrt{\log n})$ on the event \eqref{eqn:max_alpha_beta}. We also have, for some constant $C_3>0$, with probability $1-O(n^{-1})$, 
\beq\label{eqn:max_Li}
\max_{j\leq j_0 }|L_j|=O(n^{-\frac13}(\log n)^{-\frac14}).
\eeq 
The details for this bound can be obtained using a similar argument to the one found in Section 6.2 of \cite{CWL_CLT}.  
In particular, the bound \eqref{eqn:max_Li} follows from line (6.17) of that paper (where the notations $\alpha$ and $Y_i$ can be translated as $\alpha=2$ and $Y_i=L_i+O(n^{-\frac12})$ in our context).
Thus, for some constant $C_3>0$, with probability $1-\e/3$,
\begin{equation}\label{eqn:sum_xi}
	\max_{j\leq j_0}|\xi_j+\w_j\xi_{j-1}+\dots+\w_j\dots\w_3\xi_2|\leq C_3n^{-\frac13}(\log n)^{-\frac14}.
\end{equation}
Consider the event
\begin{equation}
	\cG :=\{\eqref{eqn:prod_F1}, \eqref{eqn:max_sum_eta}, \text{ and }  \eqref{eqn:sum_xi} \text{ hold} \},
\end{equation}
which holds with probability $1-\e$, for sufficiently large $n$. We now show that on this event, the third line of \eqref{eqn:Fi} is $o(n^{-\frac13})$. Since this quantity depends on $F_l$'s up to $F_{j-1}$, we can control it in the process of using induction to show 
\begin{equation}\label{eqn:Fi_indc}
	\max_{j\leq j_0}|F_j| = o(n^{-\frac13}) \quad \text{on the event } \cG.
\end{equation}
More specifically, we will show that $\max_{j\leq j_0}|F_j| < 2C_3n^{-\frac13}(\log n)^{-\frac14}$ where $C_3$ is the constant from \eqref{eqn:sum_xi}. The base case holds by \eqref{eqn:prod_F1}. Assume $\max_{l\leq j-1}|F_{l}| < 2C_3n^{-\frac13}(\log n)^{-\frac14}$. 
Then, by \eqref{eqn:vare_i}, \eqref{eqn:wi_rho} and \eqref{eqn:max_alpha_beta}, 
\beqq
\max_{l\leq j}|\phi_l|=o(n^{-\frac23}).
\eeqq
Note that the above maximum also includes $\phi_j$. Thus, for some constant $C_4>0$, 
\beqq
|\phi_j+\w_j\phi_{j-1}+\dots+\w_j\dots\w_3\phi_2|\leq \max_{l\leq j}|\phi_l|\frac1{1-\w_j}\leq C_4n^{-\frac13}(\log n)^{-3/2}.
\eeqq
Finally, by \eqref{eqn:Fi}, we have that on $\cG$,
\beqq
|F_j|\leq C_1n^{-\frac23}(\log n) +C_2n^{-\frac13}(\log n)^{-\frac12}+C_3n^{-\frac13}(\log n)^{-\frac14}+C_4n^{-\frac13}(\log n)^{-3/2} <2C_3n^{-\frac13}(\log n)^{-\frac14}.
\eeqq

This completes the induction step, and we obtain the lemma.

\subsection{Proof of Lemma \ref{lem:ratio}}
Fix $\e>0$. For $j_0=\lfloor n-n^{1/3}(\log n)^3 \rfloor$ and $t=(e\log n)^2$, it suffices to show that, for sufficiently large $n$, each of the probabilities
\begin{align}
	p_1&:=\bP\left(\max_{j\leq j_0}a_jb_j<\sqrt{mn}-tn^{1/2}\right)\leq \bP\left(a_{j_0}b_{j_0}<\sqrt{mn}-tn^{1/2}\right) \text{ and }\\
	p_2&:=\bP\left(\max_{j\leq j_0}a_jb_j>\sqrt{mn}+tn^{1/2}\right)=1-\prod_{j=1}^{j_0}\bP\left(a_jb_j<\sqrt{mn}+tn^{1/2}\right)
\end{align}
is less than $\e/2$. For any $j=1,2,\dots, j_0$, observe that
\beq
a_j^2b_j^2\stackrel{(d)}{=}\left(\sum_{i=1}^{m-n+j}g_i^2\right)\left(\sum_{k=1}^{j}(g'_k)^2\right),
\eeq
where $\stackrel{(d)}{=}$ denotes equality in distribution, and  $g_1,\dots, g_{m-n+j}, g'_1, \dots, g'_j$ are independent standard gaussian variables. This implies that $\bE a_j^2b_j^2=j(m-n+j)$ and $\Var(a_j^2b_j^2)=2j(m-n+j)(m-n+2j+2)$. Viewing $a_j^2b_j^2$ as a gaussian polynomial of degree 4 in $m-n+2j$ variables $g_i$'s and $g'_k$'s, we have the following concentration result from \cite{AS17} (see Corollary 5.49): For any $s\geq (2e)^2$, 
\beq\label{eqn:poly_conc}
\bP\left(|a_j^2b_j^2-j(m-n+j)|\geq s\sqrt{2j(m-n+j)(m-n+2j+2)}\right)\leq \exp\left(-2\sqrt{s}/e\right).
\eeq

Apply this result to $(a_{j_0}b_{j_0})^2$ with $s=(e\log n)^2$, we obtain $p_1\leq n^{-2}$. At the same time, \eqref{eqn:poly_conc} implies $\bP\left(a_jb_j<\sqrt{mn}+tn^{1/2}\right)\geq 1-n^{-2}$ for all $1\leq j\leq j_0$, which yields $p_2\leq 1-e^{-c/n}$ for some $c>0$. This completes the proof of the lemma.
\end{appendix}

\pagebreak


\end{document}